\pdfoutput=1
\documentclass{amsart} % reqno places equation numbers on the right
\newif\ifpersonal
\usepackage{amsmath,amscd,amsthm,amssymb,mathrsfs,mathtools,bm} % math related
\usepackage{microtype,lmodern,textcomp} % latex technical issues
\usepackage{enumitem,comment,braket,xspace,tikz-cd} % utilities
\usepackage[utf8]{inputenc} % input encoding
\usepackage[T1]{fontenc} % font encoding 4518376
\usepackage[centering,vscale=0.7,hscale=0.7]{geometry}
\usepackage[hidelinks]{hyperref}
\usepackage[capitalize]{cleveref}

\numberwithin{equation}{section}
\theoremstyle{plain}
\newtheorem{theorem}[equation]{Theorem}
\newtheorem*{theorem*}{Theorem}
\newtheorem{lemma}[equation]{Lemma}
\newtheorem*{lemma*}{Lemma}

\newtheorem*{claim*}{Claim}

\newtheorem{proposition}[equation]{Proposition}
\newtheorem*{proposition*}{Proposition}
\newtheorem{corollary}[equation]{Corollary}
\newtheorem*{corollary*}{Corollary}
\theoremstyle{definition}
\newtheorem{definition}[equation]{Definition}
\newtheorem*{definition*}{Definition}
\newtheorem{definition-theorem}[equation]{Definition-Theorem}
\newtheorem{definition-lemma}[equation]{Definition-Lemma}

\newtheorem{assumption}[equation]{Assumption}
\newtheorem{notation}[equation]{Notation}
\newtheorem{example}[equation]{Example}
\newtheorem{remark}[equation]{Remark}
\newtheorem*{remark*}{Remark}

\numberwithin{equation}{section}

\ifpersonal
\newcommand{\personal}[1]{\textcolor[rgb]{0,0,1}{(Personal: #1)}}

\newcommand{\todo}[1]{\textcolor{red}{(Todo: #1)}}
\else
\newcommand{\personal}[1]{\ignorespaces}
\newcommand{\discussion}[1]{\ignorespaces}
\newcommand{\todo}[1]{\ignorespaces}
\fi

\providecommand{\abs}[1]{\lvert#1\rvert}

\newcommand{\bbC}{\mathbb C}
\newcommand{\bbD}{\mathbb D}

\newcommand{\bbL}{\mathbb L}

\newcommand{\bbN}{\mathbb N}

\newcommand{\bbP}{\mathbb P}
\newcommand{\bbQ}{\mathbb Q}
\newcommand{\bbR}{\mathbb R}

\newcommand{\bbZ}{\mathbb Z}

\newcommand{\cD}{\mathcal D}
\newcommand{\cE}{\mathcal E}
\newcommand{\cF}{\mathcal F}
\newcommand{\cH}{\mathcal H}

\newcommand{\cM}{\mathcal M}

\newcommand{\cO}{\mathcal O}
\newcommand{\cP}{\mathcal P}

\newcommand{\bA}{\mathbf A}

\newcommand{\bD}{\mathbf D}

\newcommand{\bG}{\mathbf G}
\newcommand{\bH}{\mathbf H}
\newcommand{\bN}{\mathbf N}

\newcommand{\bc}{\mathbf c}
\newcommand{\bt}{\mathbf t}
\newcommand{\bPhi}{\mathbf{\Phi}}

\makeatletter
\let\save@mathaccent\mathaccent
\newcommand*\if@single[3]{%
	\setbox0\hbox{${\mathaccent"0362{#1}}^H$}%
	\setbox2\hbox{${\mathaccent"0362{\kern0pt#1}}^H$}%
	\ifdim\ht0=\ht2 #3\else #2\fi
}
\newcommand*\rel@kern[1]{\kern#1\dimexpr\macc@kerna}
\newcommand*\widebar[1]{\@ifnextchar^{{\wide@bar{#1}{0}}}{\wide@bar{#1}{1}}}
\newcommand*\wide@bar[2]{\if@single{#1}{\wide@bar@{#1}{#2}{1}}{\wide@bar@{#1}{#2}{2}}}
\newcommand*\wide@bar@[3]{%
	\begingroup
	\def\mathaccent##1##2{%
		\let\mathaccent\save@mathaccent
		\if#32 \let\macc@nucleus\first@char \fi
		\setbox\z@\hbox{$\macc@style{\macc@nucleus}_{}$}%
		\setbox\tw@\hbox{$\macc@style{\macc@nucleus}{}_{}$}%
		\dimen@\wd\tw@
		\advance\dimen@-\wd\z@
		\divide\dimen@ 3
		\@tempdima\wd\tw@
		\advance\@tempdima-\scriptspace
		\divide\@tempdima 10
		\advance\dimen@-\@tempdima
		\ifdim\dimen@>\z@ \dimen@0pt\fi
		\rel@kern{0.6}\kern-\dimen@
		\if#31
		\overline{\rel@kern{-0.6}\kern\dimen@\macc@nucleus\rel@kern{0.4}\kern\dimen@}%
		\advance\dimen@0.4\dimexpr\macc@kerna
		\let\final@kern#2%
		\ifdim\dimen@<\z@ \let\final@kern1\fi
		\if\final@kern1 \kern-\dimen@\fi
		\else
		\overline{\rel@kern{-0.6}\kern\dimen@#1}%
		\fi
	}%
	\macc@depth\@ne
	\let\math@bgroup\@empty \let\math@egroup\macc@set@skewchar
	\mathsurround\z@ \frozen@everymath{\mathgroup\macc@group\relax}%
	\macc@set@skewchar\relax
	\let\mathaccentV\macc@nested@a
	\if#31
	\macc@nested@a\relax111{#1}%
	\else
	\def\gobble@till@marker##1\endmarker{}%
	\futurelet\first@char\gobble@till@marker#1\endmarker
	\ifcat\noexpand\first@char A\else
	\def\first@char{}%
	\fi
	\macc@nested@a\relax111{\first@char}%
	\fi
	\endgroup
}
\makeatother

\newcommand{\tA}{\widetilde A}

\newcommand{\tH}{\widetilde H}

\newcommand{\tP}{\widetilde P}
\newcommand{\tQ}{\widetilde Q}
\newcommand{\tR}{\widetilde R}

\newcommand{\tT}{\widetilde T}
\newcommand{\tU}{\widetilde U}
\newcommand{\tV}{\widetilde V}

\newcommand{\tX}{\widetilde X}

\newcommand{\tPhi}{\widetilde\Phi}

\newcommand{\dbb}[1]{[\![#1]\!]}

\newcommand{\iunit}{\mathrm{i}}

\newcommand{\odd}{\mathrm{odd}}
\newcommand{\even}{\mathrm{even}}

\newcommand{\bbk}{\Bbbk}
\newcommand\fm{\mathfrak{m}}

\newcommand{\bK}{\mathbf K}
\newcommand{\fr}{\mathrm{fr}}

\newcommand{\id}{\mathrm{id}}
\newcommand{\Id}{\mathrm{Id}}

\newcommand{\iso}{\mathit{iso}}
\newcommand{\spl}{\mathrm{split}}

\newcommand{\bKlim}{\bK_\mathrm{lim}}
\newcommand{\vir}{\mathrm{vir}}
\newcommand{\ev}{\mathrm{ev}}
\newcommand{\pr}{\mathrm{pr}}
\newcommand{\longto}{\longrightarrow}

\newcommand{\blambda}{\boldsymbol{\lambda}}

\newcommand{\tphi}{\tilde{\phi}}
\newcommand{\bd}{\mathbf{d}}

\DeclareMathOperator{\im}{im}
\DeclareMathOperator{\rk}{rk}
\DeclareMathOperator{\codim}{codim}
\DeclareMathOperator{\End}{End}
\DeclareMathOperator{\Hom}{Hom}
\DeclareMathOperator{\Aut}{Aut}
\DeclareMathOperator{\Mat}{Mat}
\DeclareMathOperator{\NE}{NE}
\DeclareMathOperator{\GL}{GL}

\DeclareMathOperator{\Eu}{Eu}
\DeclareMathOperator{\Spf}{Spf}

\DeclareMathOperator{\Spec}{Spec}

\newcommand{\1}{\mathbf{1}}

\DeclareMathOperator{\Frac}{Frac} % Fraction field. 
\DeclareMathOperator{\Ext}{Ext}
\DeclareMathOperator{\ad}{ad} % Adjoint endomorphism

\DeclareMathOperator{\Diag}{Diag} % Diagonal matrix

\DeclareMathOperator{\Tor}{Tor} % Tor functor

\DeclareMathOperator{\Sp}{Sp} % Spectrum of an affinoid algebra

\DeclareMathOperator{\an}{an} % Anaytification
\newcommand{\tvarphi}{\widetilde{\varphi}}

\renewenvironment{abstract}{%
  \quotation
  \small
  \textbf{\textit{\abstractname.}} % with a normal space
}{\endquotation}

\begin{document}
\title{Decomposition and framing of F-bundles and applications to quantum cohomology}

\author{Thorgal Hinault}
\address{Thorgal Hinault, Department of Mathematics, M/C 253-37, Caltech, 1200 E.\ California Blvd., Pasadena, CA 91125, USA}
\email{thinault@caltech.edu}
\author{Tony Yue YU}
\address{Tony Yue YU, Department of Mathematics, M/C 253-37, Caltech, 1200 E.\ California Blvd., Pasadena, CA 91125, USA}
\email{yuyuetony@gmail.com}
\author{Chi Zhang}
\address{Chi Zhang, Department of Mathematics, M/C 253-37, Caltech, 1200 E.\ California Blvd., Pasadena, CA 91125, USA}
\email{czhang5@caltech.edu}
\author{Shaowu Zhang}
\address{Shaowu Zhang, Department of Mathematics, M/C 253-37, Caltech, 1200 E.\ California Blvd., Pasadena, CA 91125, USA}
\email{szhang7@caltech.edu}
\date{February 7, 2025}


\maketitle

\begin{abstract}
  F-bundle is a formal/non-archimedean version of variation of nc-Hodge structures which plays a crucial role in the theory of atoms as birational invariants from Gromov-Witten theory.
  In this paper, we establish the spectral decomposition theorem for F-bundles according to the generalized eigenspaces of the Euler vector field action.
  The proof relies on solving systems of partial differential equations recursively in terms of power series, and on estimating the size of the coefficients for non-archimedean convergence.
  The same technique allows us to establish the existence and uniqueness of the extension of framing for logarithmic F-bundles.
  As an application, we prove the uniqueness of the decomposition map for the A-model F-bundle (hence quantum D-module and quantum cohomology) associated to a projective bundle, as well as to a blowup of an algebraic variety.
  This complements the existence results by Iritani-Koto and Iritani.
\end{abstract}

\tableofcontents

\addtocontents{toc}{\protect\setcounter{tocdepth}{1}}
\section{Introduction}\label{sec:introduction}

\subsection{Motivations}
Let $X$ be a smooth projective complex variety.
The enumeration of curves in $X$ is a classical subject in algebraic geometry, enjoying a variety of approaches (see \cite{Pandharipande_13/2_ways}).
Gromov-Witten theory is one of the most widely known and the most general (e.g.\ no restriction on the dimension of $X$), (see \cite{Kontsevich_Gromov-Witten_classes,Li_Virtual_moduli,Behrend_Gromov-Witten_invariants}).
The Gromov-Witten invariants of $X$ are rational numbers depending on the genus $g$, number of marked points $n$ and cohomology classes $\phi_1, \dots, \phi_n$ of $X$.
They satisfy a notable relation called the WDVV equation, which allows them to be packaged into differential geometric data, such as Frobenius manifolds by Dubrovin (\cite{Dubrovin_Geometry_of_2D}) or semi-infinite Hodge structures by Barannikov (\cite{Baranikov_semi_infinite_hodge_structutres}).
The differential geometric framework facilitates intuitions from geometry and mirror symmetry and contributes tremendously to the development of the subject.
The framework was further extended to incorporate the integral/rational structure via the notion of nc-Hodge structure by Katzarkov-Kontsevich-Pantev \cite{Katzarkov_Hodge_theoretic_aspects}.
They established a profound gluing/decomposition theorem using the Fourier-Laplace transform of the associated D-modules (see \S 2.4.2 in loc.\ cit.).
This motivated the development of the theory of atoms for applications to birational geometry (see \cite{KKPY_Birational_invariants}).
The idea is to apply the decomposition to the A-model nc-Hodge structure (defined using Gromov-Witten invariants) associated to a smooth projective variety at a generic point of the base, and view the resulting pieces as elementary pieces of the variety called atoms.
The collection of atoms (modulo an equivalence relation induced by blowups) is expected to serve as a powerful birational invariant.

While the notions of nc-Hodge structure and atom are natural and beautiful, it is still conjectural that Gromov-Witten invariants actually give rise to an nc-Hodge structure satisfying all the axioms in \cite[\S 2.1.5]{Katzarkov_Hodge_theoretic_aspects}.
The difficulties include the convergence of the Gromov-Witten potential (\cite{Iritani_convergence_of_quantum_cohomology_by_quantum_Lefschetz}), the Gamma conjecture (\cite{Galkin_Golyshev_Iritani_gamma_classes_and_quantum_cohomology_of_Fano_manifolds}) and the opposedness axiom (\cite{Reichelt_non_affine_LG_model}).
This means that the theory of nc-Hodge structure cannot yet be unconditionally employed for the study of Gromov-Witten invariants and their applications in general.
In this paper, we consider a formal/non-archimedean distilled version of variation of nc-Hodge structures, which we call F-bundles (see \cref{sec:F-bundle}, and see \cref{subsec:related_works} for related notions).
We establish the spectral decomposition theorem for F-bundles, according to the generalized eigenspaces of the Euler vector field action, motivated by the gluing theorem for nc-Hodge structures via Fourier-Laplace transform, see \cref{sec:decomposition-max-F-bundles}.
The comparison of the F-bundle decomposition and the nc-Hodge structure decompositions is studied in \cite[\S 8]{Yu_Zhang_topological_Laplace_transform_and_decomposition_of_nc_hodge_structures}.

Furthermore, we study the notion of framing of F-bundles, analogous to the decoration on variations of nc-Hodge structures, and prove the existence and uniqueness of the extension of framing, see \cref{sec:framing-F-bundle}.
This allows us to identify F-bundles via maps on the base (analogous to a mirror map) together with a gauge transformation on the bundle.
As an application, we prove the uniqueness of the decomposition map for the A-model F-bundle (hence quantum D-module and quantum cohomology) associated to a projective bundle, as well as to a blowup of an algebraic variety.
This complements the existence result by Iritani-Koto \cite{Iritani_Quantum-cohomology-projective-bundle} and Iritani \cite{Iritani_Quantum-cohomology-of-blowups}.

\subsection{Main results}

Below we give a more detailed description of our results.

Throughout the paper, we fix a field $\bbk$ of characteristic $0$.
In the non-archimedean setting, we assume that $\bbk$ has a nontrivial valuation whose restriction to $\bbQ$ is trivial.
Let $B$ be a smooth $\bbk$-analytic space, and $\bbD_u$ the germ at $0$ in a $\bbk$-analytic closed unit disk with coordinate $u$.

An \emph{F-bundle} $(\cH, \nabla)$ over $B$ consists of a vector bundle $\cH$ over $B\times \bbD_u$ and a meromorphic flat connection $\nabla$ with poles at $u=0$, such that $\nabla_{u^2\partial_u}$ and $\nabla_{u\xi}$ are regular for any tangent vector field $\xi$ on $B$.
We refer to \cref{def:log-F-bundle} for the definition of logarithmic F-bundle in the formal case.

For any $b\in B$, we have a natural action
\begin{align*}
\mu_b \colon T_b B &\longto \End(\cH_{b,0}) \\
\xi &\longmapsto \nabla_{u\xi}|_{\cH_{b,0}} .
\end{align*}
The F-bundle is called \emph{maximal} at $b$ if the action induces an isomorphism between $T_b B$ and $\cH_{b,0}$ via a cyclic vector, see \cref{definition:maximal-F-bundle}.
This gives a commutative product on $T_b B$ by the flatness of $\nabla$.

\subsubsection{Spectral decomposition theorems}

Let $K_b \coloneqq \nabla_{u^2\partial_u}\vert_{b,0}$, it is the action of the Euler vector field on the fiber $\cH_{b,0}$.
We show that the generalized eigenspace decomposition of $K_b$ extends locally to a product decomposition of the F-bundle.
Here are the precise statements.

\begin{theorem}[Formal spectral decomposition, \cref{thm:eigenvalue_decomposition}] 
  Let $B$ be a formal neighborhood of a rational point $b$ in a smooth $\bbk$-variety, and $(\cH, \nabla)$ an F-bundle over $B$ maximal at $b$.
  Assume that we have a decomposition $\cH_{b,0} \simeq \bigoplus_{i \in I} H_i$ stable under $K_b$, and that for any $i \neq j \in I$, the spectra of $K_b|_{H_i}$ and $K_b|_{H_j}$ are disjoint.
  Then $(\cH ,\nabla )/B$ decomposes into a product of maximal F-bundles $(\cH_i ,\nabla_i )/B_i$ extending the decomposition of $\cH|_{b,0}$.
\end{theorem}

\begin{theorem}[Non-archimedean spectral decomposition, \cref{thm:NA-K-decomposition}] 
   Let $B$ be an admissible open neighborhood of a rational point $b$ in a smooth $\bbk$-analytic space, and $(\cH ,\nabla )$ an F-bundle over $B$ maximal at $b$.
   Assume that we have a decomposition $\cH_{b,0} \simeq \bigoplus_{i\in I} H_i$ stable under $K_b$, and that for any $i\neq j\in I$, the spectra of $K_b\vert_{H_i}$ and $K_b\vert_{H_j}$ are disjoint.
   Then there exists an admissible open neighborhood $U$ of $b$ such that $(\cH\vert_U ,\nabla\vert_U )/U$ decomposes into a product of maximal F-bundles $(\cH_i ,\nabla_i ) /U_i$ extending the decomposition of $\cH_{b,0}$.
\end{theorem}

For proving the spectral decomposition, first we establish a formal and a non-archimedean version of the Frobenius theorem (see \cref{theorem:Frobenius-theorem,thm:NA-Frobenius-theorem}), by solving recursively a system of partial differential equations (see \cref{theorem:existence-solution-flat-PDE}).
By the maximality assumption, we obtain an F-manifold structure on the base $B$ of the F-bundle (see \cref{def:F-manifold} and \cref{lemma:max-F-bundle-F-manifold,lemma:NA-max-F-bundle-F-manifold}).
The eigenspaces of $K_b$ induce a decomposition of the tangent space $T_b B$ as a $\bbk$-algebra, and we show that this decomposition extends locally around $b$ (\cref{theorem:decomposition-F-manifolds,thm:NA-decomposition-F-manifold}).
To do so, we first prove that the algebra structure on the tangent spaces decomposes via deformations of $\bbk$-algebras (\cref{lemma:deformation-unital-associative-commutative-algebra,lemma:NA-deformation-tangent-space-splits}).
Then, using the F-identity \eqref{eq:F-identity} of the F-manifold, we prove that the induced decomposition of the tangent bundle is a decomposition into commuting integrable distributions (\cref{proposition:A-decomposition-theorem}).
Finally, using the formal and non-archimedean versions of the Frobenius theorem, we integrate those distributions and produce a decomposition of the F-manifold $B\simeq\prod_{i\in I} B_i$.

Having decomposed the base $B$, we use maximality again to obtain a splitting of $\cH\vert_{u=0}$.
The link between the connection $\nabla$ and the F-manifold structure implies that this decomposition is stable under the residue endomorphisms $\nabla_{u^2\partial_u}\vert_{u=0}$ and $\nabla_{u\xi}\vert_{u=0}$ for any $\xi\in TB$.
Using the disjoint spectra assumption, we extend this decomposition to a decomposition of $\cH$ stable under $\nabla_{u^2\partial_u}$ by a recursive procedure, and obtain the decomposition in the formal case, see \cref{proposition:split-u-direction}.
In the non-archimedean case, through a careful analysis of the recursion and the norms of the coefficients, we show that the decomposition is convergent over an admissible open neighborhood of $b$, see \cref{proposition:split-u-direction-NA}.
Finally, using flatness, we prove that the connection also decomposes according to the splitting of $\cH$ (\cref{proposition:split-t-directions}), and that each piece is the pullback of a maximal F-bundle on $B_i$ from the decomposition of the base $B$.

\subsubsection{Extension of framing}

A \emph{framing} for an F-bundle $(\cH ,\nabla ) / B$ is roughly a local trivialization of $\cH$ in which the connection involves no positive powers of $u$ (\cref{def:framing}).
It is analogous to the notion of decoration on variations of nc-Hodge structures in \cite[\S 4.1.3]{Katzarkov_Hodge_theoretic_aspects}.
Framings do not exist in general.
We prove in the following that if a framing exists at a point $b\in B$ and is strong in the logarithmic case, then it extends uniquely and explicitly to a formal or non-archimedean analytic neighborhood.

\begin{theorem}[\cref{theorem:extension-of-framing-connection-version}]
    Let $(\cH,\nabla )/ (B,D)$ be a logarithmic F-bundle, where $B$ is a formal neighborhood of a rational point $b$ in a smooth $\bbk$-variety.
    A framing at $b$ extends to a framing over $B$ if and only if it is strong with respect to $D$ (see \cref{definition:strong-framing}).
    In this case, the extension is unique and explicitly determined.
\end{theorem}

\begin{theorem}[\cref{theorem:extension-faming-NA-F-bundle}]
  Let $B$ be an admissible open neighborhood of a rational point $b$ in a smooth $\bbk$-analytic space. 
  Let $(\cH ,\nabla )$ be a non-archimedean F-bundle over $B$. 
  Then every framing at $b$ extends uniquely and explicitly to a framing over an admissible open neighborhood $U$ of $b$ in $B$.
\end{theorem}

The proofs are carried out by reformulating the problem into a system of partial differential equations (\eqref{eq:framing-u-direction}-\eqref{eq:extension-framing-log-fixing-residues}), which is then solved inductively on the number of variables.
If there are no logarithmic directions and $(t_1,\dots, t_n)$ are coordinates on $B$ centered at $b$, we first solve \eqref{eq:framing-t-direction} in the $t_1$-direction at $t_2=\cdots = t_n=0$ order by order in $t_1$, by observing that the equation provides a recursive relation.
We use this solution as an initial condition, and then solve \eqref{eq:framing-t-direction} in the $t_2$-direction at $t_3=\cdots = t_n=0$.
Using flatness of the connection, we prove that the solution obtained solves the equation in the $t_1$-direction as well, for all $t_2$.
In this way, we solve \eqref{eq:framing-t-direction} for all directions, and we show that the solution also solves \eqref{eq:framing-u-direction} using flatness again.
In the non-archimedean case, we prove that the solution converges by bounding its coefficients using \eqref{eq:framing-t-direction}.

The extension in the formal setting works also for logarithmic F-bundles, under the assumption that the framing at $b$ is strong with respect to $D$ (see \cref{definition:strong-framing}).
This condition implies that the residues $\mu_b (v)$ at $b$ along $u=0$ have nilpotent adjoint endomorphism for $v\in T_bD$, a property we call the nilpotency condition (see \cref{definition:nilpotency-condition}).
This nilpotency condition allows us to extract a recursive relation from \eqref{eq:framing-q-direction}, so we can reconstruct a solution to the equation order by order.
We proceed as in the non-logarithmic case and solve the system of PDEs inductively on the number of variables, this time starting from the logarithmic directions.

Based on the extension of framing theorem, we give a reconstruction result for isomorphisms of logarithmic F-bundles with framing in \cref{subsec:comparison_of_framed_F-bundles}.
We can always reconstruct the bundle isomorphism from its restriction to a point and the framing.
In the maximal case, we can also reconstruct the map on the base from its restriction to a point, up to some multiplicative constants in the logarithmic directions.

\begin{proposition}[\cref{lemma:comparison-framed-F-bundles}]
For $i=1,2$, let $(\cH_i,\nabla_i)/(B_i,D_i)$ be a logarithmic F-bundle where $B_i$ is the formal neighborhood of a rational point in a smooth $\bbk$-variety.
Let $(f,\Phi) \colon (\cH_1,\nabla_1) /\allowbreak (B_1,\allowbreak D_1) \rightarrow (\cH_2,\nabla_2 )/(B_2,D_2)$ be an isomorphism between logarithmic F-bundles with $f(b_1)=b_2$.
Assume $(\cH_1,\nabla_1 )/(B_1,D_1)$ has a framing $\nabla_1^{\fr}$.
\begin{enumerate}[wide]
    \item The bundle map $\Phi$ is uniquely and explicitly determined by its restriction to $\cH_1\vert_{b_1\times\Spf \bbk\dbb{u}}$.
        \item If $(\cH_1 ,\nabla_1 )$ and $(\cH_2,\nabla_2)$ are maximal,
        then the base map $f$ is also uniquely determined by its restriction to $b_1$, up to some multiplicative constants in the logarithmic directions.
		The reconstruction is explicit after fixing compatible cyclic vectors at $b_1$ and $b_2$.
\end{enumerate}
\end{proposition}

Motivated by the extension of framing theorem and our applications in \cref{sec:app-projective-bundle}, we prove the following classification result for framed F-bundles over a point.

\begin{proposition}[\cref{corollary:classification-framed-F-bundle-point}]
Let $\cH\simeq H\times\bbk\dbb{u}$ be a trivialized rank $m$ vector bundle over $\bbk \dbb{u}$.
    Let $(\cH ,\nabla )$ and $(\cH ,\nabla')$ be two F-bundles framed in the given trivialization, and write
    \begin{align*}
        \nabla_{u\partial_u} &= u\partial_u + u^{-1} \bK + \bG, \\
        \nabla_{u\partial_u} ' &= u\partial_u + u^{-1} \bK' + \bG'.
    \end{align*}
    Assume $\bK$ has simple eigenvalues.
    Then $(\cH ,\nabla )$ is isomorphic to $(\cH,\nabla' )$ if and only if there exists $\phi\in \GL (H)$ such that 
        \begin{enumerate}
            \item $\bK = \phi^{-1}\circ \bK'\circ\phi$, and
            \item in an eigenbasis of $\bK$, we have $(\bG)_{ii} = (\phi^{-1}\circ\bG'\circ\phi )_{ii}$ for $1\leq i\leq m$.
        \end{enumerate}  
    Furthermore, the gauge equivalence is uniquely and explicitly determined by the initial condition $\phi$ at $u=0$.
\end{proposition}

Proceeding order by order in $u$, we reformulate the gauge equivalence of the two connections as a system of equations \eqref{eq:gauge-eq-F-bundle-point-initial-condition}-\eqref{eq:gauge-eq-F-bundle-point-recursion} involving the adjoint map $[\bK,\cdot ]$.
When the eigenvalues are not simple, the equations are hard to solve because the map $[\bK ,\cdot ]$ does not have an easy description.
We provide a partial classification in \cref{theorem:gauge-eq-F-bundle-point}, under the assumption that all the generalized eigenspaces of $\bK$ have the same dimension, and by restricting the type of coefficients we allow in the connections.
The assumption on the coefficients allows us to work relative to a universal algebra.
Relative to this algebra, the endomorphism $\bK$ has simple eigenvalues, and we are able to solve the system.

We illustrate an application of these results in the next paragraph.
The reconstruction of isomorphisms also has applications in the reconstruction of mirror maps in Hodge-theoretic mirror symmetry, which we plan to explore in a subsequent work.

\subsubsection{Application to the decomposition of quantum cohomology}

Let $V \to X$ be a rank $m$ vector bundle on a smooth complex projective variety $X$, $P\coloneqq \bbP(V)\xrightarrow{\pi} X$ the associated projective bundle, and $h\coloneqq c_1(\cO_P(1))$.
We have a natural splitting
\begin{equation}\label{eq: classical decomposition}
	    \iso\colon \bigoplus_{i=0}^{m-1} H^{\ast}(X,\bbQ)[-2i] \xrightarrow{\sum h^{i}\cup \pi^*} H^{\ast}(P,\bbQ).
\end{equation}

Fix an ample class $\omega_X\in H^2 (X,\bbZ)$, and a homogeneous basis $\lbrace T_j\rbrace_{0\leq j\leq N}$ of $H^{\ast} (X,\bbQ )$ extending $\lbrace \1 , \omega_X\rbrace$.
We obtain the A-model maximal F-bundle $(\cH,\nabla)$ for $P$ over a formal base $B$ with closed point $b$ given by $0\in H^{\ast} (X,\bbQ )$ (see \cref{example:modified-quantum-F-bundle}).
Let $X' \coloneqq \coprod_{0\leq i\leq m-1} X$, and $(\cH',\nabla')$ the A-model maximal F-bundle over a formal base $B'$ with closed point $b'$ given by $\Delta (a)\in H^{\ast} (X',\bbC )$.
We denote by $(a_{i,j} )$ the coordinates of $\Delta (a)$ in the basis of $H^{\ast}(X',\bbC )$ induced from $\lbrace T_{j}\rbrace$ using \eqref{eq: classical decomposition}.

Our first result shows the existence and uniqueness of a gauge equivalence over the base points.

\begin{theorem} [\cref{thm:uniqueness-projective-bundle-u=0}]
There exists an F-bundle isomorphism 
\[\Phi(u)\colon (\cH  ,\nabla )\vert_b \rightarrow (\cH' ,\nabla' )\vert_{b'},\]
whose components $\Phi_{ij}$ (as power series in $u$) are given by the cup-product with elements in $H^{\ast} (X,\bbC )$ if and only if the coordinates of the base point $\Delta (a)$ satisfy 
\begin{equation} \label{eq:coordinate-base point-intro}
\sum_{j\colon\deg T_j\neq 2} \frac{\deg T_j - 2}{2} a_{i,j} T_j = c_1V + m\lambda_i ,
\end{equation}
where $\lambda_i$ was defined in \cref{lemma:jordan-dec-Klim}.

Furthermore, in this case $\Phi$ is uniquely determined by the $H^0$-components of $\Phi_{ij}\vert_{u=0}$, and $\Delta (a)$ is uniquely determined by \eqref{eq:coordinate-base point-intro}, up to a shift in $\bigoplus_{i=1}^m H^2 (X,\bbC )$.
\end{theorem}

Next, we extend the uniqueness result over the bases $B$ and $B'$.
The existence is shown by Iritani-Koto \cite{Iritani_Quantum-cohomology-projective-bundle}.

\begin{theorem}[\cref{thm:uniqueness-projective-bundle}]
Let $(f, \Phi)\colon (\cH,\nabla)/B \to (\cH',\nabla')/B'$ be an isomorphism of F-bundles. Then 
\begin{enumerate}[wide]
    \item The bundle map $\Phi$ is uniquely and explicitly determined by its restriction to $b\in B$.
    \item The base map $f$ is uniquely and explicitly determined by its restriction to $b\in B$, up to a multiplicative constant in the $q$ direction.
\end{enumerate}
\end{theorem}

In \cref{thm:uniqueness-blowup-u=0,thm:uniqueness-blowup}, we state the analogous results in the case of blowups of smooth complex projective varieties.

\subsection{Related works} \label{subsec:related_works}

Various related but slightly different concepts of F-bundles have been studied in the literature.
We refer to \cite{Dubrovin_Geometry_of_2D, Manin_Frobenius_manifolds} for Frobenius manifolds,
\cite{Saito_Period-mapping-associated-to-a-primitive-form, Sabbah_Isomonodromic_deformations} for Saito structures, 
\cite{Hertling_tt-star-geometry-Frobenius-manifolds-connections-and-singularities,David_Hertling_Meromorphic_connecions_over_F_manifolds} for (TE)-structures and variations,
\cite{Baranikov_semi_infinite_hodge_structutres} for semi-infinite variations of Hodge structures,
\cite{Hertling_Manin_Weak_Frobenius_manifolds, Hertling_Frobenius-manifolds-moduli-space-singularities} for 
F-manifolds,
\cite{Katzarkov_Hodge_theoretic_aspects, KKP2} for nc-Hodge structures,
and \cite{Cecotti-Vafa-topological-anti-topological-fusion,Simpson_the_hodge_filtration_on_nonabelian, Takuro_kobayashi_hitchin_correspondence} for other related works. 
Logarithmic variants of Frobenius manifolds and (TE)-structures were introduced in \cite{Reichelt_Logarithmic-Frobenius-manifolds, Reichelt_Thomas_Sevenheck_Christian_logarithmic_frobenius_manifolds_hypergeometric_systems_and_quantum_d_modules}.

Works related to our decomposition theorems for F-bundles include \cite{Dubrovin_Geometry_of_2D} for the decomposition of semisimple Frobenius manifolds, \cite{Sabbah_Isomonodromic_deformations} for the decomposition of meromorphic connections, \cite{Hertling_Frobenius-manifolds-moduli-space-singularities} for the decomposition of F-manifolds, and \cite{Yu_Zhang_topological_Laplace_transform_and_decomposition_of_nc_hodge_structures} for the comparison of the spectral decomposition and the vanishing cycle decomposition of nc-Hodge structures.
Analogs of our extension of framing theorem were studied in \cite{Iritani_quantum_D_modules_and_equivariant_Floer_theory, Coates_Hodge-theoretic_mirror_symmetry_for_toric_stacks} for the $q$-direction, in \cite{David_Hertling_Meromorphic_connecions_over_F_manifolds} for the $t$-direction, and in \cite{Iritani_quantum_D_modules_and_generalied_mirror_transformations} under different assumptions.
We refer to \cite{Iritani_Quantum-cohomology-of-blowups, Iritani_Quantum-cohomology-projective-bundle} for the decomposition of quantum D-modules for projective bundles and blowups.

\subsection{Acknowledgements}

Non-archimedean F-bundles were considered in \cite{KKPY_Birational_invariants}, and play a pivotal role in the application from Gromov-Witten theory to birational geometry.
Discussions with Ludmil Katzarkov, Maxim Kontsevich and Tony Pantev provided invaluable motivations, perspectives and ideas for this paper.
Special thanks to Hiroshi Iritani who generously answered our questions and pointed out relevant literature.
We are grateful to Sheel Ganatra, Mark Gross, Yuki Koto, Todor Milanov, Yong-Geun Oh, Constantin Teleman, Yukinobu Toda, Song Yu and Xiaolei Zhao for inspiring discussions around the subject.
The authors were partially supported by NSF grants DMS-2302095 and DMS-2245099.

\section{Basic definitions and examples} \label{sec:F-bundle}

In this section, we give the basic definitions regarding F-bundles and give the example of the A-model F-bundle. 
\subsection{Notion of F-bundle}

Let $\bbD_u$ denote the germ at $0$ in a $\bbk$-analytic closed unit disk with coordinate $u$.

\begin{definition}[F-bundle] \label{def:F-bundle}
    Let $B$ be a smooth $\bbk$-analytic space (resp.\ a smooth formal scheme over $\bbk$).
    An \emph{F-bundle} $(\cH, \nabla)$ over $B$ consists of a vector bundle $\cH$ over $B\times \bbD_u$ (resp.\ over $B\times \Spf \bbk\dbb{u}$), and a meromorphic flat connection $\nabla$ on $\cH$ with poles along $u=0$, such that $\nabla_{u^2\partial_u}$ and $\nabla_{u\xi}$ are regular for any tangent vector field $\xi$ on $B$.
\end{definition}

For applications to Gromov-Witten theory (see \cref{subsection: non-archimean F-bundle example}), the base $B$, the vector bundle $\cH$ and the connection $\nabla$ should all be understood in the context of supergeometry (see \cite[\S 4]{Manin_gauge_field_theory_and_complex_geometry}).

Given a map $f \colon B' \to B$, the pullback $f^{\ast} (\cH ,\nabla ) \coloneqq ((f\times\id_u )^{\ast} \cH , (f\times \id_u)^{\ast} \nabla )$ is an F-bundle on $B'$.

In the formal case, we introduce the notion of logarithmic F-bundle. 

\begin{definition}[Logarithmic F-bundle] \label{def:log-F-bundle}
  Let $B$ be a smooth formal scheme over $\bbk$ together with a normal crossing divisor $D \subset B$.
  A \emph{logarithmic F-bundle} over $(B,D)$ consists of a vector bundle $\cH$ over $B \times \Spf \bbk\dbb{u}$ and a meromorphic flat connection $\nabla$ on $\cH$ with poles along $u=0$, such that $\nabla_{u^2\partial_u}$ and $\nabla_{u\xi}$ are regular for any log tangent vector field $\xi$ on $B$.
\end{definition}

Below we formulate several definitions for logarithmic F-bundles, which also apply to non-archimedean F-bundles, up to replacing logarithmic tangent vectors by analytic tangent vectors.

\begin{remark}[Restriction to $u=0$] \label{remark:restriction-u=0-linear-map}
    Let $(\cH ,\nabla )/(B,D)$ be a logarithmic F-bundle and $\xi$ a logarithmic vector field on $B$.
    The failure of $\cO_B$-linearity of the operator $\nabla_{u\xi}$ is given by the symbol
    \begin{align*}
      \sigma (\nabla_{u\xi})\colon  T^{\ast} B\otimes_{\cO_{B\times\Spf \bbk\dbb{u}}} \cH & \longrightarrow  \cH  \\
      df\otimes h & \longmapsto  [\nabla_{u\xi},f]h.  
    \end{align*}
  We have $\sigma (\nabla_{u\xi} ) (df\otimes h) = df(u\xi ) h$, which vanishes at $u=0$.
  We thus obtain a map
  \begin{equation} \label{eq:mu-map}
  \begin{aligned}
    \mu\colon  T B(-\log D) & \longrightarrow  \End_{\cO_B} (\cH \vert_{u=0} ) \\
    \xi &\longmapsto  \nabla_{u\xi} \vert_{u=0}.
  \end{aligned}
  \end{equation}
In a similar way, the restriction of $\nabla_{u^2\partial_u }$ to $\cH\vert_{u=0}$ is $\cO_B$-linear.
\end{remark}

Let $b = \Spec \bbk\rightarrow B$ be a closed point.
The map \eqref{eq:mu-map} induces a map
\begin{equation} \label{eq:residue-map}
  \mu_b\colon T_b B(-\log D) \longrightarrow \End (\cH_{b,0} ) .
\end{equation}

Let $K_b$ denote the action of $\nabla_{u^2 \partial_u}$ on $\cH_{b,0}$.
The flatness of $\nabla$ implies that the image of $\mu_b$ consists of commuting operators, which also commute with $K_b$.

\begin{definition} \label{definition:maximal-F-bundle}
  A logarithmic F-bundle $(\cH,\nabla)/(B,D)$ is called \emph{maximal} (resp.\ \emph{over-maximal}) at a closed point $b = \Spec \bbk \to B$ if there exists a \emph{cyclic vector} for the action $\mu_{b}$, i.e.\ a vector $h\in \cH_{b,0}$ such that the map 
  \[ T_b B(-\log D) \longto \cH_{b,0},\quad v\longmapsto \mu_b(v)(h) \]
  is an isomorphism (resp.\ epimorphism).
  It is called maximal (resp.\ over-maximal) if it is maximal (resp.\ over-maximal) everywhere.
\end{definition}

In the maximal case, the dimension of $T_b B(-\log D)$ is equal to the rank of $\cH$, and $\mu_b$ induces an inclusion from $T_b B(-\log D)$ to $\End(\cH_{b,0})$.
We obtain a commutative associative product structure on $T_b B$, given by 
\begin{equation} \label{eq:product-structure-max-F-bundle}
 \mu_b(v_1\star v_2)(h)=  \mu_b(v_2)\circ \mu_b(v_1)(h).
\end{equation}

\begin{definition} \label{definition:Euler-vector-field}
  Let $(\cH ,\nabla )/(B,D)$ be a maximal logarithmic F-bundle.
  The unique logarithmic vector field $\Eu$ on $B$ with $\mu (\Eu ) = K \coloneqq \nabla_{u^2\partial_u}\vert_{u=0}$ is called the \emph{Euler vector field}. 
\end{definition}

\begin{definition} \label{def:framing}
  For a logarithmic F-bundle $(\cH, \nabla)$ over $(B,D)$, a \emph{framing} is another flat connection $\nabla^\fr$ on $\cH$ without poles, such that in the local trivializations of $\cH$ given by $\nabla^\fr$, if we denote by $H$ the vector space of local flat sections, the original connection $\nabla$ has the form
  \begin{equation}\label{eq:framing}
    \nabla_{\partial_u} = \partial_u + \frac{1}{u^2} \bK + \frac{1}{u} \bG, \qquad \nabla_\xi = \xi + \frac{1}{u} \bA(\xi)
  \end{equation}  
  for any logarithmic vector field $\xi$ on $B$, where $\bK, \bG$ are $\End(H)$-valued functions on $B$, and $\bA$ is an $\End(H)$-valued 1-form on $B$.
\end{definition}

We give the definition of product of logarithmic F-bundles. 
The definition is analogous in the non-archimedean case.

\begin{definition}[Product of F-bundles] \label{def:product-F-bundle}
    The product of two logarithmic F-bundles $(\cH_1,\nabla_1 ) /(B_1,\allowbreak D_1)$ and $(\cH_2,\nabla_2)/(B_2,D_2)$ is the F-bundle $(\cH ,\nabla ) / (B,D)$ defined over $B = B_1\times B_2$, with divisor $D= (D_1\times B_2)\cup (B_1\times D_2)$, by 
    \begin{align*}
        \cH &= \pr_1^{\ast}\cH_1 \oplus \pr_2^{\ast}\cH_2 , \\
        \nabla &= \pr_1^{\ast}\nabla_1 \oplus \pr_2^{\ast}\nabla_2,
    \end{align*}
    where $\pr_i\colon B_1\times B_2\times \Spf\bbk\dbb{u}\rightarrow B_i\times \Spf\bbk \dbb{u}$ denotes the projection for $i=1, 2$.
\end{definition}

\subsection{Example of A-model F-bundle} \label{subsection: non-archimean F-bundle example}

Let $X$ be a smooth projective variety over $\bbC$.
The rational Gromov-Witten invariants of $X$ can be encoded in an F-bundle, called the A-model F-bundle associated to $X$, also known as the quantum D-module (see \cite{Givental_Homological_geometry_and_mirror_symmetry}).
Here we will give a logarithmic version and a non-archimedean version of the A-model F-bundle.

\subsubsection{Gromov-Witten potential and quantum product}

Fix a homogeneous basis $(T_i)_{0\leq i\leq N}$ of $H^{\ast} (X,\bbQ )$, such that $T_0=\1$ is the unit, and $(T_1,\dots, T_k)$ is a basis of $H^2 (X,\bbQ )$. 
Let $(T^i)_{0\leq i\leq N}$ denote the dual basis with respect to the cup product pairing.

Let $\bbQ \dbb{\NE (X,\bbZ )}$ denote the completion of $\bbQ [\NE (X,\bbZ )] = \bbQ [q^{\beta}\, \vert\, \beta\in \NE (X,\bbZ )]$ with respect to the maximal ideal $(q^{\beta} , \,\beta\neq 0)$.
We write $\bbk\dbb{\NE (X,\bbZ )}\coloneqq \bbQ \dbb{\NE (X,\bbZ )}\otimes_{\bbQ}\bbk$.

The \emph{genus $0$ Gromov-Witten potential} is 
\begin{equation} \label{eq:GW_potential}
  \Phi = \sum_{n\ge 0,\beta} \frac{q^\beta}{n!} \sum_{i_1,\dots,i_n} \braket{T_{i_1}\cdots T_{i_n}}_{0,n}^{\beta} t_{i_1}\cdots t_{i_n} \in \bbQ \dbb{\NE (X,\bbZ )}\dbb{t_0,\dots, t_N} ,
\end{equation}
where $\langle \cdots \rangle_{0,n}^{\beta}$ denotes the Gromov-Witten invariants of $X$ of genus $0$, class $\beta$ and observables $T_{i_1}, \dots, T_{i_n}$. 

The \emph{(big) quantum product} is given by
\begin{equation}\begin{aligned} \label{eq:quantum_product}
\star \colon H^{\ast} (X,\bbQ ) \otimes H^{\ast} (X,\bbQ ) &\longto H^{\ast} (X,\bbQ ) \otimes \bbQ \dbb{\NE (X,\bbZ )}\dbb{t_0,\dots, t_N} \\
T_i \star T_j &\longmapsto \sum_r \frac{\partial^3\Phi}{\partial t_i \partial t_j \partial t_r} T^r ,
\end{aligned}\end{equation}
where
\begin{equation} \label{eq:derivative-GW-potential}
  \frac{\partial^3\Phi}{\partial t_i \partial t_j \partial t_r} = \sum_{n\ge 0,\beta} \frac{q^\beta}{n!} \sum_{i_1,\dots,i_n} \braket{T_i T_j T_r T_{i_1}\cdots T_{i_n}}_{0,n+3}^{\beta} t_{i_1}\cdots t_{i_n} .
\end{equation}

In \cref{sec:app-projective-bundle}, we will use a quantum product at a shifted origin, which we explain in the following lemma.

\begin{lemma}\label{lemma:shifted-quantum-product}
Let $\Delta (a) = \sum_{0\leq i\leq N} a_i  T_i\in H^{\ast} (X,\bbk )$.
Assume $\Delta (a)$ has no terms of degree 1 or 2.
Then applying the shift $t = (t_0,\dots, t_N) \mapsto t+a = (t_0+a_0 , \dots , t_N+a_N)$ to $\Phi$ produces a well-defined element $\Phi (t+a)\in\bbk\dbb{\NE (X ,\bbZ )}\dbb{t_0,\dots, t_N}$.
\end{lemma}

\begin{proof}
    Before the shift, for $\alpha=(\alpha_0,\dots,\alpha_N)\in \bbN^{N+1}$, the coefficient of the monomial $q^{\beta} t_0^{\alpha_0}\cdots t_N^{\alpha_N}$ in $\Phi$ is $\frac{1}{\alpha_0!\cdots \alpha_N!}\langle T_0^{\alpha_0}\cdots T_N^{\alpha_N} \rangle_{0,\vert\alpha\vert}^{\beta} $. 
        The coefficient of the monomial $t_0^{\alpha_0}\cdots t_N^{\alpha_N}$ of $\Phi(t+a)$ is given by evaluating 
    \[ \frac{1}{\alpha_0!\cdots \alpha_N!} \frac{\partial^{|\alpha|}\Phi(t+a)}{\partial^{\alpha_0}t_0\cdots \partial^{\alpha_N}t_N}\] 
    at $t=0$.
    By the chain rule, this is the same as evaluating the derivative of the unshifted potential $\Phi$ at $t=a$.
    So, it is enough to check that this evaluation makes sense, i.e.\ that it gives an element in $\bbk\dbb{\NE (X,\bbZ )}$. 
    The coefficient of $q^{\beta}$ in the derivative of $\Phi$ is 
    \begin{equation}\label{eq:coeff-shifted-potential}
        \sum_{\gamma\in\bbN^{N+1}} \frac{1}{\gamma_0! \cdots \gamma_N!} \langle T_0^{\alpha_0+\gamma_0} \cdots T_N^{\alpha_N+\gamma_N} \rangle_{0,\vert\alpha\vert+\vert\gamma\vert}^{\beta} t_0^{\gamma_0} \cdots t_N^{\gamma_N}.
    \end{equation}
    We claim that the above sum is finite when evaluated at $a$.
    By the unit axiom of Gromov-Witten invariants, the part of the sum with $\gamma_0 >0$ is finite:
    if $T_0$ appears in a nonzero $n$-pointed Gromov-Witten invariant, then $n = 3$ and $\beta = 0$.
    We now prove that there are finitely many terms with $\gamma_0 = 0$.
    If a nonzero Gromov-Witten invariant $\langle T_0^{\alpha_0} T_1^{\alpha_1+\gamma_1} \cdots  T_N^{\alpha_N+\gamma_N}\rangle_{0,\vert\alpha\vert+\vert\gamma\vert}^{\beta}$ contributes to the sum, the formula for the virtual dimension gives
    \begin{equation*}
        \sum_{0\leq i\leq N}\alpha_i\codim  T_i + \sum_{1\leq i\leq N} \gamma_i \codim  T_i = 2(\dim X - 3 + \vert\alpha\vert + \vert\gamma\vert + \beta\cdot c_1T_X) .
    \end{equation*}
    Since we assume that there is no shift in the $H^1$ and $H^2$-directions, the monomial $t_0^{\gamma_0}\cdots t_N^{\gamma_N}$ evaluated at $t=a$ is $0$, unless $\gamma_i = 0$ for $\codim  T_i \in\lbrace 1,2\rbrace$.
    Then, when evaluating the $\gamma_0=0$ part of \eqref{eq:coeff-shifted-potential} at $t=a$, nonzero terms satisfy $\codim T_i\geq 3$ for $\gamma_i\neq 0$.
    We deduce
    \[3\vert\gamma\vert\leq \sum_{1\leq i\leq N}\gamma_i \codim T_i = 2\vert\gamma\vert + \mathrm{constant}.\]
    It follows that the sum \eqref{eq:coeff-shifted-potential} is finite at $t=a$, completing the proof.
    \end{proof}

By \cref{lemma:shifted-quantum-product}, the quantum product is also well-defined on a formal neighborhood of the shifted point $\Delta(a) \in H^*(X,\bbk)$.

\subsubsection{Logarithmic A-model F-bundle}

Let $U$ be the formal neighborhood of a cohomology class $\Delta (a)\in H^{\ast} (X,\bbk )$ at which the quantum potential is well-defined.
Using the basis $(T_i)_{0\leq i\leq N}$, we write $U = \Spf \bbk \dbb{t_0,\dots ,t_N}$.

For $\xi\in H^2 (X,\bbk )$, we define a derivation $\xi q\partial_{q}$ of $\bbk \dbb{\NE (X,\bbZ )}$ by
\[\xi q \partial_{q} (q^{\beta} ) \coloneqq (\beta\cdot \xi ) q^{\beta} .\]

\begin{definition} \label{definition:A-model-F-bundle}
    The \emph{logarithmic A-model F-bundle of $X$ at base point $\Delta (a)$} is the logarithmic F-bundle $(\cH ,\nabla )$ over $\Spf \bbk\dbb{\NE (X ,\bbZ )} \times U$ defined as follows:
    \begin{enumerate}[wide]
        \item The bundle $\cH$ is trivial with fiber $H^{\ast} (X,\bbk )$.
        \item  Let
    \begin{align*}
	   \bK&\coloneqq\bigg[c_1(T_X) + \sum_{i:\deg T_i\neq 2} \frac{\deg T_i-2}{2} t_iT_i\bigg]\star,\\
	   \bG&\coloneqq \frac{1}{2}(\deg_X  -\dim X) ,\\
	   \bA (\tau ) &\coloneqq \tau \star \, ,\quad \tau\in H^{\ast} (X ,\bbk ),\\
	   \bA (\xi ) &\coloneqq \xi \star \, ,\quad \xi\in H^{2} (X ,\bbk ),\\
    \end{align*}
    where $\deg_X (\alpha ) = i\alpha$ for $\alpha\in H^{i} (X,\bbk )$, and $\star$ is the quantum product shifted at $\Delta (a)$.
    The connection $\nabla$ is given by 
    \begin{align*}
	   \nabla_{\partial_u} & = \partial_u - \frac{1}{u^2}\bK + \frac{1}{u}\bG,\\
	   \nabla_{\partial_{\tau}} & = \partial_{\tau} + \frac{1}{u}\bA (\tau )  ,\\
    \nabla_{\xi q\partial_{q}} & = \xi q\partial_{q} + \frac{1}{u}\bA (\xi ).
    \end{align*}
   
    \end{enumerate}
\end{definition}

\subsubsection{Non-archimedean A-model F-bundle}

In the non-archimedean setting, $\bbk$ is a complete non-archimedean field of characteristic $0$ with a nontrivial valuation whose restriction to $\bbQ$ is trivial.

Let $N^1(X)/\Tor$ denote the Néron-Severi group of $X$ modulo torsion.
The valuation of $\bbk$ induces a map

\begin{equation}\label{eq:valuation-ample-cone}
    v\colon (N^1(X)/\Tor)\otimes_\bbZ\mathbb G_{\mathrm m/\bbk} \rightarrow (N^1(X)/\Tor)\otimes_\bbZ \bbR .
\end{equation}

Since the ample cone $\mathrm{Amp}(X)$ is open in $N^1(X)_\bbR$, its preimage $B_q\coloneqq v^{-1}(\mathrm{Amp}(X))$ is a $\bbk$-analytic space.
Let $B_t^\even$ be the product of a $\bbk$-analytic affine line and an open polyunit disk, where the affine line has coordinate $t_0$ and the polyunit disk has coordinates $t_i$ for $\deg T_i \in \{2,4,6,\dots\}$.
Let $B_t^\odd$ be the purely odd vector space with coordinates $t_i$ for $\deg T_i \in \{1, 3, 5, \dots\}$.
Let $B\coloneqq B_q\times B_t^\even \times B_t^\odd$.

\begin{lemma}\label{lemma:convergence-GW-potential}
    The genus $0$ Gromov-Witten potential \[\Phi=\sum_{n\geq 0,\beta}\frac{q^\beta}{n!}\sum_{i_1,\cdots,i_n}\langle T_{i_1}\cdots T_{i_n} \rangle_\beta t_{i_1}\cdots t_{i_n}\in \bbQ\dbb{\NE (X,\bbZ )}\dbb{\{t_i\}} ,\]
    defines an analytic function over $B$.
\end{lemma}

\begin{proof}
Let $\sigma \subset N^1(X)_\bbR$ be any simplicial cone generated by ample classes $\omega_1,\cdots,\omega_m$. 
Let $B_q'$ be the preimage of $\sigma$ under the valuation map \eqref{eq:valuation-ample-cone}, and $B'=B_q'\times B_t^\even \times B_t^\odd$.
Then $\Phi$ is analytic over $B'$, since the restriction of $\Phi$ to $B'$ is given by the power series with rational coefficients
\[\Phi = \sum_{n\geq 0,\beta}\frac{1}{n!}q_1^{\beta\cdot \omega_1}\cdots q_m^{\beta\cdot \omega_m}\sum_{i_1,\cdots,i_n}\langle T_{i_1}\cdots T_{i_n}\rangle_\beta t_{i_1}\cdots t_{i_n}\in \bbQ\dbb{\{q_j\},\{t_i\}} ,\]
which is polynomial in $t_0$ by the unit axiom.
Since the union of all such $\sigma$ covers the ample cone, the proof is complete.
\end{proof}

\cref{lemma:convergence-GW-potential} implies that the quantum product is convergent over the non-archimedean base space $B$.

\begin{definition}
  The \emph{non-archimedean A-model F-bundle of $X$} is the F-bundle $(\cH ,\nabla )$ over $B$ defined by the same formulas as in \cref{definition:A-model-F-bundle}.
\end{definition}

\subsubsection{Maximal logarithmic F-bundle}

The F-bundles defined above are not maximal because the base has larger dimension than the fibers.
We can cut down the base dimension by choosing one $q$-variable and eliminating one $t$-variable as follows.

Fix a nef class $\omega\in N^1 (X)$.
It induces a projection
\begin{equation} \label{eq:class_to_degree}
    \bbk  [\NE (X,\bbZ ) ]\rightarrow \bbk [q ] ,\quad q^{\beta}\mapsto q^{\beta\cdot\omega}.
\end{equation}

\begin{assumption} \label{ass:finitely_many_beta}
  Assume that for any $i_1, \dots, i_n$ and $d$, there are finitely many $\beta$ such that $\beta\cdot \omega =d$ and $\braket{T_{i_1}\cdots T_{i_n}}_{0,n}^\beta \neq 0$.
\end{assumption}

\begin{lemma} \label{lem:finitely_many_beta}
  \Cref{ass:finitely_many_beta} holds if there exists $\epsilon \in \bbQ$ such that $\omega + \epsilon c_1(T_X)$ is ample.
  In particular, it holds if $\omega$ is ample, or if $X$ is Fano.
\end{lemma}
\begin{proof}
  Recall that the virtual dimension of $\cM_{0, n}(X, \beta)$ is equal to $\dim X - 3 + \beta \cdot c_1(T_X) + n$.
  If $\braket{T_{i_1}\cdots T_{i_n}}_{0,n}^\beta \neq 0$, we have $\dim_\vir\cM_{0, n}(X, \beta) = \sum_{j=1}^n \codim T_{i_j}$.
  So $\beta \cdot c_1(T_X)$ is fixed given $T_{i_1},\dots,T_{i_n}$.
  If $\beta\cdot\omega$ is also given, then $\beta\cdot \big(\omega + \epsilon\, c_1(T_X)\big)$ is fixed too.
  This is only possible for finitely many $\beta$, since $\omega + \epsilon c_1(T_X)$ is assumed ample.
\end{proof}

\begin{lemma} \label{lem:degree_to_class}
  Under \cref{ass:finitely_many_beta}, the Gromov-Witten potential $\Phi \in \bbQ\dbb{\NE(X,\bbZ)}\dbb{t_0,\dots,t_N}$ as in \eqref{eq:GW_potential} induces an element $\Phi^\omega\in\bbQ\dbb{q}\dbb{t_0,\dots, t_N}$, via the projection \eqref{eq:class_to_degree}.
  Conversely, $\Phi$ is uniquely determined by $\Phi^\omega$.
\end{lemma}
\begin{proof}
  \Cref{ass:finitely_many_beta} guarantees that $\Phi^\omega$ is well-defined.
  Let us prove the other direction.
  Fix $i_1, \dots, i_n$ and $d$.
  Knowing $\Phi^\omega$, we can form the following series
  \[ \Psi = \sum_{r_1, \dots, r_k} \frac{1}{r_1! \cdots r_k!} \sum_{\beta \cdot \omega = d} \braket{T_{i_1} \cdots T_{i_n} T_1^{r_1} \cdots T_k^{r_k}}_{0,n+r_1+\cdots + r_k}^\beta s_1^{r_1} \cdots s_k^{r_k} \in \bbQ\dbb{s_1, \dots, s_k}, \]
  where $T_1,\dots, T_k$ constitute a basis of $H^2(X,\bbQ)$.
  By the divisor axiom, we have
  \[ \Psi = \sum_{r_1, \dots, r_k} \frac{1}{r_1! \cdots r_k!} \sum_{\beta \cdot \omega = d} \braket{T_{i_1} \cdots T_{i_n}}_{0,n}^\beta (\beta\cdot T_1)^{r_1} s_1^{r_1} \cdots (\beta\cdot T_k)^{r_k} s_k^{r_k} \in \bbQ\dbb{s_1, \dots, s_k}. \]
  Comparing the coefficients, we conclude that every $\braket{T_{i_1} \cdots T_{i_n}}_{0,n}^\beta$ is uniquely determined by $\Psi$, therefore by $\Phi^\omega$.
\end{proof}

\begin{example}[Maximal A-model F-bundle] \label{example:modified-quantum-F-bundle}
  Assume $\omega = T_1$ satisfies \cref{ass:finitely_many_beta}.
  Let $\Delta (a)\in H^{\ast} (X,\bbk )$ be a cohomology class at which the quantum potential is well-defined.
  Let $U=\Spf \bbk\dbb{t_0,\cdots, t_N}$ be the formal neighborhood of $\Delta (a)$ in $H^{\ast} (X,\bbk)$, $U'\subset U$ the closed subspace given by $t_1=0$, and $B' = \Spf \bbk\dbb{q} \times U'$.
  Then the potential $\Phi^{\omega}$ in \cref{lem:degree_to_class} produces a maximal logarithmic F-bundle over $B'$ by the same formulas as in \cref{definition:A-model-F-bundle}.
  Indeed, the multiplicative unit $\mathbf 1$ is a cyclic vector at $0$ by the unit axiom.
\end{example}

\addtocontents{toc}{\protect\setcounter{tocdepth}{2}}
\section{Spectral decomposition of maximal F-bundles} \label{sec:decomposition-max-F-bundles}

In this section, we establish the spectral decomposition theorem for maximal F-bundles in the formal and non-archimedean settings, see \cref{thm:eigenvalue_decomposition,thm:NA-K-decomposition}.
We first prove in \S\ref{subsec:preliminary-decomposition} formal and non-archimedean analogs of the Frobenius theorem in differential geometry using an argument that we call ``generalized flatness''.
We study the decomposition of the base as F-manifolds in \cref{subsec:decomposition-F-manifolds}.
The spectral decomposition theorems are presented and proved in \cref{subsec:decomposition-theorems}.

Recall that $\bbk$ is a field of characteristic $0$.
In the non-archimedean setting, we equip $\bbk$ with a complete nontrivial valuation whose restriction to $\bbQ$ is trivial.

\subsection{Frobenius theorem} \label{subsec:preliminary-decomposition}

\subsubsection{Generalized flatness for systems of PDEs}

We prove a criterion ensuring the existence of a unique formal solution to some systems of quasi-linear PDEs in \cref{theorem:existence-solution-flat-PDE}. 
We also prove a non-archimedean version in a special case in \cref{lemma:generalized-flatness-NA}.
Throughout, we set $M_0 \coloneqq \bbk^m$.
We denote by $\mathfrak{m}$ the maximal ideal $(t_1,\dots, t_n)$ in $\bbk\dbb{t_1,\dots, t_n}$.

\begin{notation} \label{notation:tuple-integers}
We use the following notations for tuples of integers:
\begin{enumerate}
    \item Let $\preceq$ denote the partial order on $\bbN^n$ defined by
    \[(r_i)_{1\leq i\leq n}\preceq (s_i)_{1\leq i\leq n} \Longleftrightarrow \forall 1\leq i\leq n,\; r_i\leq s_i .\]
    \item For $r=(r_i)\in\bbN^n$, let $\vert r\vert \coloneqq \sum_{1\leq i\leq n} r_i$.
    \item For $r = (r_i)_{1\leq i\leq n}\in\bbN^n$ and $1\leq j\leq n$, we set
    \[\tau_j (r) \coloneqq (r_1,\dots, r_{j-1} , r_j+1 , r_{j+1} , \dots , r_n)\in\bbN^n .\]
\end{enumerate}
\end{notation}

\begin{definition}\label{definition:generalized-flat-PDE}
  Let $\mathcal{D} = (D_i\colon M_0\otimes_{\bbk}\fm \rightarrow M_0\otimes_{\bbk}\bbk\dbb{t_1,\dots, t_n})_{1\leq i\leq n}$ be a system of differential operators of the form $D_i = \partial_{t_i} - f_i$, with $f_i\colon M_0\otimes_{\bbk}\fm \rightarrow M_0\otimes_{\bbk}\bbk\dbb{t_1,\dots, t_n}$ an arbitrary map.
  We say that the system $\mathcal{D}$ is \emph{generalized flat} if the two following conditions are satisfied:
  \begin{enumerate}
  \item For every $d\in\bbN$ and every $1\leq i\leq n$, the composition 
  \[ M_0\otimes_{\bbk}\fm \overset{f_i}{\longrightarrow} M_0\otimes_{\bbk}\bbk\dbb{t_1,\dots, t_n}\longrightarrow M_0\otimes_{\bbk}\left( \bbk\dbb{t_1,\dots, t_n} /\fm^d\right) \]
  factors through $M_0\otimes_{\bbk}\left(\fm /\fm^d\right)$. 
  \item If $\varphi \in M_0\otimes_{\bbk}\fm$ satisfies $D_i(\varphi) = 0\mod \mathfrak{m}^d$ for all $1\leq i\leq n$, then $\partial_{t_i} (f_j(\varphi)) = \partial_{t_j} (f_i(\varphi))\mod\mathfrak{m}^d$ for all $1\leq i,j\leq n$.
  \end{enumerate}
\end{definition}

\begin{remark}
    Condition (1) means that for $\varphi\in M_0\otimes_{\bbk}\fm$, the total $t$-degree $d$ terms of $f_i (\varphi)$ depend on terms in $\varphi$ of total $t$-degree at most $d$. 
    This assumption allows to solve the associated system of PDEs recursively.
    It is automatically satisfied if the components of $f (\varphi )$ are power series in the components of $\varphi$.
\end{remark}

Our notion of generalized flat systems of PDEs allows to prove the following existence and uniqueness result.

\begin{proposition}\label{theorem:existence-solution-flat-PDE}
  Let $(D_i\colon M_0\otimes_{\bbk}\fm \rightarrow M_0\otimes_{\bbk}\bbk\dbb{t_1,\dots, t_n})_{1\leq i\leq n}$ be a generalized flat system of differential operators.
  Then there exists a unique $\varphi\in M_0\otimes_{\bbk}\fm$ satisfying $D_i(\varphi ) =0$ for all $1\leq i\leq n$.
\end{proposition}

\begin{proof}
    In this proof, for $\varphi\in  M_0\otimes_{\bbk}\fm$ and $\ell\in\bbN^n$, we denote by $[f_i(\varphi )]_{\ell}$ the coefficient of $t^{\ell}$ in $f_i (\varphi )$.
    
    For the uniqueness, if $\varphi = \sum_{\ell\in\bbN^n} \varphi_{\ell} t^{\ell}$ is a solution of the differential system, then $\varphi$ satisfies the recursive relations with respect to $t$-monomials
    \begin{equation} \label{eq:recursion-t-direction}
        (\ell_i+1 ) \varphi_{\tau_i (\ell) } = [f_i(\varphi)]_{\ell}.
    \end{equation}
    This uniquely determines the coefficients of $\varphi$ from the initial condition $\varphi_0 = 0$. 
    
    For the existence, we construct inductively on $d\in\bbN$ an element $\varphi^{(d)}\in M_0\otimes_{\bbk} \fm$ such that
  \begin{enumerate}
    \item $\varphi^{(d)}$ has terms of degree at most $d+1$,
    \item if $d\geq 1$, then $\varphi^{(d)} = \varphi^{(d-1)}\mod \mathfrak{m}^{d+1}$,
    \item $D_i (\varphi^{(d)})  = 0 \mod \mathfrak{m}^{d+1}$ for all $1\leq i\leq n$.
  \end{enumerate}
  Set $\varphi^{(0)} \coloneqq \sum_{i=1}^n [f_i(0)]_0 t_i$, it satisfies (1), (2) and (3) for $d=0$. 

  For the inductive step, fix $d\in\bbN$ and assume $\varphi^{(d)}$ is constructed. 
  Given $\ell\in\bbN^n$ with $\vert \ell\vert = d+2$, there exists a minimal index $i_0$ and a unique $\ell '\in\bbN^n$ such that $\ell =\tau_{i_0} (\ell ')$. 
  The index $i_0$ corresponds to the first nonzero component of $\ell$.
  We set $\varphi_{\ell} \coloneqq \frac{1}{\ell_{i_0}} [f_{i_0} (\varphi^{(d)})]_{\ell'}$,
  and define 
  \[\varphi^{(d+1)} \coloneqq \varphi^{(d)} + \sum_{\substack{\ell\in\bbN^n\\ \vert\ell\vert = d+2}} \varphi_{\ell} t^{\ell} .\]
  By construction $\varphi^{(d+1)}$ satisfies (1) and (2), it remains to check (3).
  By the inductive assumption (2) and Condition (1) of generalized flatness, we have $[f_i (\varphi^{(d+1)} )]_{\ell} = [f_i (\varphi^{(d)})]_{\ell}$ for all $\ell\in\bbN^n$ such that $\vert\ell\vert\leq d+1$.
  Thus we only need to check that the added coefficients $\varphi_{\ell}$ with $\vert\ell\vert = d+2$ satisfy the recursive relations (\ref{eq:recursion-t-direction}) for all $1\leq i\leq n$. 

  Fix $\ell\in\bbN^n$ with $\vert\ell\vert = d+2$, and an index $i$. 
  Let $i_0$ be as in the definition of $\varphi_{\ell}$, then there exists a unique $\ell_0\in\bbN^n$ with $\vert\ell_0\vert = d$ such that $\ell = \tau_i\tau_{i_0} (\ell_0 )  =\tau_{i_0}\tau_i (\ell_0 )$.
  By the construction of $\varphi_{\ell}$, the recursive relation (\ref{eq:recursion-t-direction}) in the $t_i$-direction is equivalent to 
  \[\ell_i [f_{i_0} (\varphi^{(d+1)})]_{\tau_i (\ell_0) } = \ell_{i_0} [f_i(\varphi^{(d+1)})]_{\tau_{i_0} (\ell_0 )}.\]
  Since $\vert\tau_i (\ell_0 ) \vert = \vert\tau_{i_0} (\ell_0 )\vert = d+1$, the induction hypothesis (2) and Condition (1) of generalized flatness imply $[f_{i_0} (\varphi^{(d+1)})]_{\tau_i (\ell_0) } = [f_{i_0} (\varphi^{(d)})]_{\tau_i (\ell_0) }$, and similarly for the right hand side.
  Then the recursion relation for $\ell$ is equivalent to 
  \[[\partial_{t_{i_0}} f_i(\varphi^{(d)})]_{\ell_0} = [\partial_{t_i} f_{i_0} (\varphi^{(d)})]_{\ell_0} ,\]
  which follows from Condition (2) of generalized flatness.
  We conclude that $\varphi^{(d+1)}$ satisfies (3), proving the inductive step.
  
  Condition (2) of the construction implies that $\lbrace \varphi^{(d)}\mod \mathfrak{m}^d\rbrace_{d\geq 0}$ is an inductive system producing a well-defined element $\tvarphi\in M_0\otimes_{\bbk}\bbk\dbb{t_1,\dots, t_n}$ such that $\tvarphi = \varphi^{(d)} \mod\mathfrak{m}^{d+2}$ for all $d\geq 0$.
  Condition (3) of the construction implies that $\tvarphi$ satisfies the recursive relations \eqref{eq:recursion-t-direction} for all $\ell\in\bbN^n$, hence it is a solution of $D_i(\varphi ) = 0$.
  Thus $\tvarphi$ satisfies $D_i(\tvarphi ) = 0$ for $1\leq i\leq n$, completing the proof.
\end{proof}

We denote by $T_n$ the Tate $\bbk$-algebra in $n$ variables. 
For $\rho\in \sqrt{\vert \bbk^{\times}\vert}$, we denote by $T_n (\rho )$ the $\bbk$-affinoid algebra associated to the closed polydisk of radius $\rho$ and dimension $n$ (\cite[\S6.1.5]{Bosch_Non-Archimedean_analysis}), consider the norm
\[\bigg\vert \sum_{\alpha\in\bbN^n} a_{\alpha} t^{\alpha}\bigg\vert_{\rho} \coloneqq \max_{\alpha} \vert a_{\alpha} \vert \rho^{\vert\alpha\vert}.\]

\begin{lemma} \label{lemma:generalized-flatness-NA}
    For $1\leq i\leq n$ and $1\leq k\leq m$, let $Y_i^k\in T_m = \bbk\langle x_1,\dots , x_m\rangle$.
    Let $\vert Y\vert \coloneqq \max_{1\leq i,k\leq n} \vert Y_i^k\vert$, assume $\vert Y\vert >0$.
    Let $f = (f_k )_{1\leq k\leq m}\in M_0\otimes_{\bbk} \fm$ satisfying ($1\leq i\leq n$, $1\leq k\leq m$)
    \[\partial_{t_i} f_k = Y_i^k (f_1,\dots, f_m).\]
    Then the components of $f$ are convergent on the open polydisk of radius $\vert Y\vert^{-1}$ and have norms bounded by $1$. 
    Equivalently, $f$ induces a map $\Sp T_n(\rho )\rightarrow \Sp T_n$ for all $\rho \in\sqrt{\vert\bbk^{\times}\vert}$ with $0< \rho < \vert Y\vert^{-1}$.
\end{lemma}

\begin{proof}
    Write $f_i = \sum_{\alpha\in\bbN^n} f_{i,\alpha} t^{\alpha}$ and $Y_i^k = \sum_{r\in \bbN^m} y_r^{(i,k)} x^{r}$.
    We have $\vert Y\vert = \sup \vert y_r^{(i,k)}\vert$.
    By assumption we have $f_{i,0} = 0$, which ensures that the composition $Y_i^k (f_1,\dots, f_m)$ is well-defined.
    
    For $d\in \bbN$, we set $v_d\coloneqq \max_{1\leq i\leq m, \vert \alpha\vert = d} \vert f_{i,\alpha}\vert$.
    We will prove $v_d\leq \vert Y\vert^d$ by induction on $d$.
    We have $v_0 = 0\leq 1 = \vert Y\vert^0$.
    Next, fix $d>0$ and assume we have proved $v_e\leq \vert Y\vert^e$ for all $e< d$.
    Let $\alpha\in\bbN^n$ with $\vert\alpha\vert = d-1$.
    Then for $1\leq k\leq n$, as in \eqref{eq:recursion-t-direction}, we have the recursion
    \[f_{i,\tau_k (\alpha )} = \frac{1}{\alpha_k + 1} [Y_i^k (f_1,\dots, f_m) ]_{\alpha} ,\]
    where the right hand side is the coefficient of $t^{\alpha}$ in $Y_i^k (f_1,\dots, f_m)$.
    We now express this coefficient. 
    For $r\in \bbN^m$, let $\cP (r,\alpha )$ denote the set of partitions of $\alpha$ into $\vert r\vert$-tuples.
    We write an element of $\cP (r,\alpha )$ as $\lbrace \alpha_1^{(1)} , \dots , \alpha_{r_1}^{(1)} ,\alpha_1^{(2)}  , \dots , \alpha_{r_m}^{(m)} \rbrace$, where $\alpha_p^{(q)}\in \bbN^n$ for each $p,q$. 
    The coefficient can then be expressed as the finite sum 
    \[[Y_i^k (f_1,\dots, f_m) ]_{\alpha}  = \sum_{r\in\bbN^m} y_r^{(i,k)} \sum_{ \lbrace \alpha_p^{(q)}\rbrace\in\cP (r,\alpha )} \prod_{1\leq q\leq m}\prod_{1\leq p\leq r_q} f_{q,\alpha_p^{(q)}}. \]
    We deduce
    \begin{align*}
        \vert f_{i,\tau_k (\alpha )} \vert &\leq \vert Y\vert \max_{\lbrace \alpha_p^{(q)}\rbrace\in\cP (r,\alpha )}\prod_{1\leq q\leq m}\prod_{1\leq p\leq r_q}\vert f_{q,\alpha_p^{(q)}}\vert \\
        &\leq  \vert Y\vert \max_{\lbrace \alpha_p^{(q)}\rbrace\in\cP (r,\alpha )}\prod_{1\leq q\leq m}\prod_{1\leq p\leq r_q} \vert Y\vert^{\vert \alpha_p^{(q)}\vert } \\
        &\leq \vert Y\vert\times \vert Y\vert^{\vert\alpha\vert} = \vert Y\vert^d.
    \end{align*}
    Let $0 < \rho < \vert Y\vert^{-1}$ in $\sqrt{\vert\bbk^{\times}\vert}$, we then have $\vert f_{i,\alpha} \vert \rho^{\vert\alpha\vert}\leq  (\rho \vert Y\vert )^{\vert \alpha\vert}\leq 1$.
    This implies that $f_i\in T_n (\rho )$, since $\rho\vert Y\vert <1$, and that $\vert f_i\vert_{\rho}\leq 1$, and the lemma follows.
\end{proof}

\subsubsection{Frobenius theorem}
We prove the formal and non-archimedean analogs of the Frobenius theorem in differential geometry, which states that a local basis of commuting vector fields comes from coordinates.

\begin{theorem} \label{theorem:Frobenius-theorem}
  Let $B = \Spf \bbk\dbb{t_1,\dots ,t_n}$ and let $(Y_i)_{1\leq i\leq n}$ be a commuting basis of vector fields on $B$.
  Then there exists a unique automorphism $\varphi\colon B\rightarrow B$ such that $d\varphi (\partial_{t_i} ) = \varphi^{\ast} Y_i$ for all $1\leq i\leq n$.
\end{theorem}

\begin{proof}  
Let $b$ be the closed point of $B$, given by $t_1 = \cdots = t_n = 0$.
Let $\mathfrak{m} = (t_1,\dots, t_n)$ denote the maximal ideal of $\bbk\dbb{t_1,\dots, t_n}$.
We write $Y_i = \sum_k Y_i^k \partial_{t_k}$, with $Y_i^k\in \bbk\dbb{t_1,\dots, t_n}$.
Working in coordinates, giving $\varphi\colon B\rightarrow B$ is equivalent to giving $\varphi_1,\dots ,\varphi_n\in \bbk\dbb{t_1,\dots, t_n}$ such that $\varphi_i (0) = 0$.
Furthemore, $\varphi$ is invertible if and only if the differential at $b$ is invertible, i.e.\ if and only if the matrix $\left( \frac{\partial \varphi_i }{\partial t_j} \right)_{1\leq i,j\leq n}$ is invertible at $t=0$.
The condition $d\varphi (\partial_{t_i} ) = \varphi^{\ast}Y_i$ is equivalent to 
\begin{equation} \label{eq:formal-Frobenius-PDE}
    \sum_{1\leq k\leq n} \frac{\partial \varphi_i}{\partial t_k}(t) \partial_{t_k} = \sum_{1\leq k\leq n} Y_i^k (\varphi_1 (t),\dots ,\varphi_n (t)) \partial_{t_k} .
\end{equation}
Since $\varphi_i (0) = 0$, the composition on the right hand side is well-defined.

For $1\leq i\leq n$, consider the first-order quasi-linear differential operator
\begin{equation*}
    \begin{aligned}
        D_i\colon \bbk^n\otimes_{\bbk} \fm &\longrightarrow \bbk^n\otimes_{\bbk}\bbk\dbb{t_1,\dots, t_n} \\
        (\varphi_1,\dots, \varphi_n)&\longmapsto \left( \frac{\partial \varphi_k}{\partial t_i} - Y_i^k (\varphi_1,\dots, \varphi_n) \right)_{1\leq k\leq n} .
    \end{aligned}
\end{equation*}
 Equation \eqref{eq:formal-Frobenius-PDE} is equivalent to $D_i (\varphi_1,\dots ,\varphi_n ) = 0$ for $1\leq i\leq n$.
 We will prove that the system $\lbrace D_i = 0\rbrace$ is generalized flat, and apply \cref{theorem:existence-solution-flat-PDE}.
  Condition (1) of \cref{definition:generalized-flat-PDE} is satisfied because the components of $Y_i$ are power series in the argument.
  We now check Condition (2).
  Assume $(\varphi_1,\dots, \varphi_n)\in \bbk^n\otimes_{\bbk} \mathfrak{m}$ satisfies ${D_i (\varphi_1,\dots,\varphi_n) = 0\mod \mathfrak{m}^d}$ for all $1\leq i\leq n$.
  Then, since $[Y_i,Y_j] = 0$, we have ($1\leq i,k\leq n$)
  \begin{align*}
  \frac{\partial(Y_j^k(\varphi_1,\dots ,\varphi_n))}{\partial t_i} &= \sum_s \frac{\partial\varphi_s }{\partial t_i} \frac{\partial Y_j^k}{\partial t_s}(\varphi_1,\dots, \varphi_n)  \mod\mathfrak{m}^d \\
  &= \sum_s Y_i^s (\varphi_1,\dots ,\varphi_n)\frac{\partial   Y_j^k }{\partial t_s} (\varphi_1,\dots,\varphi_n)\mod\mathfrak{m}^d  \\
  &=\sum_s Y_j^s (\varphi_1,\dots,\varphi_n)\frac{\partial Y_i^k}{\partial t_s} (\varphi_1,\dots ,\varphi_n) \mod\mathfrak{m}^d \\
  &= \frac{\partial (Y_i^k (\varphi_1,\dots ,\varphi_n))}{\partial t_j}\mod\mathfrak{m}^d.
  \end{align*}
	We deduce from \cref{theorem:existence-solution-flat-PDE} that the components $(\varphi_1 ,\dots, \varphi_n )$ of $\varphi$ are uniquely determined and that they can be constructed inductively.
	The associated morphism $\varphi\colon B\rightarrow B$ is automatically an automorphism, 
    because its differential at $b$ is given by the matrix $(Y_j^i (0) )_{1\leq i,j\leq n}$, which is invertible by assumption.
\end{proof}

\begin{lemma}\label{lemma:neighborhood-k-rational-point}
    Let $X$ be a $\bbk$-analytic space, and $x\in X$ a smooth $\bbk$-rational point. 
    There exists an admissible open neighborhood $U\subset X$ of $x$ and an open immersion $U\hookrightarrow \Sp T_n$, where $n= \dim_x X$.
\end{lemma}
\begin{proof}
    Since $x$ is a smooth rigid point, there exists an admissible affinoid neighborhood $U\subset X$ of $x$ and an étale map $U\rightarrow Y\coloneqq\Sp T_n$, with $n=\dim_x X$.
    Up to shrinking $U$, we may assume that $f^{-1} (f(x)) = \lbrace x\rbrace$.
    We will show that $f_x^{\ast}\colon \cO_{Y,f(x)}\rightarrow \cO_{X,x}$ is an isomorphism, then $f$ restricts to an open immersion on an affinoid open neighborhood of $x$ by \cite[7.3.3/Corollary 6]{Bosch_Non-Archimedean_analysis}. 
    By \cite[7.3.3/Proposition 5]{Bosch_Non-Archimedean_analysis}, it is enough to check that the induced morphism $\hat{f}_x^{\ast}\colon \widehat{\cO}_{Y,f(x)}\rightarrow \widehat{\cO}_{X,x}$ on the completed local rings is an isomorphism.

    Since $f$ is étale, we have $f_x^{\ast}(\fm_{f(x)}) = \fm_x$, in particular $\widehat{\cO}_{X,x}$ is a complete $\widehat{\cO}_{Y,f(x)}$-module.
    Since $x$ is a $\bbk$-rational, the map $\hat{f}_x^{\ast}$ is an isomorphism modulo $\fm_{f(x)}$, hence $\hat{f}_x^{\ast}$ is surjective by \cite[Tag 0315]{Stacks_project}.
    The Krull dimension of noetherian local rings is invariant under completion. 
    Since $\dim_x X = \dim_{f(x)} Y$, necessarily $\hat{f}_x^{\ast}$ is injective using \cite[Tag 00KW]{Stacks_project}.
    This concludes the proof.
\end{proof}

\begin{theorem}  \label{thm:NA-Frobenius-theorem}
    Let $B$ be a smooth $\bbk$-analytic space, and $(Y_1,\dots, Y_n)$ be a commuting basis of local vector fields around a rational point $b\in B$.
    Then, there exists admissible open neighborhoods $V\subset B$ of $b$ and $U\subset \Sp T_n$ of $0$ and an isomorphism $\varphi\colon U\rightarrow V$ such that $\varphi (b) = 0$ and $d\varphi (\partial_{t_i}\vert_U) = \varphi^{\ast} (Y_i\vert_V)$.
 \end{theorem}
\begin{proof}
    By \cref{lemma:neighborhood-k-rational-point}, we may assume that $B \simeq \Sp T_n$.
    We start by applying \cref{theorem:Frobenius-theorem} to the restriction of the vector fields $(Y_i)_{1\leq i\leq n}$ to a formal neighborhood $\widehat{B} = \Spf\bbk\dbb{t_1,\dots, t_n}$ of $0\in \Sp T_n$.
    This produces a unique formal automorphism $\widehat{\varphi} = (\varphi_1,\dots ,\varphi_n )\colon \widehat{B}\rightarrow\widehat{B}$ satisfying the relations \eqref{eq:formal-Frobenius-PDE}.
    We will prove that $\widehat{\varphi}$ extends to admissible open neighborhoods of $0$.

    Let $\vert Y\vert \coloneqq \max_i \vert Y_i\vert$, and let $\rho\in\sqrt{\vert\bbk^{\times}\vert}$ such that $\rho < \min (1,\vert Y\vert^{-1} )$.
    By \cref{lemma:generalized-flatness-NA}, $\widehat{\varphi}$ extends to a map $\varphi\colon \Sp T_n(\rho )\rightarrow \Sp T_n$.
        The truncations of $\varphi$ coincide with the truncations of $\widehat{\varphi}$.
    In particular, they induce isomorphisms $T_n / \fm^d \xrightarrow{\sim}T_n (\rho ) / \fm^d$ for all $d\geq 0$.
    We conclude the proof using \cite[\S 3.3 Lemma 18(ii)]{Bosch_Lectures_on_formal_and_rigid_geometry}.
\end{proof}

\subsection{Decomposition theorems for F-manifolds}\label{subsec:decomposition-F-manifolds}

In this subsection, we prove the decomposition theorems for formal and non-archimedean versions of F-manifolds, see \cref{theorem:decomposition-F-manifolds,thm:NA-decomposition-F-manifold}.

\subsubsection{Decomposition theorem for formal F-manifolds}

The notion of F-manifold was introduced by Hertling and Manin as a weaker version of Frobenius manifolds, see \cite{Hertling_Manin_Weak_Frobenius_manifolds} and the monograph \cite[I.\S5]{Manin_Frobenius_manifolds}.

\begin{definition}[F-manifold] \label{def:F-manifold}
  Let $B$ be a smooth formal scheme or a smooth $\bbk$-analytic space.
  An F-manifold structure on $B$ is a $\mathcal{O}_B$-bilinear commutative associative product $\star$ on the tangent bundle $TB$, satisfying the \emph{F-identity}: for any (local) vector fields $X,Y,Z,W$ we have
  \begin{equation} \label{eq:F-identity}
      P_{X\star Y}(Z,W)=X\star P_Y(Z,W) + (-1)^{\abs{X}\abs{Y}} Y\star P_X(Z,W),
  \end{equation}
  where
  \[P_X(Z,W)\coloneqq [X,Z\star W] - [X,Z]\star W - (-1)^{\abs{X}\abs{Z}} Z\star[X,W].\]
\end{definition}

We prove the following decomposition result for formal F-manifolds.

\begin{theorem}\label{theorem:decomposition-F-manifolds}
  Let $B$ be a formal neighborhood of a rational point $b$ in a smooth $\bbk$-variety.
  Let $\star$ denote an F-manifold structure with unit on $B$.
  Assume that there exists a splitting as $\bbk$-algebras
  \begin{equation}\label{eq:splitting-tangent-space}
      T_b B  = \bigoplus_{i\in I} A_i .
  \end{equation}
  Then there exists formal F-manifolds $(B_i,\star_i )$ such that 
  \begin{enumerate}
    \item $(B,\star )$ is isomorphic to $ \prod_{i\in I} (B_i,\star_i )$ as F-manifolds with unit,
    \item and the induced decomposition of $(TB,\star )$ restricts to \eqref{eq:splitting-tangent-space} at $b$.
  \end{enumerate}
\end{theorem}

The idea of the proof is the following.
We obtain a decomposition of $TB$ into sheaves of subalgebras in \cref{lemma:deformation-unital-associative-commutative-algebra}, induced from that of $T_bB$.
\cref{proposition:A-decomposition-theorem} will show that the direct summands of $TB$ define commuting foliations (in the sense of \cite[Definition 2.1]{Araujo_Fano-foliations}).
We can then integrate them using the Frobenius theorem (\cref{theorem:Frobenius-theorem}).

\begin{lemma}\label{lemma:deformation-unital-associative-commutative-algebra}
  Let $A$ be a unital associative commutative $\bbk$-algebra and $I$ a finite set.
  Assume $A$ admits a splitting $A \simeq \bigoplus_{i\in I} A_i$ as $\bbk$-algebras.
  Then the splitting extends over any deformation of $A$ over $\bbk\dbb{t_1,\dots, t_n}$.
\end{lemma}

\begin{proof}
  Let $\tR \coloneqq \bbk\dbb{t_1,\dots ,t_n}$ and let $\tA$ be an $\tR$-algebra which is a deformation of $A$.
  Let $\fm = (t_1,\dots, t_n)$, and for $k\in\bbN$, $A_k\coloneqq \tA /\fm^{k+1} \tA$ and $B_k \coloneqq (\tR / \fm^{k+1} )^{\oplus I}$.

  We will prove by induction on $\ell\geq 0$ that for any $0\leq k\leq \ell$, there are $\tR$-algebra maps $B_k\rightarrow A_k$ that fit into a commutative diagram
  \begin{equation}\label{cd:tower-deformation-algebra}
  \begin{tikzcd}
    A_{\ell} \ar[r] & A_{\ell-1} \ar[r] & \cdots\ar[r] & A_1\ar[r] & A_0 \\
    B_{\ell} \ar[r]\ar[u] & B_{\ell-1} \ar[r]\ar[u] & \cdots \ar[r]& B_1\ar[u]\ar[r] & B_0. \ar[u]
  \end{tikzcd}
  \end{equation}
  For $\ell=0$, the $\tR$-algebra structures on $A_0\simeq A$ and $B_0\simeq \bbk^{\oplus I}$ are induced by the compositions of the quotient map $\tR\rightarrow \bbk$ with the  structural maps $\bbk\rightarrow A$ and $\bbk\rightarrow \bbk^{\oplus I}$.
  In particular, the map $B_0\rightarrow A_0$ provided by the splitting $A\simeq \bigoplus_{i\in I} A_i$ is a map of $\tR$-algebras.

  Now assume that  the maps $B_k\rightarrow A_k$ are constructed for $k\leq \ell$.
  Let us prove that the dashed arrow exists in the commutative diagram of $\tR$-algebras
  \[\begin{tikzcd}
    A_{\ell+1}\ar[r] & A_{\ell} \\
    B_{\ell+1}\ar[r]\ar[u,dashed] & B_{\ell}. \ar[u]
  \end{tikzcd}\]
  In other words, we are looking for a lift of the composite map $B_{\ell+1}\rightarrow A_{\ell}$ to $A_{\ell+1}$.
  Since $\ker (A_{\ell+1}\rightarrow A_{\ell}) = \fm^{\ell+1} A_{\ell+1}$, the algebra $A_{\ell+1}$ is a square-zero extension of $A_{\ell}$.
  Then, the obstruction to the existence of this lift is a class in $\Ext_{\tR}^1 (\bbL_{B_{\ell+1}/\tR} \otimes_{B_{\ell+1}} A_{\ell} ,  \fm^{\ell+1} A_{\ell+1})$.
  Since
  \[\bbL_{B_{\ell+1}/\tR} \simeq \bbL_{B_{\ell+1} / (\tR / \fm^{\ell+1})}  = 0,\]
    the obstruction vanishes, and the lift always exists, concluding the induction.

  By functoriality of limits in the category of $\tR$-algebras, we obtain a map of $\tR$-algebras
  \[ \tR^{\oplus I} \simeq \lim_k B_k  \longrightarrow \lim_k A_{k} \simeq \tA ,\]
  concluding the proof.
\end{proof}

We now state two lemmas needed to prove \cref{proposition:A-decomposition-theorem}.

\begin{lemma} \label{lemma:submodule-torsion-free}
    Let $R$ be a local domain.
    Let $f\colon M\rightarrow N$ be a surjective morphism of finite free $R$-modules, and $D\subset M$ a free submodule.
    Assume (1) $D\cap \ker f = 0$, (2) $\rk D = \rk N$ and (3) $M/D$ is torsion-free.
    Then $f$ restricts to an isomorphism $D\xrightarrow{\sim} N$.
\end{lemma}

\begin{proof}
  Let $S \coloneqq \Frac (R)/R$.
    We have $N/f(D)\simeq M /(\ker f +D)$.
  We prove that this module is torsion-free.
  Since $\ker f \cap D= 0$, we have a short exact sequence
  \[0\longrightarrow \ker f\longrightarrow M/D \longrightarrow M/(\ker f+D) \longrightarrow 0 .\]
  Applying $\otimes_R S$ gives the exact sequence
  \[0 = \Tor_1^R (M/D, S) \longrightarrow \Tor_1^R (M/(\ker f+D),S)\longrightarrow \ker f\otimes_R S\overset{\varphi}{\longrightarrow} M/D\otimes_R S ,\]
  and we see that $M/(\ker f+D)$ is torsion-free if and only if $\varphi$ is injective.
  Since $M/D$ is torsion-free, the module $D\otimes_S R$ is identified with a submodule of $M\otimes_R S$ and we have $M/D\otimes_R S\simeq (M\otimes_R S)/ (D\otimes_R S)$.
    Since $M/\ker f\simeq N$ is torsion-free, the module $\ker f\otimes_R S$ is identified with a submodule of $M\otimes_R S$,
    and $\varphi$ corresponds to the composition
  \[\ker f\otimes_R S\longrightarrow M\otimes_R S\longrightarrow (M\otimes_R S) / (D\otimes_R S) ,\]
  where the first map is the canonical inclusion and the second map is the canonical projection.
    Then, since $\ker f +D$ is torsion-free, we have
  \[\ker (\varphi ) \simeq (\ker f\otimes_R S) \cap (D\otimes_R S) \simeq (\ker f\cap D)\otimes_R S = 0 .\]
  We deduce that $N/f(D)$ is torsion-free.
  But since $\rk N = \rk f(D)$, the quotient $N/f(D)$ is a torsion module.
  We conclude that $N/f(D) =0$, and the lemma follows.
\end{proof}

\begin{lemma} \label{lemma:foliation-commuting-vector-fields}
  Let $B = \Spf \bbk\dbb{t_1,\dots ,t_n}$.
  Let $\cD$ be a free $\cO_B$-subsheaf of $TB$ stable under the Lie bracket and such that $TB/\cD$ is torsion-free.
  Then $\cD$ admits an $\cO_B$-basis of commuting vector fields.
\end{lemma}

\begin{proof}
    We denote by $\partial_i$ the vector field associated to $t_i$.
    The coordinates $(t_i)_{1\leq i\leq n}$ provide a trivialization $TB = \bigoplus_{1\leq i\leq n} \cO_B \partial_i$.

    Let $m$ denote the rank of $\cD$, then up to reordering the coordinates we may assume $\cD\cap \bigoplus_{m+1\leq i\leq n} \cO_B \partial_i = 0$.
    If $m=n$ there is nothing to show.
    Assume $m<n$, then there exists $i_1$ such that $\cO_B\partial_{i_1}\cap \cD = 0$.
    Then $\cD^{(1)} \coloneqq \cD \oplus \cO_B \partial_{i_1}$ is a free $\cO_B$-module of rank $m+1$. We can thus apply the same argument inductively until we obtain a free $\cO_B$-module of rank $n$, and obtain in this way vector fields $(\partial_{i_1} ,\dots ,\partial_{i_{n-m+1}} )$ such that $\cD \cap \bigoplus_{m+1\leq j\leq n} \cO_B \partial_{i_j} = 0$.

    Let $B'= \Spf \bbk\dbb{t_1,\cdots,t_m}$ and $\pi\colon B\rightarrow B'$ denote the canonical projection. 
    Let $\psi\colon \cD\rightarrow \pi^{\ast} TB'$ denote the restriction of $d\pi\colon TB\rightarrow\pi^{\ast} TB'$.
    The kernel of $d\pi$ is $\bigoplus_{m+1\leq i\leq n}\cO_B \partial_i$, so $\psi$ is injective.
    By \cref{lemma:submodule-torsion-free}, $\psi$ is an isomorphism.
    Let $\partial_i'$ denote the vector field of $B'$ associated to $t_i$.
    We define $X_i \coloneqq \psi^{-1} (\pi^{\ast} \partial_i ')$.
    By construction $(X_i)_{1\leq i\leq m}$ is an $\cO_B$-basis of $\cD$.

    We now check that $[X_i, X_j] = 0$.
    The $\cO_B$-linearity of $d\pi$ and $\pi^{\ast}$ implies
    \[d\pi [X_i, X_j] = \pi^{\ast}[ \partial_i ', \partial_j' ]  = 0 .\]
    Since $[X_i, X_j]$ is a section of $\cD$ and $d\pi$ restricted to $\cD$ is an isomorphism, we deduce that $X_i$ and $X_j$ commute.
\end{proof}

\begin{proposition} \label{proposition:A-decomposition-theorem}
 Let $B = \Spf \bbk\dbb{t_1,\dots ,t_n}$ and $\star$ an F-manifold structure with unit on $B$.
  Assume that we have a decomposition into sheaves of subalgebras $(TB ,\star ) = \bigoplus_{i\in I} (\cD_i ,\star \vert_{\cD_i} )$.
  Then:
  \begin{enumerate}
      \item For all $i$ we have $[\cD_i ,\cD_i ]\subset \cD_i$.
      \item For $i\neq j$ we have $[\cD_i ,\cD_j ] =0$.
      \item There exists an automorphism $\varphi\colon B\rightarrow B$ and a partition $\lbrace 1,\cdots, n\rbrace = \coprod_{i\in I} J_i$ such that, for each $i\in I$, the pullback $\varphi^{\ast} \cD_i$ is generated by $\lbrace d\varphi (\partial_{t_j} ) \rbrace_{j\in J_i}$.
  \end{enumerate}
\end{proposition}

\begin{proof}
  For $i\in I$, let $p_i \colon TB\rightarrow \cD_i$ denote the projection, corresponding to the multiplication by the identity section $e_i$ of $\cD_i$.
  We have $p_i^2 = p_i$, $p_i\circ p_j =\delta_{ij}$ and $\bigoplus_{i\in I}p_i = \id$, thus $\ker p_i = \bigoplus_{j\neq i} \cD_j$.

    Let $i\in I$, we prove that $\cD_i$ is stable under Lie bracket.
    Let $X,Y$ be sections of $\cD_i$.
  Since $e_i\star X = X$, the F-identity gives
  \[P_X (e_i,Y) = e_i\star P_X (e_i, Y) + X\star P_{e_i} (e_i ,Y) .\]
  The left-hand side equals
  \[P_X (e_i,Y) = [X,Y] - [X,e_i]\star Y - e_i \star [X,Y] ,\]
  and the terms on the right-hand side are
  \begin{align*}
  e_i\star P_X (e_i,Y) &= e_i\star \left( [X,Y] - [X,e_i]\star Y - e_i\star [X,Y] \right) \\
  &= -e_i\star [X,e_i]\star Y \\
  &= - Y\star [X,e_i],
  \end{align*}
  and
  \begin{align*}
  X\star P_{e_i} (e_i,Y) &= X\star \left( [e_i,Y] - e_i\star [e_i ,Y] \right) \\
  &= X\star [e_i, Y] - X\star e_i\star [e_i,Y] \\
  &= 0,
  \end{align*}
  where we used $e_i\star X = X$, $e_i\star Y= Y$, $e_i\star e_i =e_i$ and the commutativity of the product.
  Thus, the F-identity above reduces to $[X,Y] = e_i\star [X,Y]$.
  Equivalently, $[X,Y]$ is a section of $\cD_i$, proving (1).

  Fix $i,j\in I$ with $i\neq j$.
  Let $X$ and $Y$ be sections of $\cD_i$ and $\cD_j$ respectively,
  in particular $e_i\star X =X$ and $e_j\star Y= Y$.
  We need to show $[X,Y] = 0$.
  We have 
  \begin{align*}
  [X,Y] &= [e_i\star X,e_j\star Y] \\
  &= P_{e_i\star X} (e_j , Y) + [e_i\star X, Y]\star Y + [e_i\star X , Y]\star e_j \\
  &= P_{e_i\star X} (e_j, Y) + \left( P_{e_j} (e_j ,X) +[e_j , e_i]\star X + [e_j ,X]\star e_i\right)\star Y \\
  &\quad + \left( P_Y (e_i,X) + [Y,e_i]\star X + [Y,X]\star e_i\right)\star e_j \\
  &= e_i\star P_X (e_j, Y) + X\star P_{e_i} (e_j, Y) + Y\star P_{e_j} (e_i,X) + e_j\star P_Y (e_i,X) \\
  &= (e_i-e_j) \star [X,Y] + X\star [e_i,Y] + Y\star [e_j ,X] .
  \end{align*}
  Multiplication by $e_i$ shows that $X\star [e_i, Y] = 0$.
  By symmetry, we also have $Y\star [e_j ,X] = 0$, so the equation reduces to 
  \[[X,Y] = (e_i-e_j)\star [X,Y] .\]
  Multiplication by $e_k$ for $k$ different from $i$ and $j$ gives $e_k\star [X,Y] = 0$, so $[X,Y]$ is a section of $\cD_i\oplus\cD_j$.
  We then have
  \[ (e_i +e_j)\star [X,Y] = [X,Y] = (e_i-e_j)\star [X,Y] .\]
  We deduce $e_j\star [X,Y] = 0$, and by symmetry $e_i\star [X,Y] = -e_i\star [Y,X] = 0$.
  Thus $[X,Y] = 0$, and (2) is proved.

  By (1) and (2), the decomposition $TB = \bigoplus_{i\in I} \cD_i$ is a decomposition into commuting subsheaves of Lie algebras.
  For each $i\in I$, we have $TB/\cD_i \simeq \bigoplus_{j\neq i} \cD_j$, which is torsion-free.
  By \cref{lemma:foliation-commuting-vector-fields}, $\cD_i$ admits an $\cO_B$-basis of commuting vector fields.
  By (2), these bases assemble into a basis of commuting vector fields for sections of $TB$.
  Then (3) follows by applying \cref{theorem:Frobenius-theorem} to the union of these bases.
\end{proof}

\begin{proof}[Proof of \cref{theorem:decomposition-F-manifolds}]
        By \cite[Tag 0C0S(2)]{Stacks_project}, we may assume that the base $B$ has the form $\Spf\bbk\dbb{t_1,\dots,t_n}$.
  The sheaf of algebras $(TB,\star )$ corresponds to a formal deformation of $(T_b B ,\star\vert_b )$ over $\bbk\dbb{t_1,\dots , t_n}$.
  By \cref{lemma:deformation-unital-associative-commutative-algebra}, we obtain a decomposition into sheaves of subalgebras $(TB ,\star ) = \bigoplus_{i\in I} (\cD_i , \star\vert_{\cD_i} )$ extending the decomposition of the fiber at $b$.
  Let $\varphi\colon B\rightarrow B$ be the change of coordinates provided by \cref{proposition:A-decomposition-theorem}(3) and let $\lbrace 1,\dots, n\rbrace = \coprod_{i\in I} J_i$ be the associated partition.
  Let $\cE_i \coloneqq \bigoplus_{j\in J_i} \cO_B \partial_{t_j}\subset TB$, its image under $d\varphi$ generates $\varphi^{\ast}\cD_i$.
  
  Since $\varphi$ is an automorphism of the formal neighborhood of a point, the differential $d\varphi\colon TB\rightarrow \varphi^{\ast} TB$ is an isomorphism.
  Then, we can produce another F-manifold structure on $B$, which we denote by $\varphi^{\ast} (\star )$, such that $\varphi\colon (B,\varphi^{\ast} (\star) )\rightarrow (B,\star )$ is an isomorphism of F-manifolds.
    Let $B_i\coloneqq \Spf \bbk\dbb{t_j ,\; j\in J_i}$, 
  let $\iota_i\colon B_i\rightarrow B$ be the canonical closed immersion.
  By construction the subsheaves $\cE_i$ are stable under $\varphi^{\ast} (\star )$.
  Thus the restriction $\varphi^{\ast} (\star )\vert_{\cE_i}$ is well-defined, and induces an F-manifold structure $\star_i$ on $B_i$, such that $\iota_i\colon (B_i ,\star_i )\rightarrow (B, \varphi^{\ast} (\star ))$ is a closed immersion of F-manifolds.
  Since $(B, \varphi^{\ast}( \star))\simeq  \prod_{i\in I} (B_i ,\star_i)$, we obtain (1), and (2) holds by construction.
\end{proof}

\subsubsection{Decomposition theorem for non-archimedean F-manifolds}
\label{subsubsec:decomposition-NA-F-manifold}

\begin{theorem} \label{thm:NA-decomposition-F-manifold}
    Let $B$ be a smooth $\bbk$-analytic space endowed with an F-manifold structure $\star$ with unit, and $b\in B$ a $\bbk$-rational point.
	Assume there exists a splitting as $\bbk$-algebras
	\begin{equation}\label{eq:NA-splitting-tangent-bundle}
	    T_b B = \bigoplus_{i\in I} A_i .
	\end{equation}
    Then there exist an admissible open neighborhood $U$ of $b$ and non-archimedean F-manifolds with unit $(U_i,\star_i)$ such that
	\begin{enumerate}
		\item $(U,\star\vert_U)$ is isomorphic to $\prod_{i\in I} (U_i,\star_i)$ as F-manifolds with unit,
        \item and the induced decomposition of $(TU,\star\vert_U )$ restricts to \eqref{eq:NA-splitting-tangent-bundle} at $b$. 
	\end{enumerate}
\end{theorem}

\begin{lemma} \label{lemma:NA-deformation-tangent-space-splits}
	Let $(B,\star)$ and $b$ be as in \cref{thm:NA-decomposition-F-manifold}.
	Assume there exists a splitting as $\bbk$-algebras
	\[T_b B = \bigoplus_{i\in I} A_i .\]
	Then there exists an admissible open neighborhood $U$ of $b$, and a decomposition into sheaves of subalgebras $(TU , \star\vert_U ) = \bigoplus_{i\in I} (\cD_i ,\star\vert_{\cD_i})$ extending the decomposition of $T_b B$.
\end{lemma}

\begin{proof}
    In this proof, we view the rigid $\bbk$-analytic spaces as Berkovich spaces.
    Then the base $B$ is Hausdorff.
    Let $X\coloneqq \Spec^{\an} TB$ be the relative analytic spectrum.
    Since $TB$ is a finite free $\cO_B$-module, the structural map $f\colon \Spec^{\an} TB\rightarrow B$ is proper as Berkovich spaces, in particular proper as topological spaces.

    The splitting of $T_b B$ produces a surjection $X_b = \Spec^{\an} T_b B \rightarrow \coprod_{i\in I} \Sp \bbk$.
    This implies that $X_b = \coprod_{i\in I} X_{b,i}$, where $X_{b,i}$ is the preimage of the $i$-th copy of $\Sp\bbk$.
    Let $U\subset B$ be the open neighborhood of $b$ given by \cref{lemma:topological-connected-components}, with $f^{-1} (U) = \coprod_{i\in I} W_i$.
    We obtain a map $X\times_B U \rightarrow \coprod_{i\in I} U$ extending $X_b\rightarrow \coprod_{i\in I} \Sp\bbk$ by mapping $W_i$ to the $i$-th copy of $U$ under $f$.
    This is equivalent to a map of sheaves of $\cO_U$-algebras $\cO_U^{\oplus I}\rightarrow TU$, producing the desired splitting.
\end{proof}

\begin{lemma} \label{lemma:topological-connected-components}
    Let $f\colon X\rightarrow B$ be a proper map between Haussdorff topological spaces. 
    Let $b\in B$, assume that $f^{-1} (b) = \coprod_{i\in I} X_{b,i}$ for a finite set $I$.
    Then, there exists an open neighborhood $U\subset B$ of $b$ such that
    $f^{-1} (U)$ is a disjoint union $\coprod_{i\in I} W_i$, and $W_i\cap f^{-1} (b) = X_{b,i}$.
    \end{lemma}

\begin{proof}
    Since $f$ is proper, the fiber $f^{-1} (b)$ is compact.
    Hence, each $X_{b,i}$ is compact.
    Since $X$ is Hausdorff, there exists open subsets $V_i\subset X$ containing $X_{b,i}$ with $V_i\cap V_j =\emptyset$ for $i\neq j$.
    Since $f$ is proper, it is closed, so $U \coloneqq f\big(\big(\bigcup_i V_i\big)^{\complement} \big)^{\complement}$ is open in $B$.
    Let $W_i\coloneqq V_i\cap f^{-1} (U)$.
    Since $f^{-1} (U) \cap (\bigcup_{i\in I} V_i)^{\complement} = \emptyset$, we have $f^{-1} (U) = \coprod_{i\in I} W_i$.
    By construction of $V_i$, we have $W_i\cap f^{-1} (b) = X_{b,i}$, completing the proof.
    \end{proof}

\begin{proof}[Proof of \cref{thm:NA-decomposition-F-manifold}]
By \cref{lemma:NA-deformation-tangent-space-splits}, there exists an admissible open neighborhood $U_1$ of $b$ and a decomposition into sheaves of subalgebras
\[(TU_1,\star\vert_{U_1} ) = \bigoplus_{i\in I} (\cD_i , \star \vert_{\cD_i}) ,\]
extending the decomposition of $T_b B$.
As in the proof of \cref{proposition:A-decomposition-theorem}, the F-identity implies that $\lbrace \cD_i\rbrace_{i\in I}$ define commuting integrable distributions on $TU_1$.

Up to shrinking $U_1$, we can choose a local basis of commuting vector fields $(Y_j)_{j\in J_i}$ of $\cD_i$ at $b$, and assemble them into a local commuting basis of $TU_1$ at $b$.
By \cref{thm:NA-Frobenius-theorem}, there exists admissible opens $U_2\subset U_1$ and $V\subset \Sp T_n$ and an isomorphism $\varphi\colon V\rightarrow U_2$ such that $d\varphi (\partial_{t_j} ) = \varphi^{\ast} (Y_j )$, where $\lbrace t_j\rbrace$ are the analytic coordinates on $V$ centered at $0$.
We conclude as in the formal case (see \cref{theorem:decomposition-F-manifolds}).

\end{proof}

\subsection{Decomposition theorems for maximal F-bundles}  \label{subsec:decomposition-theorems}

In this subsection, we establish the spectral decomposition theorem for maximal F-bundles (see \cref{thm:eigenvalue_decomposition,thm:NA-K-decomposition}).

We consider a maximal F-bundle $(\cH,\nabla )$ over a formal (resp.\ admissible open) neighborhood of a rational point $b$ in a smooth $\bbk$-variety (resp.\ $\bbk$-analytic space).
Let $K_b\coloneqq \nabla_{u^2\partial_u}\vert_{b,0}$.
Consider a decomposition of the fiber $\cH_{b,0} \simeq \bigoplus_{i\in I} H_{i}$ stable under $K_b$, such that the induced endomorphisms $K_b\vert_{H_i}$ and $K_b\vert_{H_j}$ have disjoint spectra for each $i\neq j$.
Our spectral theorem asserts that this produces a decomposition of $(\cH ,\nabla )$ into a product of maximal F-bundles $(\cH_i,\nabla_i )/B_i$ extending the decomposition of $\cH_{b,0}$.
We refer to \cref{sec:introduction} for an outline of the proof.

\subsubsection{The formal case} \label{subsubsec:decomposition-F-bundles}

\begin{lemma} \label{lemma:max-F-bundle-F-manifold}
    Let $B$ be a formal neighborhood of a rational point $b$ in a smooth $\bbk$-variety.
    Let $(\cH ,\nabla )/B$ be an F-bundle maximal at $b$, and let $h\colon B\rightarrow \cH\vert_{u=0}$ be a section of cyclic vectors (see \cref{definition:maximal-F-bundle}).
    The data $\lbrace (\cH ,\nabla) ,h\rbrace$ induce a formal F-manifold structure on $B$ with identity.
\end{lemma}

\begin{proof}
    Evaluation on the section of cyclic vectors $h$ provides an isomorphism $\eta\coloneqq \mu (\cdot )(h)\colon TB\rightarrow \cH\vert_{u=0}$, and a commutative and associative product on $TB$ as in (\ref{eq:product-structure-max-F-bundle}).
    Furthermore $e\coloneqq \eta^{-1}(h)$ is an identity for this product since for a vector field $X$ we have
    \[\eta (X\star e) = \mu (X)\circ \eta (\eta^{-1} (h)) =\mu (X) (h) = \eta (X) .\]
    We refer to \cite[Lemma 10]{David_T-structure-2d-F-manifold} for the proof of the F-identity, which is given there for (TE)-structures.  
\end{proof}

\begin{lemma} \label{lemma:image-adjoint-jordan-dec}
    Let $H$ be a $\bbk$-vector space of finite dimension, and $U\in \End_{\bbk} (H)$.
    Assume we have a decomposition 
    $H = \bigoplus_{i\in I} H_i$ stable under $U$, such that the induced endomorphisms $U\vert_{H_i}$ and $U\vert_{H_j}$ have disjoint spectra for $i\neq j$.
    Then 
    \begin{enumerate}
        \item $\ker [\cdot , U]\subset \bigoplus_{i\in I} \End_{\bbk} (H_i)$, and
        \item $[\cdot , U]$ restricts to an isomorphism of $\bigoplus_{i\neq j} \Hom_{\bbk} (H_j,H_i)$ onto itself.
    \end{enumerate}
\end{lemma}

\begin{proof}
    Let $\bbk^a$ denote an algebraic closure of $\bbk$.
    The disjoint spectra assumption implies that $H_i\otimes_{\bbk} \bbk^a$ is a direct sum of generalized eigenspaces for $U$.
    In particular, any endomorphism that commutes with $U$ preserves this decomposition, proving (1).
    It follows that the restriction $[\cdot , U]: \bigoplus_{i\neq j} \Hom_{\bbk} (H_j ,H_i)\rightarrow \bigoplus_{i\neq j} \Hom_{\bbk} (H_j ,H_i)$ is injective, hence an isomorphism by comparing dimensions, proving (2).
\end{proof}

\begin{proposition} \label{proposition:split-u-direction}
Let $(\cH ,\nabla )$ be an F-bundle over a formal neighborhood $B=\Spf\bbk\dbb{t_1,\dots, t_n}$ of $b=0$ in an affine space. 
Let $K = \nabla_{u^2\partial_u}\vert_{u=0}$ and $\cH_{b,0}=\bigoplus_{i\in I}H_i$ a decomposition stable under $K_b$ such that the induced endomorphisms on $H_i$ have disjoint spectra.

Let $\cH|_{u=0} = \bigoplus_{i\in I} \cH_{i,0}$ be a decomposition extending the decomposition of $\cH_{b,0}$, and stable under $K$.
Then it extends to a decomposition $\cH = \bigoplus_{i\in I}\cH_i$ such that $u^2\nabla_{\partial_u} (\cH_i)\subset \cH_{i}$.
\end{proposition}

\begin{proof}
Write $t = (t_1,\dots,t_n)$ and $H=\cH_{b,0}$.
Choose a trivialization $\Phi\colon \cH\simeq H\times \Spf \bbk\dbb{t,u}$ such that $ \cH_{i,0}\simeq H_i\times \Spf \bbk\dbb{t} $.
Write the connection in the $u$ direction as 
\[\nabla_{\partial_u} = \frac{\partial}{\partial u}+\frac{U(t)}{u^2},\]
where $U(t,u)=\sum_{k\geq 0}U_k(t)u^k$ for $U_k(t)\in \End(H)\dbb{t}$. 
By assumption, $U_0(t)\in \bigoplus_{i\in I}\End(H_i)\dbb{t}$. 

We will construct an automorphism $P(t,u)\in \Aut (H\times\Spf\bbk\dbb{t,u} )$ with $P(t,0) = \id$, such that $P(t,0) = \id$ and $P^{-1}UP + P^{-1}\frac{\partial P}{\partial u}\in \bigoplus_{i\in I} \End (H_i)\dbb{t,u}$.
Given such a $P(t,u)$, defining $\cH_i$ to be the constant extension of $H_i$ in the trivialization $P^{-1}\circ\Phi$ provides the desired splitting of $\cH$.

For $m\geq 1$,  
$T_m(t)\in \End(H)\dbb{t}$ and $P(t,u) = \id +u^m T_m(t)\in \GL(H)\dbb{t,u}$, write $(P^{\ast}\nabla )_{\partial_u} = \frac{\partial}{\partial u} + u^{-2}\tU(t,u).$
We have 
\begin{equation}\label{eq: difference-in-U}
    \tU (t,u) - U(t,u) = \sum_{k\geq 0}(-1)^{k+1}u^{m(k+1)}T_m(t)^k[T_m(t),U]+\sum_{k\geq 0}(-1)^kmu^{m(k+1)+1}T_m(t)^{k+1} .
\end{equation}
Note that the right-hand side of  (\ref{eq: difference-in-U}) has degree at least degree $m$ in $u$, and the coefficient of $u^m$ is $-[T_m(t),U_0(t)]$.

Let $<$ denote the degree lexicographic order on $\bbN^n$.
For $v=(v_1,\cdots,v_n)$, we write $t^v = t_1^{v_1}\cdots t_n^{v_n}$.
Now for $T_m(t) = t^v T_{m,v}$ with $T_{m,v}\in \End(H)$, we have $-[T_m(t),U_0(t)] = -[T_{m,v},U_0(0)]t^v + T't^{v'}$ where $T'\in \End(H)\dbb{t}$ and $v< v'$.
Write $U_k(t) = \sum_{w\in \bbN^n}U_{k,w}t^w$. 
By \cref{lemma:image-adjoint-jordan-dec}, we can choose $T_{m,v}$ such that $U_{m,v}-[T_{m,v},U_0(0)]\in \bigoplus_{i\in I} \End(H_i)$. 
By induction on $v\in \bbN^n$ using the lexicographic order on $\bbN^n$, we can assume $U_m(t) \in \bigoplus_{i\in I}\End(H_i)\dbb{t}$.
By induction on $m\geq 1$, we can further make $\tU(t,u)\in \bigoplus_{i\in I}\End(H_i)\dbb{t,u}$, completing the proof.
\end{proof}

\begin{lemma} \label{lemma:commutator-splits-implies-splits}
        Write $t = ( t_1,\dots, t_n )$.
    Let $\tH$ be a finite free $\bbk\dbb{t}$-module, and $U(t)\in \End (\tH )$.
    Let $\tH = \bigoplus_{i\in I} \tH_i$ be a decomposition of $\tH$ stable under $U(t)$.
    Assume that for $i\neq j$, the induced endomorphisms $U(t)\vert_{\tH_i}$ and $U(t)\vert_{\tH_j}$ have disjoint spectra.
    Let $X(t)\in \End (\tH )$ such that $[X(t),U(t)]\in \bigoplus_{i\in I} \End (\tH_i )$, then $X(t)\in \bigoplus_{i\in I} \End (\tH_i )$.
\end{lemma}

\begin{proof}
    Let $R \coloneqq \bbk\dbb{t_1,\dots, t_n}$ and $K\coloneqq \Frac (R)$ its fraction field. 
    Working over $K$, \cref{lemma:image-adjoint-jordan-dec} implies that $\ker [\cdot , U]\subset \bigoplus_{i\in I} \End_R (\tH_i )$.

    We have a decomposition $\End_R (\tH)= \bigoplus_{i,j\in I} \Hom_{R} (\tH_j, \tH_i)$.
    Let $X_{i,j}$ denote the components of $X$ with respect to this splitting. 
    Let $Y \coloneqq \sum_{i\neq j} X_{i,j}$ denote the off-diagonal part of $X$. 
    We will prove that $Y=0$.
    Since $U\in \bigoplus_{i\in I} \End (\tH_i )$, the commutator $[Y,U]$ has vanishing diagonal, i.e.\ it lies in $\bigoplus_{i\neq j} \Hom_R (\tH_j,\tH_i )$.
    Furthermore, using the assumption, we see that $[Y,U] = [X,U] - \sum_{i\in I} [X_{i,i} , U]$ is block diagonal. 
    It follows that $[Y,U] =0$, hence $Y\in \bigoplus_{i\in I} \End_R (\tH_i)$.
    By definition, $Y$ is off-diagonal, so $Y=0$, proving the lemma.
\end{proof}

The following proposition implies that the decomposition in \cref{proposition:split-u-direction} induces a decomposition of F-bundle $(\cH,\nabla)\simeq \bigoplus_{i\in I}(\cH_i,\nabla_i)$ over $B$, where $\nabla_i$ is the restriction of $\nabla$ to $\cH_i$.

\begin{proposition} \label{proposition:split-t-directions}
In the setting of \cref{proposition:split-u-direction}, we have $u\nabla_{\xi}(\cH_i)\subset \cH_i$ for any vector field $\xi$ on $B$. 
\end{proposition}

\begin{proof}
Write $t=(t_1,\cdots,t_n)$.    Let $H \coloneqq \cH|_{b,0}$, and $H = \bigoplus_{i\in I} H_i$ the splitting induced by the decomposition of $\cH$.
    Fix a trivialization $\cH\simeq H\times\Spf\bbk-\dbb{t,u}$ such that $\cH_i\simeq H_i\times\Spf\bbk\dbb{t,u}$, and write
    \[\nabla = d + u^{-1}\sum_{1\leq i\leq n} T_i (t,u) dt_i + u^{-2} U(t,u)du ,\]
    with $U(t,u) = \sum_{k\geq 0} U_k (t) u^k$ and $T_i(t,u) = \sum_{k\geq 0} T_{i,k} (t)u^k$.
    By assumption, we have $U(t,u)\in \bigoplus_{i\in I} \End (H_i)\dbb{t,u}$.
    In particular, $U_0 (t)$ induces endomorphisms in $\End (H_i)\dbb{t}$ for all $i\in I$, and the assumption on the decomposition at $t=u=0$ implies that those have disjoint spectra.

    Fix $i\in \lbrace 1,\dots, n\rbrace$.
    The flatness equation $[\nabla_{\partial_u} ,\nabla_{\partial_{t_i}} ] = 0$ reads
    \[\frac{\partial (u^{-1} T_i )}{\partial u} - \frac{\partial (u^{-2} U)}{\partial t_i} = u^{-3} [T_i, U] .\]
    Splitting this equation according to powers of $u$ gives $[T_{i,0} , U_{0} ] =0$, and for $k\geq 1$:
    \begin{equation}\label{eq:flatness-u-t-block-diagonal}
        [T_{i,k} , U_0] = (k-2)T_{i,k-1} - \frac{\partial U_{k-1}}{\partial t_i} - \sum_{\substack{k_1+k_2=k \\ k_1 <k}} [T_{k_1} , U_{k_2} ] .
    \end{equation}
    We prove by induction on $k\geq 0$ that $T_{i,k}$ is block diagonal, i.e.\ $T_{i,k}\in \bigoplus_{i\in I} \End (H_i)\dbb{t}$.
    The base case $k=0$ follows from \cref{lemma:commutator-splits-implies-splits}, because $T_{i,0}$ commutes with $U_0 (t)$.
    Now, let $k\geq 1$ and assume $T_{i,\ell} (t)$ is block diagonal for $\ell <k$.
    Since each $U_{\ell} (t)$ is assumed block diagonal, the right-hand side of \eqref{eq:flatness-u-t-block-diagonal} is block diagonal. 
    Applying \cref{lemma:commutator-splits-implies-splits}, we obtain that $T_{i,k} (t)$ is also block diagonal, completing the proof.
\end{proof}

It remains to show that the above decomposition $(\cH,\nabla)\simeq \bigoplus_{i\in I}(\cH_i,\nabla_i)$ is compatible with the decomposition of the base. 

\begin{lemma}\label{lemma:pullback-F-bundle}
  Let $B \simeq B_1\times B_2$ be a formal neighborhood of $b=0$ in a product of affine spaces, and $(\cH ,\nabla ) /B$ be an F-bundle over $B$. 
  Assume that $\nabla_{u\xi}\vert_{u=0} = 0$ for all vector fields $\xi$ in the directions of $B_2$.
  Then there exists an F-bundle $(\cH_1,\nabla_1)/B_1$ such that $\pr_1^*(\cH_1,\nabla_1)\simeq  (\cH,\nabla)$, where $\pr_1$ is the projection $B\simeq B_1\times B_2\rightarrow B_1$.
\end{lemma}
\begin{proof}
  For $i=1,2$, let $t_i = ( t_{i,j} , \, 1\leq j\leq n_i )$ denote coordinates on $B_i$.
  Let $\cH_1\coloneqq \cH|_{B_1\times \{0\}\times \Spf \bbk\dbb{u}}$.
  By assumption, $\nabla$ has no pole at $u=0$ in the directions of $B_2$.
  Since $\nabla$ is flat, given any trivialization of $\cH_1$ we can extend it uniquely by $\nabla$ to a trivialization of $\cH$ over $B_1\times B_2 \times \Spf \bbk\dbb{u}$. 
  This defines an isomorphism $\pr_1^*\cH_1\simeq \cH$, and in this trivialization we have
  \[\nabla = d + u^{-1}\sum_{1\leq j\leq n_1} T_{1,j} (t_1,t_2,u)dt_{1,j} + u^{-2} U(t_1,t_2,u)du .\]
  Since $\nabla$ is flat, we have for all $1\leq j\leq n_1$ and $1\leq k\leq n_2$
  \[
    \frac{\partial (u^{-1} T_{1,j})}{\partial t_{2,k}} = 0 ,\quad
    \frac{\partial (u^{-2} U)}{\partial t_{2,k}} = 0 .
  \]
    Hence, the connection matrices in the directions of $B_1$ and the $u$-direction are independent of $t_2$.
  This means that the connection is equal to the pullback of a connection on $B_1\times \Spf\bbk\dbb{u}$, completing the proof.
\end{proof}

\begin{theorem}[Spectral decomposition theorem] \label{thm:eigenvalue_decomposition}
  Let $B$ be a formal neighborhood of a rational point $b$ in a smooth $\bbk$-variety, and $(\cH, \nabla)$ an F-bundle over $B$ maximal at $b$.
  Write $K_b = \nabla_{u^2\partial_u}|_{b,0}$.
  Assume that we have a decomposition $\cH_{b,0} \simeq \bigoplus_{i \in I} H_i$ stable under $K_b$, and that for any $i \neq j \in I$, the spectra of $K_b|_{H_i}$ and $K_b|_{H_j}$ are disjoint.
  Then $(\cH ,\nabla )/B$ decomposes into a product of maximal F-bundles $(\cH_i ,\nabla_i )/B_i$ extending the decomposition of $\cH|_{b,0}$.
\end{theorem}
\begin{proof}
As in the proof of \cref{theorem:decomposition-F-manifolds}, we may assume the base $B$ has the form $\Spf\bbk\dbb{t_1,\cdots,t_n}$.
Let $h:B\rightarrow \cH \vert_{u=0}$ be a section of cyclic vectors, providing an isomorphism
\[\eta\coloneqq u\nabla |_{u=0}(h):TB\xrightarrow{\sim} \cH|_{u=0}.\]
This induces an F-manifold structure $(B,\star)$ on $B$ by \cref{lemma:max-F-bundle-F-manifold}. In particular, we have a decomposition $T_bB=\bigoplus_{i\in I} E_i$ with $E_i=\eta_b^{-1}(H_i)$.  
Since the spectra of $K_b|_{H_i}$ and $K_b|_{H_j}$ are disjoint, up to extending the base field, each $H_i$ is a direct sum of generalized eigenspaces for $K_b$.
Since $\nabla$ is flat, it follows that $T_bB=\bigoplus_{i\in I} E_i$ is a splitting of $\bbk$-algebra. 
By \cref{theorem:decomposition-F-manifolds}, we obtain a decomposition of F-manifold $B\simeq \prod_{i\in I} (B_i,\star_i)$, extending the decomposition at $T_b B$. 
This induces a decomposition of $TB = \bigoplus_{i\in I} \cE_i$ as $\cO_B$-algebras.
We refer to sections of $\cE_i$ as being in the directions of $B_i$.

Under $\eta$, we obtain a decomposition $\cH|_{u=0} \simeq \bigoplus_{i\in I}\cH_{i,0}$.
Since the action of $K$ corresponds to multiplication by the Euler vector field, this decomposition is stable under $K$,
and extends the decomposition of $\cH_{b,0}\simeq \bigoplus_{i\in I}H_i$.
By Propositions \ref{proposition:split-u-direction} and \ref{proposition:split-t-directions}, this further extends to a decomposition $(\cH,\nabla) \simeq \bigoplus_{i\in I}(\cH_i,\nabla_i)$. 

For each $i\in I$ and $\xi$ not in the directions of $B_i$, the action of $(\nabla_i)_{u\xi}|_{u=0}$ on $\cH_{i,0}$ under $\eta$  is the restriction of $\xi\star$ to the subalgebra $\cE_i$, hence it vanishes.
Then by \cref{lemma:pullback-F-bundle}, $(\cH_i,\nabla_i)/B$ isomorphic to a pullback of F-bundle from $B_i$, which we also denote as $(\cH_i,\nabla_i)/B_i$.
We thus have a decomposition of F-bundle
\[(\cH ,\nabla) \simeq \bigoplus_{i\in I} \pr_i^{\ast} (\cH_i,\nabla_i ) ,\]
where $\pr_i\colon B\simeq \prod_{j\in I} B_j\rightarrow B_i$ is the projection to the $i$-th component.

It remains to check that each F-bundle in the decomposition is maximal.
Let $j_i\colon B_i\hookrightarrow B$ be the canonical closed immersion, and $h_i\coloneqq j_i^{\ast} h$.
We claim that $h_i$ is a section of cyclic vectors for $(\cH_i,\nabla_i ) /B_i$, i.e.\ the map $\eta_i\colon \xi\mapsto (\nabla_{i})_{u\xi}\vert_{u=0} (h_i)$ is an isomorphism $TB_i\xrightarrow{\sim} \cH_i\vert_{u=0}$.
Since $B_i$ is the formal neighborhood of a point in an affine space, it is enough to check that the stalk of $\eta_i$ at the closed point $b_i$ of $B_i$ is an isomorphism.
This stalk is the composition of the isomorphisms
\[T_{b_i} B_i\longrightarrow E_i \xrightarrow{\eta_b\vert_{E_i}} H_i , \]
hence it is an isomorphism, completing the proof.
\end{proof}

\begin{example}[rank $1$ maximal F-bundle] \label{example-rank-1-maximal}
  Let $B = \Spf \bbk\dbb{t}$ and $b=0 \in B$.
  Let $(\cH ,\nabla )/B$ be an F-bundle, maximal at $b$.
  Fixing a trivialization of $\cH$, we write the connection as $\nabla =d +  u^{-2} p (t,u)du + u^{-1} q(t,u) dt$ .
  Flatness of $\nabla$ reduces to the equation $\frac{\partial (u^{-2}p)}{\partial t} = \frac{\partial (u^{-1}q)}{\partial u}$.
  Solutions are parameterized by functions $\psi(t,u)\in\bbk\dbb{t,u}$ by the rule
  \[ p = u\frac{\partial \psi}{\partial u} - \psi,\quad q = \frac{\partial \psi}{\partial t}  .\]
  The F-bundle is maximal at $t=0$ if and only if $q (0,0)\neq 0$ or, in terms of $\psi$, $\frac{\partial \psi}{\partial t}(0,0)\neq 0$.
\end{example}

\begin{example}[simple eigenvalues] \label{example:simple-K-operator}
    Let $B$ be the formal neighborhood of $b=0$ in an $n$-dimensional affine space.
    Let $(\cH ,\nabla )/B$ be an F-bundle, maximal at $b$.
    Assume that $K_b=u^2\nabla_{\partial_u}\vert_{b,u=0}$ has simple eigenvalues.
    Then $(\cH ,\nabla )/B$ is isomorphic to a product of rank $1$ maximal F-bundles.

    Concretely, there exists a change of coordinates $f\colon  \prod_{1\leq i\leq n}\Spf \bbk\dbb{t_i}\xrightarrow{\sim} B$, and a trivialization of $f^{\ast}(\cH,\nabla )$ in which the connection takes the form
    \[f^{\ast}\nabla = d + u^{-1}\begin{pmatrix}
      \frac{\partial \psi_1}{\partial t_1} dt_1 & & 0 \\
       & \ddots & \\
       0 & & \frac{\partial \psi_n}{\partial t_n} dt_n
    \end{pmatrix}
    + u^{-2}\begin{pmatrix}
       u\frac{\partial \psi_1}{\partial u} - \psi_1  & & 0 \\
        & \ddots & \\
      0   & & u\frac{\partial \psi_n}{\partial u} - \psi_n
    \end{pmatrix}du ,\]
    with $\psi_i\in \bbk\dbb{t_i,u}$ functions such that $-\psi_i (0,0)$ is an eigenvalue of $K_b$, and $\frac{\partial \psi_i}{\partial t_i} (0,0)\neq 0$ (see \cref{example-rank-1-maximal}).

    When $K_b$ has simple eigenvalues, the change of coordinates is obtained by integrating a basis of sections of eigenvectors for the connection in the $u$-direction.
\end{example}

\subsubsection{The non-archimedean case}

Next, we prove the spectral decomposition theorem in the non-archimedean case. The proof builds on the formal case, but an additional challenge lies in bounding the norms of the coefficients of the gauge transform and establishing non-archimedean convergence. We achieve these bounds through a detailed analysis of the recursive relations of the coefficients, see \cref{proposition:split-u-direction-NA}.

\begin{lemma} \label{lemma:NA-max-F-bundle-F-manifold}
    Let $B$ be an admissible open neighborhood of a rational point $b$ in a smooth $\bbk$-analytic space.
    Let $(\cH ,\nabla )/B$ be a non-archimedean F-bundle maximal at $b$.
    Then there exists an admissible open neighborhood $U\subset B$ of $b$ such that $(\cH ,\nabla )$ admits a section of cyclic vectors, and the data $\lbrace (\cH,\nabla ) , h\rbrace$ induces a non-archimedean F-manifold structure on $U$ with identity.
\end{lemma}

\begin{proof}
    Being maximal is an open condition, so there exists an admissible open neighborhood $U\subset B$ of $b$ over which a section of cyclic vector $h$ exists.
    The proof is then identical to the formal case, and relies on explicit computations in local analytic coordinates centered at $b$.
\end{proof}

\begin{proposition} \label{proposition:split-u-direction-NA}
Let $(\cH ,\nabla )$ be an F-bundle over $B=\Sp\bbk\langle t_1,\dots, t_n\rangle$, and let $b=0\in B$.
Let $K = \nabla_{u^2\partial_u}\vert_{u=0}$ and $\cH_{b,0}=\bigoplus_{i\in I}H_i$ a decomposition stable under $K_b$ such that the induced endomorphisms on $H_i$ have disjoint spectra.

Let $\cH|_{u=0} = \bigoplus_{i\in I} \cH_{i,0}$ be a decomposition extending the decomposition of $\cH_{b,0}$, and stable under $K$.
Then, there exists an admissible open neighborhood $U\subset B$ of $b$ and a decomposition $\cH\vert_U = \bigoplus_{i\in I}\cH_i$ such that $\cH_i\vert_{u=0} = \cH_{i,0}\vert_U$ and $u^2\nabla_{\partial_u} (\cH_i)\subset \cH_{i}$.
\end{proposition}

\begin{proof}
    We keep the setting and notations of \cref{proposition:split-u-direction}, in particular $H\coloneqq \cH_{b,0}$.
    Let $\leq$ denote the degree lexicographic order on $\bbN^n$.
    We denote by $\tau (v)$ the direct successor of $v\in\bbN^{n}$ for this order.
  The gauge transformation $P$ constructed in the formal case is an ordered product
  \[P = \prod_{m\geq 1} P_m , \quad P_m  = \prod_{v\in \bbN^{n}} P_{m,v} ,\]
  where $P_{m,v} = \id + u^m t^v T_{m,v}$ and $T_{m,v}\in \End (H)$.
  Let $\phi$ denote the inverse of the restriction of $[\cdot , U_0(0)]$ to $\bigoplus_{i\neq j}\Hom (H_j,H_i)$.
  The gauge transformations $P_{m,v}$ are constructed inductively, and characterized by the following relations:
  \begin{align}\label{eq:recursion-split-u-direction-1}
    T_{m,v} &= \phi ( \text{off-diagonal part of the term } u^mt^v \text{ in } \tU_{m,v} ) , \\
    \label{eq:recursion-split-u-direction-2}
    \tU_{m,\tau(v)} &= P_{m,v}^{-1} \tU_{m,v} P_{m,v} + u^2 P_{m,v}^{-1} \frac{\partial P_{m,v}}{\partial u}, \\
    \label{eq:recursion-split-u-direction-3}
    \tU_{m+1,0} &= P_m^{-1} \tU_{m,0} P_m + u^2 P_m^{-1} \frac{\partial P_m}{\partial u} ,
  \end{align}
  and $\tU_{1,0} = U(t,u)$ is the initial connection matrix.
    For an element $M(t,u) = \sum_{m,v} M_{m,v}u^m t^v\in\End (H)\dbb{t,u}$ and $\delta,\varepsilon >0$, we let 
  \[\vert M(t,u)\vert_{\delta ,\varepsilon}  \coloneqq \sup_{m\in\bbN,v\in\bbN^n} \vert M_{m,v}\vert \delta^m \varepsilon^{\vert v\vert} . \] 
  We denote by $\bbD (\delta ,\varepsilon )$ the polydisk $\lbrace \vert u\vert \leq \delta, \vert t\vert\leq \varepsilon\rbrace$.

    Since the gauge transformations restrict to $\id$ at $u=0$, all the matrices $\tU_{m,v} (t,u)$ have the same constant term.
    We denote this common value by $U_0$, and
    set $\tV_{m,v} (t,u) \coloneqq \tU_{m,v} (t,u) - U_0$.
    Fix $\delta\leq 1$ and $\varepsilon\leq 1$ such that $\delta \vert \phi\vert \leq 1$ and $\vert \phi\vert\vert \tV_{1,0}\vert_{\delta , \varepsilon} < 1$. 
    This is possible, since $\tV_{1,0} (0,0) = 0$.
  
  We prove by a double induction on $m$ and $v$ the inequalities
  \begin{align} \label{eq:induction-convergence-gauge-transform-split-u}
      \vert u^m t^v T_{m,v}\vert_{\delta,\varepsilon} \leq \vert \phi\vert\vert \tV_{m,v}\vert_{\delta,\varepsilon}\leq  \vert \phi\vert \vert \tV_{m,0}\vert_{\delta,\varepsilon} \leq  \vert \phi\vert\vert \tV_{1,0}\vert_{\delta,\varepsilon} < 1 .
  \end{align}
  We use the lexicographic order on the product $\bbN_{>0}\times\bbN^n$, i.e.\ $(m,v)< (m',v')$ if and only if $m<m'$ or $m=m'$ and $v<v'$.
  For $m=1$, $v=0$, the inequalities follow from \eqref{eq:recursion-split-u-direction-1} and the choice of $(\delta,\varepsilon )$.
  Now fix $(m,v)\in\bbN_{>0}\times\bbN^n$ with $(m,v)> (1,0)$, and assume all the inequalities proved for $(m',v')< (m,v)$.
  Equation \eqref{eq:recursion-split-u-direction-1} gives 
  \[\vert u^mt^v T_{m,v}\vert_{\delta,\varepsilon} \leq \vert\phi\vert \vert \tV_{m,v}\vert_{\delta,\varepsilon} .\]
  We now bound $\vert \tV_{m,v}\vert_{\delta,\varepsilon}$.
  If $v>0$, then we can write $v=\tau (w)$ for some $w\geq 0$. 
  The difference between $\tV_{m,\tau (w)}$ and $\tV_{m,w}$ is given by \eqref{eq: difference-in-U}:
  \begin{align*}
      \tV_{m,\tau(w)} - \tV_{m,w} &= \tU_{m,\tau(w)} - \tU_{m,w} \\
      &= \sum_{k\geq 0} (-1)^{k+1} (u^m t^w)^{k+1} T_{m,w}^k [T_{m,w} , \tV_{m,w}] \\
      &+ \sum_{k\geq 0} (-1)^{k+1} (u^m t^w)^{k+1} T_{m,w}^k [T_{m,w} , U_0] \\
      &+ \sum_{k\geq 0} (-1)^k u(u^mt^w)^{k+1} T_{m,w}^{k+1} .
  \end{align*}
  Let us bound each term on the right hand side. 
  Since $\vert u^mt^wT_{m,w}\vert_{\delta,\varepsilon}  <1$, we have for all $k\geq 0$
  \[\vert (u^mt^w)^{k+1}T_{m,w}^k [T_{m,w} ,\tV_{m,w}]\vert_{\delta,\varepsilon} \leq \vert u^mt^wT_{m,w}\vert_{\delta,\varepsilon}^{k+1} \vert \tV_{m,w}\vert_{\delta,\varepsilon} < \vert \tV_{m,w}\vert_{\delta,\varepsilon} . \]
  By the definition of $\phi$ and \eqref{eq:recursion-split-u-direction-1}, we have
  \begin{align}\label{eq:estimate-commutator-U_0}
      \vert u^mt^w[T_{m,w} , U_0] \vert_{\delta,\varepsilon} = \vert u^mt^w \phi^{-1} (T_{m,w})\vert_{\delta,\varepsilon}\leq \vert \tV_{m,w}\vert_{\delta,\varepsilon} .
  \end{align}
  We can then bound the second term for all $k\geq 0$
  \[\vert (u^m t^w)^{k+1} T_{m,w}^k [T_{m,w} , U_0]\vert_{\delta,\varepsilon} \leq \vert u^mt^wT_{m,w}\vert_{\delta,\varepsilon}^k \vert [T_{m,w} , U_0]\vert_{\delta,\varepsilon} < \vert \tV_{m,w}\vert_{\delta,\varepsilon} .\]
  For the third term, using the induction hypothesis and $\delta\vert\phi\vert\leq 1$, we obtain for all $k\geq 0$
  \[\vert u^{m(k+1)+1}t^{w(k+1)} T_{m,w}^{k+1} \vert_{\delta,\varepsilon}\leq \delta (\vert \phi\vert \vert \tV_{m,w}\vert_{\delta,\varepsilon})^{k+1}\leq \delta\vert \phi\vert \vert \tV_{m,w}\vert_{\delta,\varepsilon} \leq \vert \tV_{m,w}\vert_{\delta,\varepsilon} ,\]
  where we used $\vert\phi\vert\vert\tV_{m,w}\vert_{\delta,\varepsilon}\leq 1$ in the second inequality.
  Using those bounds, we obtain the inequalities
  \begin{align*}
      \vert \tV_{m,\tau (w)}\vert_{\delta,\varepsilon} \leq \max (\vert \tV_{m,w} \vert_{\delta,\varepsilon}, \vert \tV_{m,\tau(w)} - \tV_{m,w}\vert_{\delta,\varepsilon} )
      \leq \vert \tV_{m,w} \vert_{\delta,\varepsilon} \leq \vert \tV_{m,0}\vert_{\delta,\varepsilon},
  \end{align*}
  proving the inductive step when $v >0$.
  If $v=0$, then necessarily $m> 1$ and we can write $m=m'+1$. 
  We compare $\tV_{m'+1,0}$ to $\tV_{m',0}$.
  To do so, write $P_{m'} = \id + u^{m'} R_{m'}(t,u)$.
  Similarly to the previous case, using \eqref{eq: difference-in-U} we obtain
  \begin{align*}
      \tV_{m'+1,0} - \tV_{m',0} &= \tU_{m'+1,0} - \tU_{m',0} \\
      &= \sum_{k\geq 0} (-1)^{k+1} u^{m'(k+1)} R_{m'}^k [R_{m'} , \tV_{m,w}] \\
      &+ \sum_{k\geq 0} (-1)^{k+1} u^{m'(k+1)} R_{m'}^k [R_{m'}, U_0] \\
      &+ \sum_{k\geq 0} (-1)^k u^{m'(k+1)+1} R_{m'}^{k+1} ,
  \end{align*}
  and we will use the induction hypothesis to bound each term. 
  Since $\vert u^{m'}t^v T_{m',v}\vert_{\delta,\varepsilon} < 1$ for all $v\in\bbN^n$, we have $\vert u^{m'}R_{m'}\vert_{\delta,\varepsilon} < 1$.
  In particular, similarly to the case $v>0$, the first term is bounded by $\vert \tV_{m',0}\vert_{\delta,\varepsilon}$.
  To handle the other terms, we use the explicit formula
  \[u^{m'}R_{m'} = \sum_{\substack{k\geq 1 \\ w\in\bbN^n}\\} u^{km'}t^w\sum_{\substack{w_1+\cdots + w_k = w \\ w_1 > \cdots > w_k}}  T_{m',w_1}\cdots T_{m',w_k}  . \]
  Using this formula, we obtain
  \begin{align*}
      \vert [u^{m'}R_{m'} , U_0]\vert_{\delta,\varepsilon} &\leq \max_{\substack{k\geq 1, w\in\bbN^n \\ w_1+\cdots + w_k=w}} \vert u^{km'}t^w [T_{m',w_1} \cdots T_{m',w_k} , U_0] \vert_{\delta,\varepsilon} \\
      &\leq \max_{\substack{k\geq 1, w\in\bbN^n \\ w_1+\cdots + w_k=w}}\max_{1\leq i\leq k}\Bigg( \prod_{\substack{1\leq j\leq k \\j\neq i}} \vert u^{m'}t^{w_j} T_{m',w_j} \vert_{\delta,\varepsilon} \times \vert u^{m'}t^{w_i} [T_{m',w_i} , U_0] \vert_{\delta,\varepsilon} \Bigg) \\
      &\leq \max_{\substack{k\geq 1, w\in\bbN^n \\ w_1+\cdots + w_k=w}}\max_{1\leq i\leq k}\Bigg( \prod_{\substack{1\leq j\leq k \\j\neq i}} \vert\phi\vert\vert \vert \tV_{m',w_j}\vert_{\delta,\varepsilon} \times \vert \tV_{m',w_i}\vert_{\delta,\varepsilon} \Bigg) \\
      &\leq \max_{\substack{k\geq 1, w\in\bbN^n \\ w_1+\cdots + w_k=w}} \vert \tV_{m',w_i}\vert_{\delta,\varepsilon}  \leq \vert \tV_{m',0}\vert_{\delta,\varepsilon} .
  \end{align*}
  For the second inequality, we used the formula for the commutator of a product.
    The third inequality follows from the induction hypothesis at step $(m',w_j)$, and the inequality \eqref{eq:estimate-commutator-U_0} applied to $T_{m',w_i}$.
  The fourth and fifth inequalities follow from the induction hypothesis.
  Then, similarly to the case $v>0$, we obtain that the second term is bounded by $\vert \tV_{m',0}\vert_{\delta,\varepsilon}$.
  We now consider the third term.
  For $k\geq 1$ and $w_1 ,\cdots ,w_k\in\bbN^n$, since $\vert \phi\vert \vert \tV_{m',0}\vert_{\delta,\varepsilon}\leq 1$ by the induction hypothesis, we have 
  \begin{align*}
      \vert u^{km'} t^{w_1+\cdots + w_k} T_{m',w_1}\cdots T_{m',w_k}\vert_{\delta,\varepsilon} \leq (\vert \phi\vert \vert \tV_{m',0}\vert_{\delta,\varepsilon} )^k
      \leq \vert \phi\vert \vert \tV_{m',0}\vert_{\delta,\varepsilon}.
  \end{align*}
  In particular, we have the better bound $\vert u^{m'}R_{m'}\vert_{\delta,\varepsilon}\leq \vert \phi\vert \vert \tV_{m',0}\vert_{\delta,\varepsilon}$.
  Since $\vert \phi\vert \vert \tV_{m',0}\vert_{\delta,\varepsilon}\leq 1$, we obtain the bound on the third term for all $k\geq 0$
  \begin{align*}
      \vert  u^{m'(k+1)+1} R_{m'}^{k+1}\vert_{\delta,\varepsilon} \leq \delta (\vert \phi\vert \vert \tV_{m',0}\vert_{\delta,\varepsilon})^{k+1}
      \leq \delta \vert \phi\vert \vert \tV_{m',0}\vert_{\delta,\varepsilon} \leq \vert \tV_{m',0}\vert_{\delta,\varepsilon} .
  \end{align*}
  Similarly to the case $v>0$, we deduce
  \begin{align*}
      \vert \tV_{m'+1,0}\vert_{\delta,\varepsilon} \leq \max (\vert \tV_{m'+1,0} \vert_{\delta,\varepsilon}, \vert \tV_{m'+1,0} - \tV_{m',0}\vert_{\delta,\varepsilon} )
      \leq \vert \tV_{m',0} \vert_{\delta,\varepsilon} \leq \vert\tV_{1,0}\vert_{\delta,\varepsilon} ,
  \end{align*}
  concluding the induction.

  Now, \eqref{eq:induction-convergence-gauge-transform-split-u} implies that the product defining $P$ is convergent on the polydisk $\bbD (\delta,\varepsilon )$, that $P^{-1}$ is also convergent on $\bbD (\delta,\varepsilon )$, and that $\vert P\vert_{\delta,\varepsilon} = \vert P^{-1}\vert_{\delta,\varepsilon}  =1$.
  In particular, the decomposition constructed in the formal case extends to an admissible open neighborhood of $(b,0)$, completing the proof.
\end{proof}

\begin{theorem}[Non-archimedean spectral decomposition theorem] \label{thm:NA-K-decomposition}
    Let $B$ be a $\bbk$-analytic space, $b\in B$ a smooth $\bbk$-rational point, 
    and $(\cH ,\nabla )$ an F-bundle over $B$ maximal at $b$.
    Write $K_b = \nabla_{u^2\partial_u}\vert_{b,0}$.
    Assume that we have a decomposition $\cH_{b,0} \simeq \bigoplus_{i\in I} H_i$ stable under $K_b$, and that for any $i\neq j\in I$, the spectra of $K_b\vert_{H_i}$ and $K_b\vert_{H_j}$ are disjoint.
    Then there exists an admissible open neighborhood $U$ of $b$ such that the restriction $(\cH\vert_U ,\nabla\vert_U )/U$ decomposes into a product of maximal F-bundles $(\cH_i ,\nabla_i ) /U_i$ extending the decomposition of $\cH_{b,0}$.
\end{theorem}

\begin{proof}
    By \cref{lemma:neighborhood-k-rational-point}, we can find an admissible neighborhood $U$ of $b$ isomorphic to an admissible open neighborhood of $0$ in a $\bbk$-analytic affine space.
    Hence, we may assume that $B = \Sp T_n$ and $b=0$.
    By \cref{lemma:NA-max-F-bundle-F-manifold}, up to shrinking $B$ we can find a section of cyclic vectors $h\colon B\rightarrow \cH\vert_{u=0}$, providing an isomorphism
    \[\eta\coloneqq (u\nabla)\vert_{u=0} (h)\colon TB\longrightarrow \cH\vert_{u=0},\]
    and an F-manifold structure $\star$ on $B$.
    The splitting of $\cH_{b,0}$ induces a splitting of $T_bB$ as a $\bbk$-algebra.
    By \cref{thm:NA-decomposition-F-manifold}, there exists an admissible neighborhood $U$ of $b$ such that $(U,\star\vert_U )$ is isomorphic to a product of F-manifolds $\prod_{i\in I} (U_i,\star_i)$, and the induced decomposition of $TU$ extends the decomposition of $T_bB$.
    
    We keep denoting by $(\cH ,\nabla )$ the restriction of the F-bundle to $U$.
    The decomposition of $TU$ induces a decomposition $\cH\vert_{u=0} \simeq \bigoplus_{i\in I} \cH_{i,0}$ satisfying the assumptions of \cref{proposition:split-u-direction-NA}.
    As in the formal case, this implies that there exists F-bundles $(\cH_i,\nabla_i)/U_i$ such that 
    \[(\cH,\nabla ) \simeq \bigoplus_{i\in I}\pr_i^{\ast} (\cH_i ,\nabla_i ),\]
    where $\pr_i\colon U\simeq \prod_{j\in I} U_j\rightarrow U_i$ is the projection.

    Let $b_i$ denote the image of $b$ under the projection $U\rightarrow U_i$, let $j_i\colon U_i\hookrightarrow U$ denote the canonical closed immersion and $h_i \coloneqq j_i^{\ast}h$.
    As in the formal case, the stalk at $b_i$ of the map $\eta_i\coloneqq (u\nabla_i)\vert_{u=0}(h_i)\colon TU_i\rightarrow \cH_i\vert_{u=0}$ is an isomorphism. 
    Hence $(\cH_i,\nabla_i) /U_i$ is maximal at $b_i$.
    Up to shrinking $U_i$, this implies that $(\cH_i ,\nabla_i ) /U_i$ is maximal, completing the proof.
\end{proof}

\section{Framing of F-bundles} \label{sec:framing-F-bundle}

In this section, we prove the extension of framing theorems (\cref{theorem:extension-of-framing-connection-version,theorem:extension-faming-NA-F-bundle}).
In \cref{subsec:comparison_of_framed_F-bundles}, we apply the extension of framing to obtain a uniqueness result for isomorphisms between maximal F-bundles admitting a framing (\cref{lemma:comparison-framed-F-bundles}).
In \cref{subsec:F-bundle-over-a-point}, we provide a partial classification of framed F-bundles over a point, up to gauge equivalence, under some assumptions on the coefficients of the connection (\cref{theorem:gauge-eq-F-bundle-point}).
When the $K$-operator of the F-bundle has simple eigenvalues, we obtain a full classification in \cref{corollary:classification-framed-F-bundle-point}.
We will apply those results to the A-model F-bundles in \cref{sec:app-projective-bundle}.

\subsection{Extension of framing for logarithmic formal F-bundles}

\subsubsection{Main result} \label{subsubsec:extension-framing-log-F-bundles}

Here we state the theorem of extension of framing, and fix the notations for the proof.

\begin{definition} \label{definition:strong-framing}
    Let $(\cH ,\nabla ) /(B,D)$ be a logarithmic F-bundle and $b\in B$ a rational point.
    We say that a framing $\nabla_b^{\fr}$ for the restricted F-bundle $(\cH ,\nabla)\vert_{b}$ is \emph{strong} with respect to $D$ if for any function $q$ vanishing on $D$, the endomorphism $\nabla_{uq \partial_q}\vert_{b\times\Spf\bbk\dbb{u}}$ is independent of $u$ in a $\nabla_b^{\fr}$-flat trivialization of $\cH\vert_{b\times\Spf\bbk\dbb{u}}$.
\end{definition}

\begin{theorem}[Extension of framing]\label{theorem:extension-of-framing-connection-version}
    Let $(\cH,\nabla )/ (B,D)$ be a logarithmic F-bundle, where $B$ is a formal neighborhood of a rational point $b$ in a smooth $\bbk$-variety.
    A framing $\nabla_b^{\fr}$ for the restricted F-bundle $(\cH,\nabla)\vert_b$ extends to a framing for $(\cH,\nabla )$ if and only if $\nabla_b^{\fr}$ is strong with respect to $D$.
    In this case, the extension is uniquely and explicitly determined from $\nabla_b^{\fr}$ and $(\cH ,\nabla )$.
    \end{theorem}

We refer to \cref{example:counter-example-framing} for a counter-example to the existence part of \cref{theorem:extension-of-framing-connection-version} without assuming the framing is strong with respect to $D$.

Write $B = \Spf\bbk\dbb{q_1,\dots, q_s,t_1,\dots, t_n}$, where $\prod_{1\leq i\leq s} q_i = 0$ is a local equation for $D$ at $b$.
Let $m$ be the rank of $\cH$ and $H\coloneqq \cH_{b,0}$ the fiber of $\cH$.
We start with any trivialization $\iso\colon \cH\simeq H\times B\times \Spf\bbk\dbb{u}$ extending a $\nabla_b^{\fr}$-flat trivialization of $\cH\vert_{b\times\Spf\bbk\dbb{u}}$.
Let $\Omega$ denote the connection form of $\nabla$ in the trivialization $\iso$.
Fix a basis of $H$, and write
\begin{equation} \label{eq:notation-connection-matrix}
    \Omega = \sum_{1\leq i \leq s} u^{-1}q_i^{-1} Q^i(q,t,u)dq_i + \sum_{1\leq j\leq n} u^{-1} T^j(q,t,u)  dt_j +  u^{-2} U(q,t,u) d u ,
\end{equation}
where $U,Q^i,T^j\in \Mat (m\times m, \bbk \dbb{q_i,t_j,u})$.
The framing assumption at $b$ allows us to assume that $U(0,0,u)$ is linear in $u$.
The assumption that the endomorphism $\nabla_{uq_i\partial_{q_i}}\vert_{q=t=0}$ is $\nabla_b^{\fr}$-flat means that $Q^i (0,0,u)$ is independent of $u$.

We want to modify the trivialization $\iso$ by an automorphism of $H\times B\times\Spf\bbk\dbb{u}$, to produce a new trivialization extending $\iso\vert_{b\times\Spf\bbk\dbb{u}}$ and in which $\nabla$ is framed.
Equivalently, we seek a gauge transformation $P(q,t,u)\in \GL (m, \bbk\dbb{q_i,t_j,u})$ and matrices $K(q,t),G(q,t),\tQ^i  (q,t),\tT^j (q,t)$ in $\Mat (m\times m ,\bbk\dbb{q_i,t_j})$ such that
\begin{align}
	P^{-1} \partial_u P + u^{-2} P^{-1} UP &= u^{-2} K + u^{-1} G , \label{eq:framing-u-direction}\\
    P^{-1} \partial_{q_i} P + u^{-1}q_i^{-1}P^{-1} Q^i P &= u^{-1} q_i^{-1} \tQ^i ,\label{eq:framing-q-direction} \\
	P^{-1} \partial_{t_j} P + u^{-1} P^{-1} T^j P &= u^{-1} \tT^j , \label{eq:framing-t-direction}
\end{align}
and satisfying $P(0,0,u) = \id$.
By identifying the polar part at $u=0$, we get an expression for the matrices $K,G,\tQ^i,\tT^j$. 
In particular, setting $P_0\coloneqq P(q,t,0)$, we have the following expressions
\begin{equation} \label{eq:extension-framing-log-fixing-residues}
	\tQ^i = P_0^{-1} Q_{-1}^i P_0 \quad\mathrm{and}\quad \tT^j = P_0^{-1} T_{-1}^j  P_0 ,
\end{equation}
with $Q_{-1}^i = \nabla_{uq_i\partial_{q_i}}\vert_{u=0}$ and $T_{-1}^j = \nabla_{u\partial_{t_j}}\vert_{u=0}$.
We will construct $P$ in \cref{subsubsec:inductive-framing} order by order in each variable, starting with the logarithmic directions.

\subsubsection{Two matrix lemmas}
We now state two matrix lemmas that we will use for the proof of \cref{theorem:extension-of-framing-connection-version}.

\begin{lemma} \label{lemma:matrix-t-direction}
    Let $R$ be a ring. 
    \begin{enumerate}[wide]
        \item Let $T\in \Mat (m\times m, R\dbb{t})$.
        Let $(X_k(t))_{k\in\bbN}$ be a sequence of matrices in $\Mat (m\times m, R\dbb{t} )$ satisfying
    \[\partial_t X_k = -[T , X_{k+1} ] .\]
    Then $(X_k (t))_{k\in\bbN}$ is uniquely determined by $(X_k (0))_{k\in\bbN}$. 
    In particular, if $X_k (0) = 0$ for all $k\geq 0$, then $X_k (t) = 0$ for all $k\geq 0$.
        \item Let $n\in\bbN$, and $T_1,\dots, T_n\in \Mat (m\times m ,R\dbb{t_1,\dots, t_n} )$.
        Let $(X_k(t) )_{k\in\bbN}$ be a sequence of matrices in $\Mat (m\times m, R\dbb{t_1,\dots , t_n} )$ satisfying for all $1\leq i\leq n$
    \[\partial_{t_i} X_k = - [T_i, X_{k+1} ].\]
    Then $(X_k (t))_{k\in \bbN}$ is uniquely determined by $(X_k(0))_{k\in\bbN}$.
    In particular, if $X_k(0) = 0$ for all $k\geq 0$, then $X_k (t) = 0$ for all $k\geq 0$.
    \end{enumerate}
\end{lemma}

\begin{proof}
    For (1), we write $X_k(t) = \sum_{\ell\in\bbN} X_{\ell ,k} t^{\ell}$.
    For $d\geq 0$, we have
    \[(d+1)!X_{d+1 , k} = \left.\frac{\partial^{d+1}X_k}{\partial^{d+1} t} \right\vert_{t=0} = - \left.\frac{\partial^{d}}{\partial^{d} t} [T,X_{k+1} ] \right\vert_{t=0} =- \sum_{s=0}^d \binom{d}{s} \left.\left[\frac{\partial^{d-s} T}{\partial^{d-s} t}  , \frac{\partial^{s}X_{k+1}}{\partial^{s} t}  \right] \right\vert_{t=0}.\]
    This provides a recursive relation for $\lbrace X_{d+1,k} \rbrace_{k\in \bbN}$ in terms of $\lbrace X_{r,k} ,\; r\leq d\rbrace_{k\in\bbN}$.
    Thus, $(X_k)_{k\geq 0}$ is uniquely determined by $(X_k (0))_{k\geq 0}$.

    For (2), we apply inductively on $1\leq i\leq n$ the single variable case with the ring $R\dbb{t_1,\dots ,t_{i-1} }$.
    In this way, we prove that for $1\leq i\leq n$, the sequence $ (X_k\vert_{t_{i+1} = \cdots = t_n = 0} )_{k\in\bbN}$ is uniquely determined by the sequence $ (X_k\vert_{t_i = \cdots = t_n = 0} )_{k\in\bbN}$.
    Thus $(X_k)_{k\in\bbN}$ is uniquely determined by the initial condition $(X_k\vert_{t_1=  \cdots = t_n = 0} )_{k\in\bbN}$.

    For both (1) and (2), choosing $X_k (t) = 0$ for all $k\geq 0$ provides a sequence that satisfies the assumptions of the lemma, with the initial condition $X_k(0)= 0$.
    It follows from the uniqueness that this is the only solution to the equations such that $X_k(0) = 0$ for all $k\geq 0$.
\end{proof}

\begin{lemma} \label{lemma:matrix-q-direction}
    Let $R$ be a ring.
    For $1\leq i\leq s$, let $Q_i\in\Mat (m\times m, R\dbb{q_1,\dots ,q_s} )$ such that $\phi_i\coloneqq \ad (Q_i)\vert_{q = 0} $ is nilpotent.
    Let $(X_k(q))_{k\in\bbN}$ be a sequence of matrices in $\Mat (m\times m ,R\dbb{q_1,\dots , q_s} )$ satisfying for all $1\leq i\leq s$
    \[q_i\partial_{q_i} X_k =  [Q_i, X_{k+1} ] .\]
    Then, for any initial condition  $(X_k (0))_{k\in\bbN}$, there exists at most one solution  $(X_k (q))_{k\in\bbN}$.
    In particular, if $X_k(0) = 0$ for all $k\geq 0$, then $X_k(q) = 0$ for all $k\geq 0$.
\end{lemma}

\begin{proof}
    We use \cref{notation:tuple-integers}.
    In particular, given tuples of integers $\ell = (\ell_i )_{1\leq i\leq n}$ and $r= (r_i)_{1\leq i\leq n}$, the length of $\ell$ is $\vert\ell\vert = \ell_1+ \cdots + \ell_n$, and we write $r\preceq \ell$ if $r_i\leq \ell_i$ for all $1\leq i\leq n$.
    We denote the linear differential operator $q_i\partial_{q_i}$ by $D_i$, so the equations are $D_i X_k = [Q_i,X_{k+1}]$.

	First, a direct induction shows that for all $n\in\bbN$ we can express $D_i^{n+1} X_k$ as a linear combination of terms of the form
	\begin{equation} \label{eq:terms-recursive-relation-q-direction}
		\left[ D_i^{a_1} Q_i ,\left[ \cdots ,\left[ D_i^{a_u} Q_i , X_{k+u}\right]\cdots \right]\right],
	\end{equation}
	with $1\leq u\leq n+1$ and $(a_v)_{1\leq v\leq u}\in\bbN^u$ satisfying $a_1+\cdots + a_u + u =n+1$.
	If we denote the coefficient of such a term by $\alpha_n (a_1,\dots , a_u )$, it is elementary to see that the sequence $(\alpha_n)_{n\in\bbN}$ is fully determined by the initial condition $\alpha_0 (0) =1$ and the recursion relation
	\[\alpha_{n+1} (a_1,\dots ,a_u) = \sum_{a_v \neq 0} \alpha_n (a_1,\dots ,a_v-1,\dots, a_u) + \delta_{a_u ,0} \alpha_n (a_1,\dots, a_{u-1}) .\]
    
    Write $X_k (q) = \sum_{r\in\bbN^s} X_{r,k} q_1^{r_1}\cdots q_s^{r_s}$.
    We will show that for $d\geq 1$, the terms $\lbrace X_{\ell,k} ,\; \vert \ell\vert = d\rbrace_{k\in\bbN}$ are determined by $\lbrace X_{r,k} ,\; \vert r\vert < d\rbrace_{k\in\bbN}$.
    It will follow directly that $( X_k (q))_{k\in\bbN}$ is uniquely determined by the initial term $(X_k (0))_{k\in\bbN}$.
    Fix $\ell\in\bbN^s$ with $\vert\ell\vert=d$ and $k\in\bbN$.
    We express $X_{\ell,k}$ in terms of $\lbrace X_{r,k+s} , \; \vert r\vert < d , \; s\geq 1\rbrace$.
    Fix $i$ such that $\ell_i \neq 0$, and let $n\in\bbN$.
	We note that the coefficient of $q^{\ell}$ in $D_i^{n+1} X_k$ is $\ell_i^{s+1} X_{\ell ,k}$.
	On the other hand, by the previous paragraph $D_i^{n+1} X_k$ is a linear combination of terms of the form (\ref{eq:terms-recursive-relation-q-direction}).
    The coefficient of $q^{\ell}$ in (\ref{eq:terms-recursive-relation-q-direction}) is expressed in terms of derivatives of $Q_i$ and coefficients $X_{r,k+u}$ with $r\preceq \ell$ and $u\geq 1$.
	If $X_{\ell , k+u}$ appears in a term,
    then only the constant term of the terms involving $Q_i$ contribute.
	If $a>0$, then $D_i^a Q_i$ has no constant term, so $X_{\ell , k+u}$ appears in the relation if and only if $a_1 = \cdots = a_u = 0$.
	Given the condition $a_1+ \cdots +a_u + u = n+1$, this implies $u = n+1$ and we conclude that
	\[\ell_i^{n+1} X_{\ell ,k} = \phi_i^{n+1} (X_{\ell , k+n+1}) + \lbrace \mathrm{terms\ involving\ derivatives\ of\ } Q_i \mathrm{\ and } \ X_{r,k+u}\ \mathrm{ with\ } \vert r\vert < d \rbrace. \]
    Since $\phi_i$ is nilpotent, for $n$ large enough the right hand side does not depend on $\lbrace X_{\ell,k}\rbrace_{k\in\bbN}$, and we obtain a recursive relation determining uniquely $X_{\ell ,k}$ as a function of terms already known.
	This completes the proof.
\end{proof}

\subsubsection{Proof of \cref{theorem:extension-of-framing-connection-version}} \label{subsubsec:inductive-framing}

We formulate a condition under which we are able to solve the system of PDEs \eqref{eq:framing-u-direction}-\eqref{eq:framing-t-direction} recursively.

\begin{definition}[Nilpotency condition]\label{definition:nilpotency-condition}
    Let $(\cH, \nabla)/(B,D)$ be a logarithmic F-bundle, where $B$ is a formal neighborhood of a rational point $b$ in a smooth $\bbk$-variety.
    We say that $(\cH ,\nabla )/(B,D)$ satisfies \emph{the nilpotency condition at $b$} if for all vector $v\in T_b D$, the adjoint $\ad \mu_{b} (v)$ is nilpotent (see \eqref{eq:residue-map} for $\mu_b$).
\end{definition}

\begin{lemma}\label{lemma:strong-framing-implies-nilpotency-condition}
    Let $(\cH, \nabla)/(B,D)$ be a logarithmic F-bundle, where $B$ is a formal neighborhood of a rational point $b$ in a smooth $\bbk$-variety.
    If there exists a framing for $(\cH ,\nabla )$ at $b$ that is strong with respect to $D$, then $\mu_b (v)$ is nilpotent for every $v\in T_bB$.
    In particular, $(\cH ,\nabla )$ satisfies the nilpotency condition at $b$.
\end{lemma}

\begin{proof}
    Write $B = \Spf\bbk\dbb{q_1,\dots, q_s, t_1,\dots, t_n}$, with $q_i$ the logarithmic directions.
    Let $\nabla_b^{\fr}$ be a framing at $b$ that is strong with respect to $D$, fix a trivialization of $\cH$ extending a $\nabla_b^{\fr}$-flat trivialization.
    Fix $1\leq i\leq s$ and write 
    \[\nabla_{q_i\partial_{q_i}} = q_i \partial_{q_i} + u^{-1} Q (q,t,u) .\]
    By the assumption, $Q_0\coloneqq Q (0,0,u)$ is independent of $u$.
    Since $\nabla_b^{\fr}$ is a framing, we have $\nabla_{\partial_u}\vert_{b\times\Spf\bbk\dbb{u}} = \partial_u + u^{-2} K + u^{-1} G$, with $K$ and $G$ constant endomorphisms of $\cH_{b,0}$.
    In this trivialization, the flatness equation $[\nabla_{\partial_u} ,\nabla_{q_i\partial_{q_i}}] =0$ restricted to $b\times\Spf\bbk\dbb{u}$ reads
    \[-Q_0 = u^{-1} [Q_0,K] + [Q_0,G] .\]
    In particular $[Q_0,G] = -Q_0$.
    It follows that $[Q_0 ,[Q_0,-G]] = [Q_0,Q_0] = 0$.
    Jacobson's lemma (\cite[Lemma 4, p.\ 44]{Jacobson_Lie-Algebras}) implies that $[Q_0,-G] = Q_0$ is nilpotent, proving the first part of the lemma.
    Since the adjoint of a nilpotent endomorphism is nilpotent, the second part follows.
\end{proof}

The next series of lemmas will enable us to prove \cref{theorem:extension-of-framing-connection-version} by framing the connection inductively in each direction.
Given a logarithmic F-bundle $(\cH ,\nabla )/(B,D)$ over $B = \Spf\bbk\dbb{q_1,\dots, q_s,t_1,\dots, t_n}$, a closed subscheme $B'\subset B$ and a subsheaf $\cF\subset T_B (-\log D)$, 
we will say that $(\cH ,\nabla )$ is \emph{framed in the directions of $\cF$ at $B'$} if there exists a trivialization of $\cH$ such that $\nabla_{\xi}\vert_{B'}$ takes the form \eqref{eq:framing} for any section $\xi$ of $\cF$, i.e.\ the restriction of the connection matrix in the direction $\xi$ to $B'$ has no positive powers of $u$.
If we formulate multiple conditions involving several subsheaves and closed subschemes, we mean that there exists a trivialization in which the connection form satisfies all the formulated conditions.

\begin{lemma} \label{lemma:existence-gauge-transformation-q-direction-induction}
    Let $(\cH ,\nabla )$ be a logarithmic F-bundle over $\Spf \bbk\dbb{q_1,\dots, q_s}$ (without $t$-variables) satisfying the nilpotency condition (\cref{definition:nilpotency-condition}), fix $1\leq i\leq s$.
    Assume it is framed in all $q$-directions at $\lbrace q_j = 0 , i\leq j\leq s\rbrace$.
    Then there exists a gauge transformation $P (q_1,\dots, q_s,u)$ such that $P\vert_{q_{i} =\cdots = q_s = 0} = \id$ and $P^{\ast}\nabla$ is framed in the $q_i$-direction at $\lbrace q_j = 0 , i+1\leq j\leq s\rbrace$. 
    In particular, $P^{\ast}\nabla$ is still framed in all $q$-directions at $\lbrace q_j = 0 , i\leq j\leq s\rbrace$.

    \end{lemma}

\begin{proof}
    We let $q \coloneqq \lbrace 1,\dots ,s\rbrace$, $q^{\leq i} \coloneqq \lbrace q_1,\dots, q_i\rbrace$, $q^{\geq i} \coloneqq \lbrace q_i ,\dots, q_s\rbrace$ and $q^{>i} \coloneqq \lbrace q_{i+1} ,\dots, q_s \rbrace$.
    Let $u^{-1}q_i^{-1} Q(q,u)$ denote the connection matrix in the $q_i$-direction in a trivialization of $\cH$ provided by the partial framing assumption.
    Write $Q (q,u) = \sum_{\ell,k\geq 0} Q_{\ell, k-1} q_i^{\ell}u^k$, by the framing assumption we have $Q\vert_{q^{\geq i} = 0} = Q_{0,-1}\vert_{q^{\geq i} = 0}$.
    
    We seek a gauge transformation $P(q,u)$ such that 
    \begin{align*}
        \partial_{q_i} P\vert_{q^{>i} = 0} &= u^{-1} q_i^{-1}\left( - Q P + PP_0^{-1} Q_{-1} P_0\right)\vert_{q^{>i}= 0} ,\\
        P\vert_{q^{\geq i}=0} &= \id ,
    \end{align*}
    where $P_0 \coloneqq P(q,0)$ and $Q_{-1} \coloneqq Q(q,0)$.
    We look for $P$ of the form $P(q,u) = \sum_{\ell,k\geq 0} P_{\ell ,k}  q_i^{\ell} u^k$, where $P_{\ell ,k}$ depends on $\lbrace q_1,\dots ,q_{i-1}\rbrace$,
    We construct the solution $P$ order by order in powers of $q_i$, by expressing $\lbrace P_{\ell +1 ,k}\rbrace_{k\in\bbN}$ in terms of $\lbrace P_{\ell' ,k} ,\; \ell ' \leq \ell \rbrace_{k\in\bbN}$ for $\ell\in\bbN$.
    
    The initial condition gives $P_{0,0} = \id$ and $P_{0,k} = 0$ for $k >0$.
    Let $\ell\in\bbN$ and $k\in\bbN$.
    We isolate a monomial $q_i^{\ell}u^k$ in the differential equation and obtain 
    \[(\ell +1)P_{\ell+1 ,k} = -\sum_{\substack{\ell_1 + \ell_2 = \ell +1 \\ k_1+k_2 = k+1}} Q_{\ell_1,k_1-1}\vert_{q^{>i} = 0} P_{\ell_2,k_2} + \sum_{\ell_1+\ell_2+\ell_3+\ell_4 = \ell+1} P_{\ell_1 , k+1} (P_0^{-1} )_{\ell_2} Q_{\ell_3,-1}\vert_{q^{>i}=0} P_{\ell_4,0} ,\]
    where $(P_0^{-1} )_{\ell_2}$ is the coefficient of $q_i^{\ell_2}$ in $P_0^{-1}$.
    Using the framing assumption at $q^{\geq i} = 0$ and the initial condition for $P$, we isolate terms involving $\lbrace P_{\ell+1 ,k'}\rbrace_{k'\in \bbN}$ and obtain the relation for all $k\geq 0$
    \begin{equation} \label{eq:relation-q-direction}
        P_{\ell+1,k } = \psi_{\ell , k} (P) - \frac{1}{\ell +1} [Q\vert_{q^{\geq i} = 0} , P_{\ell+1 ,k+1} ],
    \end{equation}
    where 
    \begin{multline*}
        \psi_{\ell,k} (P) \coloneqq \frac{1}{\ell +1} \Bigg( -\sum_{\substack{k_1+k_2=k+1\\ \ell_1+\ell_2 = \ell +1 \\ \ell_2 <\ell +1 }} Q_{\ell_1 ,k_1-1}\vert_{q^{>i} = 0} P_{\ell_2 , k_2} \\
        + \sum_{\substack{\ell_1+\ell_2+\ell_3+\ell_4 = \ell+1 \\ 0< \ell_1 <\ell+1}} P_{\ell_1 , k+1}(P_0^{-1} )_{\ell_2} Q_{\ell_3 , -1} \vert_{q^{>i} = 0} P_{\ell_4 , 0} \Bigg) .
    \end{multline*}
    Note that $\psi_{\ell ,k} (P)$ only depends on $\lbrace P_{\ell ',k'} ,\; \ell ' <\ell +1 , \, k'\leq k+1\rbrace$.

    Let $E\coloneqq \Mat (m\times m ,\bbk\dbb{q_1,\dots, q_{i-1} })^{\bbN}$. 
    Consider the linear maps $\tau\colon E\rightarrow E$ given by the shift $\lbrace M_k\rbrace_{k\in\bbN}\mapsto \lbrace M_{k+1}\rbrace_{k\in\bbN}$ and $\Phi\colon E\rightarrow E$ given by $\lbrace M_k\rbrace_{k\in\bbN} \mapsto \lbrace [Q\vert_{q^{\geq i} = 0} , M_k]\rbrace_{k\in\bbN}$.
    The relations \eqref{eq:relation-q-direction} give 
    \begin{equation}\label{eq:relation-q-direction-sequence}
        \left(\id_E + \frac{1}{\ell+1} \Phi\circ\tau \right) \lbrace P_{\ell+1 ,k}\rbrace_{k\in\bbN} = \lbrace \psi_{\ell ,k} (P)\rbrace_{k\in\bbN} .
    \end{equation}
    We prove that $\id_E + \frac{1}{\ell +1}\Phi\circ \tau$ is invertible.
    To do so, it is enough to prove that it is invertible at $q_1 = \cdots = q_{i-1} = 0$.
    The map $\Phi\vert_{q_1= \dots = q_{i-1} = 0}$ is nilpotent, since $\ad (Q\vert_{q =u=0} )$ is.
    The maps $\tau$ and $\Phi$ commute, so the composition $\Phi\circ\tau\colon E\rightarrow E$ is also nilpotent at $q_1 = \cdots = q_{i-1} = 0$.
    Hence $\id_E + \frac{1}{\ell +1}\Phi\circ \tau$ is invertible at $q_1 = \cdots = q_{i-1} = 0$.
    It follows that $\id_E + \frac{1}{\ell+1} \Phi\circ\tau$ is invertible, and composing \eqref{eq:relation-q-direction-sequence} with its inverse provides a recursive relation determining the coefficient of $q_i^{\ell+1}$ from lower order terms.
        Hence the differential equation admits a solution $P (q ,u)$ such that $P\vert_{q^{\geq i}=0} =\id$.
    The initial condition implies that the connection $P^{\ast} \nabla$ is still framed in all $q$-directions at $q^{\geq i}=0$. This completes the proof.
\end{proof}

\begin{lemma} \label{lemma:induction-framing-q-direction-without-t}
    Let $(\cH ,\nabla )$ be a logarithmic F-bundle over $\Spf\bbk\dbb{q_1,\dots, q_s}$ (with no $t$-variables) satisfying the nilpotency condition (\cref{definition:nilpotency-condition}), fix $1\leq i\leq s$.
    Assume it is framed in all $q$-directions at $\lbrace q_j = 0 , i\leq j\leq s\rbrace$, and framed in the $q_i$-direction at $\lbrace q_j = 0 , i+1\leq j\leq s\rbrace$.
    Then $(\cH ,\nabla )$ is framed in all the $q$-directions at $\lbrace q_j = 0 , i+1\leq j\leq s\rbrace$.
\end{lemma}

\begin{proof}
    Let $q^{\leq i} \coloneqq \lbrace q_1,\dots, q_i\rbrace$.   
    The partial framing assumption provides a trivialization of $\cH$.
    For $1\leq i'\leq s$, let $u^{-1}q_{i'}^{-1}Q^{i'} (q^{\leq i} ,u ) = q_{i'}^{-1} \sum_{k\geq 0} Q_{k-1}^{i'} (q^{\leq i} )u^{k-1}$ denote the restriction of the connection matrix in the $q_{i'}$-direction to $q_{i+1} = \cdots = q_s = 0$.
    The framing assumption means that $Q_k^{i'}\vert_{q_i=0} = 0$ and $Q_k^i = 0$ for all $k\geq 0$ and $1\leq i'\leq s$.
    
    Fix $1\leq i'\leq s$, with $i'\neq i$.
    For $k\geq 0$, the $u^k$ term of the flatness equation $[\nabla_{q_{i'}\partial_{q_{i'}}} , \nabla_{q_i\partial_{q_i}} ] =0$ provides the equation 
    \[q_i\partial_{q_i} Q_k^{i'} = -[ Q_{-1}^i  ,Q_{k+1}^{i'} ] .\]
    Since $\ad (Q_{-1}^s(0) )$ is nilpotent, we can apply \cref{lemma:matrix-q-direction} with $R = \bbk\dbb{q_1,\dots , q_{i-1}}$ and $X_k = Q_k^{i'}$.
    We deduce that $Q_k^{i'} = 0$ for all $k\geq 0$, proving that the connection is also framed in the $q_{i'}$-direction at $q_{i+1} = \cdots = q_s = 0$.
\end{proof}

\begin{lemma} \label{lemma:framing-initial-condition-t-directions}
    Let $(\cH ,\nabla )$ be a logarithmic F-bundle over $\Spf\bbk\dbb{q_1,\dots, q_s,t_1,\dots, t_n}$ framed in the $q$-directions at $t=0$.
    Then there exists a gauge transformation $P$ such that $P\vert_{t=0} = \id$ and $P^{\ast}\nabla$ is framed in all the $q$-directions and $t$-directions at $t=0$.
\end{lemma}

\begin{proof}
    We work in a trivialization of $\cH$ provided by the partial framing assumption.
    For $1\leq i\leq s$, let $u^{-1}q_i^{-1} Q^i (q,t,u)$ denote the connection matrix in the $q_i$-direction in this trivialization.
    For $1\leq j\leq n$, let $u^{-1} T^j (q,t,u)$ denote the connection matrix in the $t_j$-direction in this trivialization. 
    Let 
    \[P(q,t,u) \coloneqq \prod_{j=1}^n \left(\id-t_j \frac{T^j (q,0,u)- T^j (q,0,0)}{u} \right) .\]
    Note that $P(q,t,u)$ only has non-negative powers of $u$, because $T^j (q,0,u) - T^j (q,0,0)$ has no constant term in $u$.
    We have $P\vert_{t=0} = P^{-1}\vert_{t=0} = \id$, and we compute $\frac{\partial P}{\partial q_i}\big\vert_{t=0} = 0$ and $\frac{\partial P}{\partial t_j}\big\vert_{t=0} = - u^{-1} (T^j (q,0,u) - T^j (q,0,0))$.
    The connection matrix of $P^{\ast}\nabla$ in the $t_j$-direction at $t=0$ is 
    \[\bigg[ P^{-1}\frac{\partial P}{\partial t_j} + u^{-1} P^{-1}T^j P\bigg]\bigg\vert_{t=0} = u^{-1} (-T^j(q,0,u) + T^j (q,0,0) + T^j (q,0,u) )= u^{-1} T^j (q,0,0) , \]
    which is framed. 
    The connection matrix of $P^{\ast}\nabla$ in the $q_i$-direction at $t=0$ is 
    \[\bigg[ P^{-1}\frac{\partial P}{\partial q_i} + u^{-1}q_i^{-1} P^{-1}Q^i P\bigg]\bigg\vert_{t=0} = u^{-1}q_i^{-1} Q^i (q,0,u), \]
    which is also framed.
    The lemma is proved.
\end{proof}

\begin{lemma} \label{lemma:existence-gauge-transformation-t-direction-induction}
    Let $(\cH ,\nabla )$ be a logarithmic F-bundle over $\Spf\bbk\dbb{q_1,\dots, q_s,t_1,\dots, t_n}$ framed in the $q$-directions at $t=0$, fix $1\leq j\leq n$.
    Assume it is framed in all $t$-directions at $\lbrace t_i=0 , j\leq i\leq n\rbrace$.
    Then there exists a gauge transformation $P(q,t,u)$ such that $P\vert_{t_j = \cdots = t_n = 0} = \id$ and $P^{\ast} \nabla$ is framed in the $t_j$-direction at $\lbrace t_i=0 , j+1\leq i\leq n\rbrace$, framed in all the $q$-directions at $t=0$, and in all the $t$-directions at $\lbrace t_i=0 , j\leq i\leq n\rbrace$.
    \end{lemma}

\begin{proof}
    Let $t^{\leq j}\coloneqq \lbrace t_1,\dots, t_j\rbrace$, $t^{\geq j}\coloneqq \lbrace t_j,\dots, t_n\rbrace$ and $t^{>j} \coloneqq \lbrace t_{j+1},\dots, t_n\rbrace$.
    Let $u^{-1} T (q,t,u)$ denote the connection matrix in the $t_j$-direction in a trivialization of $\cH$ provided by the partial framing assumption.
    Write $T (q,t,u) = \sum_{\ell,k\in\bbN} T_{\ell , k-1} t_j^{\ell} u^k$,
    by the framing assumption we have $T\vert_{t^{\geq j}=0} = T_{0,-1}\vert_{t^{\geq j}  =0}$.
    
    We seek a gauge transformation $P(q,t,u)$ such that
    \begin{align*}
        \partial_{t_j} P\vert_{t^{>j} = 0} &= u^{-1}\left( - T P + PP_0^{-1} T_{-1} P_0\right)\vert_{t^{>j} = 0}, \\
        P\vert_{t^{\geq j}=0} &= \id , 
    \end{align*}
    where $P_0 \coloneqq P(q,t,0)$ and $T_{-1}\coloneqq T(q,t,0)$.
    We look for $P$ of the form $P(q,t,u) = \sum_{\ell, k\geq 0} P_{\ell ,k} t_j^{\ell} u^k$, where $P_{\ell ,k}$ depends on the variables $\lbrace q_1,\dots  q_s,t_1,\dots , t_{j-1}\rbrace$.
    The differential equation provides a recursive relation for $\lbrace P_{\ell ,k}\rbrace_{k\in\bbN}$. 
    By isolating the coefficient of $t_j^{\ell}u^k$ we obtain
    \begin{equation}\label{eq:recursion-framing-t-direction}
        (\ell +1) P_{\ell+1 ,k} = -\sum_{\substack{\ell_1+\ell_2 = \ell \\ k_1+k_2 = k+1}} T_{\ell_1 , k_1-1}\vert_{t^{>j}=0} P_{\ell_2 ,k_2} + \sum_{\ell_1+\ell_2+\ell_3+\ell_4 = \ell} P_{\ell_1,k+1} (P_0^{-1})_{\ell_2} T_{\ell_3,-1}\vert_{t^{>j} =0} P_{\ell_4,0} ,
    \end{equation}
    where $(P_0^{-1} )_{\ell_2}$ denotes the coefficient of $t_j^{\ell_2}$ in $P_0^{-1}$.
    This determines $P$ from the initial data $\lbrace P_{0,k} \rbrace_{k\in\bbN}$, i.e.\ from $P\vert_{t^{\geq j} =0} = \id$.
    Hence the differential equation admits a solution $P(q,t,u)$ such that $P\vert_{t^{\geq j}} = \id$.
    By construction, $P^{\ast}\nabla$ is framed in the $t_j$-direction at $t^{>j}=0$.

    We now check that the other $t$-directions are still framed at $t^{\geq j} =0$, and that the $q$-directions are still framed at $t=0$.
    Since $P\vert_{t^{\geq j} =0} = \id$, the connection matrices at $t^{\geq j} =0$ are modified by the first derivatives of $\sum_{k\geq 0}P_{1,k}u^k$.
    From the recursion \eqref{eq:recursion-framing-t-direction}, the initial condition for $P$ and the framing assumption for $T$ we obtain that $P_{1,k} = - T_{0,k}\vert_{t^{>j}= 0} = 0$ for all $k\geq 0$.
    We conclude that $P^{\ast}\nabla$ remains framed in all the $t$-directions at $t^{\geq j}=0$ and in all the $q$-directions at $t=0$, concluding the proof.
\end{proof}

\begin{lemma} \label{lemma:induction-framing-t-direction}
    Let $(\cH ,\nabla )$ be a logarithmic F-bundle over $\Spf\bbk\dbb{q_1,\dots, q_s,t_1,\dots, t_n}$, fix $1\leq j\leq n$.
    Assume it is framed in all the $t$-directions at $\lbrace t_i=0 , j\leq i\leq n\rbrace $, and framed in the $t_j$-direction at $t_{j+1} = \cdots = t_n = 0$.
    Then $(\cH ,\nabla )$ is framed in all the $t$-directions at $\lbrace t_i=0 , j+1\leq i\leq n\rbrace$.
\end{lemma}

\begin{proof}
    Let $t^{\leq j} \coloneqq \lbrace t_1,\dots, t_j\rbrace$.   
    The partial framing assumption provides a trivialization of $\cH$.
    For $1\leq j'\leq n$, let $u^{-1} T^{j'}(q,t^{\leq j},u) = \sum_{k\geq 0} T_{k-1}^{j'} (q,t^{\leq j}) u^{k-1}$ denote the restriction of the connection matrix in the $t_{j'}$-direction to $t_{j+1} = \cdots = t_n = 0$.
    The framing assumption means that $T_k^{j'}\vert_{t_j=0} = 0$ and $T_k^j = 0$ for all $k\geq 0$ and $1\leq j'\leq n$.

    Fix $1\leq j'\leq n$, with $j'\neq j$.
    For $k\geq 0$, the $u^k$ term of the flatness equation $[\nabla_{\partial_{t_{j'}}} ,\nabla_{\partial_{t_j}} ] = 0$ provides the equation
    \[\partial_{t_j} T_k^{j'} = -[T_{-1}^j , T_{k+1}^{j'}] .\]
    We apply \cref{lemma:matrix-t-direction}(1) with $R = \bbk\dbb{q_1,\dots ,q_s ,t_1,\dots , t_{j-1}}$, $X_k = T_k^{j'}$ and the initial condition $T_k^{j'}\vert_{t_j= 0} =0$,
    and deduce that $T_k^{j'}(q,t^{\leq j})  =0$ for all $k\geq 0$.
    Thus, the connection is also framed in the $t_{j'}$-direction at $t_{j+1} = \cdots = t_n = 0$.
\end{proof}

\begin{lemma} \label{lemma:framed-t-and-q-no-t-implies-framed-t-and-q}
    Let $(\cH ,\nabla )$ be a logarithmic F-bundle over $\Spf \bbk\dbb{q_1,\dots, q_s,t_1,\dots, t_n}$.
    Assume it is framed in the $t$-directions and framed in the $q$-directions at $t=0$.
    Then $(\cH ,\nabla )$ is also framed in the $q$-directions. 
\end{lemma}

\begin{proof}
    In a trivialization provided by the framing assumption, 
    denote by $u^{-1}T_{-1}^j (q,t)$ the connection matrix in the $t_j$-direction ($1\leq j\leq n$) and by $u^{-1}q_i^{-1} Q^i (q,t,u)$ the connection matrix in the $q_i$-direction ($1\leq i\leq s$).
    Write $Q^i = \sum_{k\geq 0} Q_{k-1}^i (q,t) u^k$.
    The framing assumption means that $Q_k^i\vert_{t = 0} = 0$ for $1\leq i\leq s$ and $k\geq 0$.

    Fix $1\leq i\leq s$.
    For $k\geq 0$, the $u^k$ term of the flatness equation $[\nabla_{\partial_{t_j}}  ,\nabla_{q_i\partial_{q_i}} ] = 0$ is 
    \[\partial_{t_j} Q_k^i = -[T_{-1}^j , Q_{k+1}^i]  .\]
    We apply \cref{lemma:matrix-t-direction}(2) with $R = \bbk\dbb{q_1,\dots, q_s}$, $X_k = Q_k^i$ and the initial condition $Q_k^i\vert_{t=0} = 0$, and deduce that $Q_k^i (q,t) = 0$ for all $k\geq 0$.
    Thus, the connection is also framed in the $q_i$-direction.
\end{proof}

\begin{lemma} \label{lemma:t-and-q-directions-implie-u-direction}
    Let $(\cH ,\nabla )$ be a logarithmic F-bundle over $\Spf\bbk\dbb{q_1,\dots, q_s,t_1,\dots, t_n}$ satisfying the nilpotency condition (\cref{definition:nilpotency-condition}).
    Assume it is framed in the $q$-directions and $t$-directions, and framed in the $u$-direction at $q=t=0$.
    Then $(\cH ,\nabla )$ is also framed in the $u$-direction. 
\end{lemma}

\begin{proof}
    In a trivialization provided by the framing assumption,
    let $u^{-1} q_i^{-1} Q^i (q,t)$ (resp.\ $u^{-1}T^j (q,t)$) denote the connection matrix in the $q_i$-direction (resp.\ $t_j$-direction).
    Let $u^{-2} U(q,t,u)$ denote the connection matrix in the $u$-direction.
    Write $U(q,t,u) = \sum_{k\geq 0} U_{k-2} (q,t) u^{k}$.
    The framing assumption means that for $k\geq 0$, we have $U_k (0,0) = 0$.
    
    For $k\geq 0$, and $1\leq i\leq s$, the $u^k$ term of the flatness equation $[\nabla_{\partial_u} , \nabla_{q_i\partial_{q_i}} ] = 0$ provides the equation
    \[q_i\partial_{q_i} (U_k) = -[Q^i , U_{k+1} ] .\]
    We restrict this equation to $t=0$.
    Since $\ad (Q^i (0,0) )$ is nilpotent, we can apply \cref{lemma:matrix-q-direction} with $R =\bbk$ and $X_k = U_k (q,0)$ to deduce that $U_k (q,0) = 0$ for all $k\geq 0$.

    Next, for $k\geq 0$, the $u^k$ term of the flatness equation $[\nabla_{\partial_u} , \nabla_{\partial_{t_j}} ] = 0$ provides the equation
    \[\partial_{t_j} (U_k) = -[T^j , U_{k+1} ] .\]
    We apply \cref{lemma:matrix-t-direction}(2) with $R = \bbk \dbb{q_1,\dots ,q_s}$, $X_k = U_k (q,t)$ and the initial condition $U_k (q,0) = 0$, and deduce that $U_k (q,t) = 0$ for all $k\geq 0$.
    Thus, the connection is also framed in the $u$-direction.
\end{proof}

We can now finish the proof of \cref{theorem:extension-of-framing-connection-version}.

\begin{proof}[Proof of \cref{theorem:extension-of-framing-connection-version}]
Fix a trivialization $\cH\simeq H\times (B\times\Spf\bbk\dbb{u})$ extending the trivialization of $\cH\vert_{b\times\Spf\bbk\dbb{u}}$ induced by $\nabla_b^{\fr}$.
As explained after \cref{theorem:extension-of-framing-connection-version}, the content of the theorem reduces to proving existence and uniqueness of a solution $P(q,t,u)$ to the overdetermined nonlinear system of PDEs \eqref{eq:framing-u-direction}-\eqref{eq:framing-t-direction} with initial condition $P(0,0,u) = \id$.

We prove the existence part of the statement.
If there exists a framing $\nabla^{\fr}$ extending $\nabla_b^{\fr}$, then we see that $\nabla_b^{\fr}$ is strong with respect to $D$ by working in a $\nabla^{\fr}$-flat trivialization.
Conversely assume that $\nabla_b^{\fr}$ is strong with respect to $D$, in particular the nilpotency condition is satisfied by \cref{lemma:strong-framing-implies-nilpotency-condition}.
We first frame the restricted F-bundle $(\cH',\nabla')\coloneqq (\cH,\nabla )\vert_{t=0}$, defined over the base $B' \coloneqq \Spf\bbk\dbb{q_1,\dots, q_s}$.
Applying inductively \cref{lemma:existence-gauge-transformation-q-direction-induction,lemma:induction-framing-q-direction-without-t} on $i\in\lbrace 1,\dots, s\rbrace$, we obtain a gauge transformation $P(q,u)$ such that $P(0,u) = \id$ and $P^{\ast}\nabla'$ is framed in all the $q$-directions.
Note that to apply the lemmas for the base case $i=1$, we use that $\nabla_b^{\fr}$ is strong with respect to $D$.
Extending this gauge transformation constantly in the $t$-directions, we obtain a gauge transformation $P_1 (q,u)\in\Aut (\cH)$ with $P_1 (0,u) =\id$ such that $\nabla_1\coloneqq P_1^{\ast}\nabla$ is framed in all the $q$-directions at $t=0$.
By \cref{lemma:framing-initial-condition-t-directions}, we obtain a gauge transformation $P_2 (q,t,u)\in\Aut (\cH )$ with $P_2 (q,0,u) = \id$ such that $\nabla_2\coloneqq P_2^{\ast}\nabla_1$ is framed in all the $q$-directions and $t$-directions at $t=0$.
Applying inductively \cref{lemma:existence-gauge-transformation-t-direction-induction,lemma:induction-framing-t-direction} on $j\in\lbrace 1,\dots ,n\rbrace$, we obtain a gauge transformation $P_3 (q,t,u)\in\Aut (\cH)$ with $P(q,0,u) = \id$ such that $\nabla_3\coloneqq P_3^{\ast}\nabla_2$ is framed in all the $q$-directions at $t=0$, and in all the $t$-directions along $B$.
By \cref{lemma:framed-t-and-q-no-t-implies-framed-t-and-q}, the connection $\nabla_3$ is also framed in all the $q$-directions along $B$.
Since $\nabla_{3,\partial_u}\vert_{q=t=0} = \nabla_{\partial_u}\vert_{q=t=0}$, the connection $\nabla_3$ is framed in the $u$-directions at $q=t=0$.
We conclude by \cref{lemma:t-and-q-directions-implie-u-direction} that $\nabla_3$ is framed in the $u$-direction as well. 
Thus the gauge transformation $\tP \coloneqq P_3P_2P_1$ solves the system \eqref{eq:framing-u-direction}-\eqref{eq:framing-t-direction} with the initial condition $\tP (0,0,u) = \id$, concluding the proof of existence.

We now prove uniqueness.
Assume the system of PDEs is written in a trivialization in which the connection is framed. 
In particular, the nilpotency condition is satisfied by \cref{lemma:strong-framing-implies-nilpotency-condition}.
From the equations in the directions of $B$ we obtain recursive relations as in \eqref{eq:relation-q-direction-sequence} and \eqref{eq:recursion-framing-t-direction}.
Hence, any solution is uniquely determined by the condition $P(0,0,u) = \id$.
\end{proof}

\subsubsection{Framings on rank 1 F-bundles}\label{subsubsec:framing-rank-1}

F-bundles do not admit framings in general (see \cite[\S IV.5.b]{Sabbah_Isomonodromic_deformations} for a sufficient condition), even though we established the extension of framing in \cref{theorem:extension-of-framing-connection-version}.
Here we discuss the existence of framing on rank 1 F-bundles.

\begin{proposition}
    Let $B$ be a formal neighborhood of a rational point $b$ in a smooth $\bbk$-variety.
    Let $(\cH ,\nabla )/B$ be a (non-logarithmic) formal F-bundle of rank $1$.
    Then it admits a framing.
\end{proposition}

\begin{proof}
    We keep the notations of the proof of \cref{theorem:extension-of-framing-connection-version}.
    In the non-logarithmic case there are no $q$-variables, and in the rank $1$ case the matrices are elements of $\bbk\dbb{t,u}$, so they commute.
    Then $K = U_{-2}$, $G = U_{-1}$ and $\tT^i = T_{-1}^i$ for $1\leq i\leq n$.
    The system of PDEs (\ref{eq:framing-u-direction})-(\ref{eq:framing-t-direction}) is then
  \begin{align*}
    \partial_u P (t,u) + P (t,u) U_{\geq 0} (t,u) &= 0 , \\
    \partial_{t_i} P (t,u) + P(q,t,u) T_{\geq 0}^i (t,u) &= 0 ,
  \end{align*}
  where $U_{\geq 0} = \sum_{k\geq 0} U_k u^{k}$ and $T_{\geq 0}^i = \sum_{k\geq 0} T_k^i u^{k}$.
  We furthermore need $P(0,0)\neq 0$ in order for $P(t,u)$ to be invertible.

  It is readily checked, using flatness, that the ansatz 
  \begin{multline} \label{eq:framing-rank-1}
    P(t,u) = \\
    \exp\left( -\sum_{i=1}^n\int_0^{t_i} \left(T_{\geq 0}^i (t_1,\dots , t_{i-1} ,s_i ,0,\dots, 0 ,u) + T_{\geq 0}^i (0,u) \right)ds_i - \int_0^u U_{\geq 0} (0,v) d v \right)
  \end{multline}
  solves the system of PDEs, and is invertible since $P(0,0 ) =1$.
\end{proof}

In the following example, we discuss the case of rank $1$ logarithmic F-bundle, and provide a counter-example to the existence part of \cref{theorem:extension-of-framing-connection-version} without assuming the framing is strong with respect to $D$.

\begin{example} \label{example:counter-example-framing}

    Let $\cH$ be the trivial rank $1$ bundle over $\Spf\bbk\dbb{q,u}$.
    Let $\nabla = d + \Omega$ be the connection on $\cH$ with $\Omega = \alpha \frac{dq}{q}$, where $\alpha\in \bbk$.
    Then $(\cH,\nabla )$ is a F-bundle and $\nabla_0^{\fr} =d$ is a framing for $(\cH,\nabla )\vert_{q=0}$.
    It is strong with respect to $D$ if and only if $\alpha = 0$.
    The differential system to solve in order to extend the framing is
    \begin{align*}
        \frac{\partial P}{\partial u} &= 0 ,\\
        q\frac{\partial P}{\partial q} + \alpha P &=  0.
    \end{align*}
    If $\alpha\neq 0$, all solutions to this system are scalar multiples of $\alpha q^{-1}$. 
    In particular, they are not well-defined at $q=0$. \end{example}

\subsection{Extension of framing for non-archimedean F-bundles}

In this subsection, we establish the theorem of extension of framing for non-archimedean F-bundles, building on the results of the previous subsection.

\begin{theorem} \label{theorem:extension-faming-NA-F-bundle}
    Let $B$ be a smooth $\bbk$-analytic space, and $b\in B$ a $\bbk$-rational point.
    Let $(\cH ,\nabla )$ be a non-archimedean F-bundle over $B$. 
    Then every framing of $(\cH, \nabla)$ at $b$ extends uniquely and explicitly to a framing over an admissible open neighborhood $U$ of $b$ in $B$.
\end{theorem}

We need to show that the gauge transformation $P(t,u)$ constructed in the formal case is convergent on an admissible open neighborhood of $t=0$, $u=0$.
This gauge transformation is characterized by $P(0,u) = 0$ and the equations \eqref{eq:framing-t-direction} for $1\leq j\leq n$.
We use these equations to obtain estimates on the coefficients of $P(t,u)$.

\begin{lemma} \label{lemma:estimate-inverse-matrix}
    Let $(R,\vert\cdot\vert )$ be a Banach $\bbk$-algebra. 
	Let $Q = \id +\sum_{r\geq 1} Q_r t^r\in \Mat (m\times m, R)\dbb{t}$, and write $Q^{-1} = \id +\sum_{r\geq 1} (Q^{-1})_{r} t^r$.
	For $\ell \geq 1$ we have
	\[\vert (Q^{-1})_{\ell} \vert \leq \max_{\substack{\ell \geq k\geq 1 , \; r_i\geq 1\\r_1+\cdots + r_k = \ell}} \prod_{i=1}^k \vert Q_{r_i}\vert  .\]
\end{lemma}
\begin{proof}
We have $Q^{-1} = \sum_{k\geq 0} (-1)^k \left(\sum_{r\geq 1} Q_r t^r\right)^k$.
Isolating the coefficient of $t^{\ell}$ ($\ell\geq 1$) we obtain
\[(Q^{-1})_{\ell} = \sum_{k\geq 0} \sum_{\substack{r_1+\cdots +r_k = \ell \\ r_i\geq 1}} \prod_{1\leq i\leq k} Q_{r_i} ,\]
and we see that only the range $1\leq k\leq \ell$ contributes.
This completes the proof.
\end{proof}

\begin{proposition} \label{prop:solution-t-direction-converge}
Let $(R,\vert\cdot\vert )$ be a Banach $\bbk$-algebra and let $T\in \Mat (m\times m, R)\langle t,u\rangle$.
Let $P(t,u) \in \Mat (m\times m, R)\dbb{t,u}$ be the unique solution of the system
\begin{align*}
	\partial_{t} P &= u^{-1} \left( -T P + PP_0^{-1} T_{-1} P_0\right), \\
	P(0,u) &= \id,
\end{align*}
where $P_0\coloneqq P(t,0)$.
Then $P$ is convergent on the open disk of radius $\min \big(1, \frac{1}{\vert T\vert} \big)$, meaning that for all $0< \rho < \min \big(1,\frac{1}{\vert T\vert}\big)$ in $\sqrt{\vert\bbk^{\times}\vert}$ we have $P\in \Mat (m\times m, R)\langle \rho^{-1} t,\rho^{-1} u)$.
\end{proposition}

\begin{proof}
We write $T = \sum_{\substack{\ell \geq 0\\ k\geq -1}} T_{\ell ,k} t^{\ell}u^{k+1}$.
Since we assume $T$ is convergent on the closed unit disk, we have for all $\ell\geq 0$, $k\geq -1$
\begin{equation} \label{eq:convergence-T-matrix}
	 \vert T_{\ell ,k}\vert\leq \vert T\vert .
\end{equation}

Let $P \coloneqq \id + \sum_{\substack{\ell\geq 1 \\ k\geq 0}} P_{\ell,k} t^{\ell} u^k$ and $v_{\ell,k} \coloneqq \vert P_{\ell,k} \vert $.
If we show $v_{\ell,k}\leq \alpha^{\ell+k}$ for $\alpha >0$, then $P(t,u)$ converges on the open polydisk of radius $\frac{1}{\alpha}$.

We have seen in \cref{lemma:existence-gauge-transformation-t-direction-induction} that $P$ is uniquely determined by the recursion
 \[(\ell +1) P_{\ell+1 ,k} = -\sum_{\substack{\ell_1+\ell_2 = \ell \\ k_1+k_2 = k+1}} T_{\ell_1 , k_1-1} P_{\ell_2 ,k_2} + \sum_{\ell_1+\ell_2+\ell_3+\ell_4 = \ell} P_{\ell_1,k+1} (P_0^{-1})_{\ell_2} T_{\ell_3,-1} P_{\ell_4,0} .\]
Applying the norm, we obtain
\begin{align*}
	(\ell +1) v_{\ell+1,k} &\leq \max \left( \max_{\substack{\ell_1 +\ell_2= \ell \\k_1+k_2 = k+1}} \vert T_{\ell_1,k_1 -1} \vert \vert P_{\ell_2,k_2}\vert , \max_{\substack{\ell_1+\ell_2+\ell_3+\ell_4=\ell\\ \ell_1\neq 0}} \vert P_{\ell_1,k+1} \vert \vert(P_0^{-1})_{\ell_2} \vert \vert(T_{-1} )_{\ell_3}\vert \vert (P_0)_{\ell_4}\vert\right) \\
	&\leq \vert T\vert \cdot \max \left( \max_{\substack{\ell_1 + \ell_2= \ell \\k_1+k_2 = k+1}} v_{\ell_2,k_2} , \max_{\substack{\ell_1+\ell_2+\ell_3+\ell_4=\ell\\ \ell_1\neq 0}} v_{\ell_1 ,k+1}v_{\ell_4,0} \vert(P_0^{-1})_{\ell_2} \vert    \right),
\end{align*}
where on the second inequality we use (\ref{eq:convergence-T-matrix}).

Let $\alpha\coloneqq \max (1,\vert T\vert )$.
We use the above inequality to prove by induction on $\ell\geq 0$ that
\[\forall k\geq 0,\quad v_{\ell,k}\leq \alpha^{\ell} .\]
For $\ell = 0$, we have $v_{0,k} = \delta_{0,k}$ so the inequality is obvious.
Now assume $v_{r,k}\leq \alpha^r$ for all $r\leq \ell$.
By \cref{lemma:estimate-inverse-matrix}, we then have $\vert (P_0^{-1})_r\vert \leq \max_{1\leq i\leq r} \alpha^{i} = \alpha^r$ for all $r\leq \ell$.
Since $\alpha\geq \vert T\vert$, we deduce that
\begin{align*}
    (\ell+1 )v_{\ell +1 , k} &\leq \vert T\vert \max \left( \max_{\substack{s\leq \ell \\ k_2\leq k+1}} \alpha^{s} , \max_{\substack{\ell_1+\ell_2+\ell_4\leq \ell \\ \ell_1\neq 0}}\alpha^{\ell_1+\ell_2+\ell_4}\right) \\
    &= \vert T\vert \alpha^{\ell} \leq \alpha^{\ell +1} \leq (\ell +1)\alpha^{\ell +1}.
\end{align*}
This concludes the inductive step.

Since $\alpha \geq 1$, we have $v_{\ell ,k} \leq \alpha^{\ell} \leq \alpha^{\ell +k}$ for all $\ell,k\geq 0$.
We deduce that $P$ converges on the open disk of radius $\frac{1}{\alpha}$, completing the proof.
\end{proof}

We can now finish the proof of \cref{theorem:extension-faming-NA-F-bundle}.

\begin{proof}[Proof of \cref{theorem:extension-faming-NA-F-bundle}]
    Up to restricting to an open neighborhood of $b$, we may assume that $B = \Sp T_n$ by \cref{lemma:neighborhood-k-rational-point}.
    Let $( t_1,\dots ,t_n )$ be local analytic coordinates centered at $b$.
    After rescaling we can assume that the connection matrices converge on $\Sp \bbk\langle t_1, \dots, t_n, u\rangle$.

    As in the formal case, we can reformulate the extension of framing problem into a system of PDEs \eqref{eq:framing-u-direction} and \eqref{eq:framing-t-direction}.
    We can solve the equations \eqref{eq:framing-t-direction} inductively on the number of $t$-variables, and by \cref{lemma:t-and-q-directions-implie-u-direction} the equation \eqref{eq:framing-u-direction} will be automatically satisfied.
    Using \cref{prop:solution-t-direction-converge} inductively,
    we obtain that at each step the solution, i.e\ the gauge transformation, converges on an admissible open neighborhood of $b$.
    \end{proof}

\subsection{Reconstruction of isomorphism of framed maximal F-bundles}\label{subsec:comparison_of_framed_F-bundles}

In this subsection, we explain how to use the extension of framing for logarithmic F-bundles (\cref{theorem:extension-of-framing-connection-version}) to reconstruct an isomorphism of framed maximal F-bundles compatible with the framings.
This is useful for establishing the uniqueness of mirror maps in applications to enumerative geometry.

\begin{definition}[Compatibility of framings]
    For $i=1,2$ let $(\cH_i,\nabla_i ,\nabla_i^{\fr} ) /(B_i,D_i)$ be two framed logarithmic F-bundles.
    A morphism $(f,\Phi) \colon (\cH_1,\nabla_1 ) /(B_1,D_1) \rightarrow (\cH_2,\nabla_2 )/ (B_2,D_2)$ of logarithmic F-bundles is said to be \emph{compatible with the framings} if $\Phi\circ\nabla_1^{\fr} = (f\times\id_u)^{\ast}\nabla_2^{\fr} \circ \Phi$.
\end{definition}

\begin{proposition} 
\label{lemma:comparison-framed-F-bundles}
For $i=1,2$, let $(\cH_i,\nabla_i)/(B_i,D_i)$ be a logarithmic F-bundle where $B_i$ is the formal neighborhood of a rational point in a smooth $\bbk$-variety.
Let $(f,\Phi) \colon (\cH_1,\nabla_1) /(B_1,D_1) \rightarrow (\cH_2,\nabla_2 )/(B_2,D_2)$ be an isomorphism of logarithmic F-bundles with $f(b_1)=b_2$.
Assume $(\cH_1,\nabla_1 )/\allowbreak (B_1,\allowbreak D_1)$ has a framing $\nabla_1^{\fr}$.
\begin{enumerate}[wide]
    \item The bundle map $\Phi$ is uniquely determined by its restriction to $\cH_1\vert_{b_1\times\Spf \bbk\dbb{u}}$.
        \item If $(\cH_1 ,\nabla_1 )$ and $(\cH_2,\nabla_2)$ are maximal,
        then the map on the bases $f$ is also uniquely determined by its restriction to $b_1$, up to some multiplicative constants in the logarithmic directions.
		The reconstruction is explicit after fixing compatible cyclic vectors at $b_1$ and $b_2$.
\end{enumerate}
\end{proposition}
\begin{proof}
        For (1), let $H_i$ denote the fiber of $\cH_i$ over $b_i$, and $\phi\in \Hom (H_1,H_2)$ the restriction of $\Phi$ at $b_1$.
    Fix a $\nabla_1^{\fr}$-flat trivialization $\Psi_1$ of $\cH_1$ and an arbitrary trivialization $\Psi_2$ of $(f\times\id_u)^{\ast}\cH_2$, producing the commutative diagram 
    \[
    \begin{tikzcd}[ampersand replacement=\&]
    \cH_1 \ar[r,"\Psi_1"] \ar[d, "\Phi"] \& H_1\times B_1\times \Spf \bbk\dbb{u} \ar[d, "\tPhi"] \\
    (f\times\id_u)^{\ast} \cH_2 \ar[r, "\Psi_2"] \& H_2\times B_1\times \Spf \bbk\dbb{u} .
    \end{tikzcd}
    \]
    Denote by $\varphi\colon \cH_1 \rightarrow (f\times\id_u)^{\ast}\cH_2$ the map obtained from $\phi$ by taking its constant extension with respect to the trivializations $\Psi_1$ and $\Psi_2$.
    If $\tphi = \tPhi\vert_{(b_1,0)}$, then $\varphi = \Psi_2^{-1}\circ (\tphi\times\id_{B_1\times\Spf \bbk\dbb{u}} )\circ \Psi_1$.
    Define two connections on $(f\times\id_u)^{\ast}\cH_2$
    \begin{align*}
        \nabla_1' &\coloneqq (f\times\id_u)^{\ast}\nabla_2 = \Phi \circ\nabla_1\circ\Phi^{-1}, \\
        \nabla_2' &\coloneqq \varphi \circ \nabla_1 \circ \varphi^{-1}.
    \end{align*}
    In the trivialization $\Psi_2$ we see that $\nabla_1 '$ is framed over all $B_1$, and $\nabla_2'$ is framed only at $b_1$.
    Furthermore $\nabla_1'$ and $\nabla_2'$ are gauge equivalent under $\Phi\circ\varphi^{-1}$, and $\Phi\circ\varphi^{-1}\vert_{b_1} = \id$.
    We conclude from \cref{theorem:extension-of-framing-connection-version} that $\Phi\circ\varphi^{-1}$ is unique, then so is $\Phi$ provided that we know $\Phi\vert_{b_1\times\Spf \bbk\dbb{u}}$.
    This proves (1).

    Next we prove (2), and assume that the F-bundles are maximal.
    The framing $\nabla_1^{\fr}$ induces unique framings $\nabla_2^{\fr}$ (resp.\ $\nabla_2^{\fr\,\prime}$) on $(\cH_2,\nabla_2 )$ (resp.\ $f^{\ast} (\cH_2,\nabla_2 )$) such that in the diagram
    \[(\cH_1,\nabla_1  )\xrightarrow{(\id_{B_1} , \Phi )}  f^{\ast} (\cH_2,\nabla_2 ) \xrightarrow{(f , \id )}(\cH_2,\nabla_2 ),\]
    all the morphisms are compatible with the framings.
    Furthermore, the framing $\nabla_2^{\fr\,\prime}$ is determined by $\nabla_1^{\fr}$ and $\Phi$, hence is already known by (1).

    Let $h_1$ be a $\nabla_1^{\fr}$-flat section of cyclic vectors for $(\cH_1,\nabla_1 )$.
    Because of the compatibility of the framings, $h_1$ induces a $\nabla_2^{\fr\,\prime}$-flat section of cyclic vectors $h_2'\coloneqq \Phi (h_1)$ for $f^{\ast}(\cH_2,\nabla_2 )$,
    and a $\nabla_2^{\fr}$-flat section of cyclic vectors $h_2 \coloneqq (f^{-1}\times\id_u)^{\ast} (h_2')$  for $(\cH_2,\nabla_2 )$.
    We obtain isomorphisms $\eta_i\coloneqq \mu_{\cH_i} (\cdot )(h_i)\colon TB_i(-\log D_i)\rightarrow \cH_i\vert_{u=0}$ that fit into a commutative diagram
    \[
    \begin{tikzcd}
      TB_1 (-\log D_1) \ar[r,"df"] \ar[d,"\eta_1"] & f^{\ast} TB_2 (-\log D_2) \ar[d,"f^{\ast} (\eta_2)"] \\
      \cH_1\vert_{u=0} \ar[r, "\Phi\vert_{u=0}"] & (f\times\id_u)^{\ast}\cH_2\vert_{u=0},
    \end{tikzcd}
    \]
    where all arrows are isomorphisms.
    The maps $\eta_1$ and $\Phi\vert_{u=0}$ are already known. 
    We have $f^{\ast}(\eta_2 ) = \mu_{(f\times\id_u)^{\ast}\cH_2} (\cdot ) (h_2')$ by construction and compatibility of $(f,\id )$ with the framings.
    So $f^{\ast} (\eta_2)$ is determined by $h_1$, $\nabla_2^{\fr\,\prime}$ and $\Phi$, hence is known.
    We deduce that $df$ is determined by $h_1$, $\nabla_1^{\fr}$ and $\Phi$.
    
    Since $B_i$ are formal neighborhoods of points, the differential $d f$ determines $f$ uniquely, up to some multiplicative constants in the logarithmic directions.
    To see this, choose coordinates $(q,t) = (q_1,\dots, q_r, t_1,\dots, t_n)$ for $(B_1,D_1)$, centered at $b_1$, where $\prod_{1\leq i\leq r}q_i = 0$ is a local equation for $D_1$.
    Similarly, choose coordinates $(p,s)= (p_1,\dots, p_r, s_1,\dots, s_n)$ for $(B_2,D_2)$ centered at $f(b_1)$.
    In coordinates, the restriction of $f$ to $B_1$ is given by $f = (f_1,\dots, f_{r+n})$ where $q_i = f_i (p,s)$ and $t_j = f_{r+j} (p,s)$.
    The differential $df$ corresponds to a map of $\bbk\dbb{q,t}$-modules
    \[\Psi\colon \Gamma (B_1,f^{\ast} \Omega_{B_2}^1(\log D_2)) \rightarrow \Gamma (B_1, \Omega_{B_1}^1 (\log D_1)) ,\]
    given by the pullback of differential forms, i.e.
    \[\Psi (d\log p_i) = d\log f_i = \frac{df_i}{f_i} ,\quad \Psi (ds_j) = d f_{r+j} . \]
    We conclude the proof by integrating the differential forms.
\end{proof}

\subsection{Equivalence of F-bundles over a point} \label{subsec:F-bundle-over-a-point}

For applications in \cref{sec:app-projective-bundle}, we present some results here for the classification of framed F-bundles over a point up to gauge equivalence, see \cref{theorem:gauge-eq-F-bundle-point} and \cref{corollary:classification-framed-F-bundle-point}.

Let $(\cH ,\nabla ,\nabla^{\fr} )$ be a framed F-bundle over a point.
Fix a $\nabla^{\fr}$-flat trivialization  $\cH\simeq H\otimes_{\bbk}\bbk \dbb{u}$ and write
\[\nabla_{u\partial_u}  = u\partial_u + u^{-1}\bK + \bG , \]
with $\bK,\bG\in \End_{\bbk} (H)$.

We assume that the endomorphism $\bK$ induces a $\bbk$-vector space decomposition $H = \bigoplus_{1\leq k\leq m} H_k$ into generalized eigenspaces, and all $H_k$ have same dimensions.
Then we have a $\bbk$-vector space $H_0$ and a splitting of the fiber
\begin{equation} \label{eq:splittig-fiber-F-bundle-point}
    \iso\colon  H_0^{\oplus m}\xrightarrow{\sim} H .
\end{equation}
So we can represent endomorphisms on $H$ as $m\times m$ matrices with coefficients in $\End_{\bbk} (H_0)$.
In particular we write $\bK = (\bK_{ij} )_{1\leq i,j\leq m}$ and $\bG = (\bG_{ij} )_{1\leq i,j\leq m}$.
By construction $\bK_{ij} = 0$ if $i\neq j$ and $\bK_{ii} = \xi_i \id_{H_0} + \bN_i$ with $\xi_i\in\bbk$ and $\bN_i$ a nilpotent endomorphism.

Fix $\bc_1,\dots ,\bc_r,\bd\in \End_{\bbk} (H_0)$ such that $\bc_i$ are nilpotent endomorphisms, $[\bc_i ,\bc_j] = 0$ and $[\bd ,\bc_i] = d_i \bc_i$ for $d_i\in\bbN_{>0}$.

\begin{definition}\label{def:F-bundle-point-restricted-connection}
    We denote by $\cF ( H ,\iso , \bd , (\bc_i )_{1\leq i\leq r} )$ the space of connections $\nabla '$ on $\cH$ which, in the fixed $\nabla^{\fr}$-flat trivialization, are of the form
    \[\nabla_{u\partial_u}' = u\partial_u + u^{-1} \bK '+ (\mu '\bD + \bH') ,\]
    where
    \begin{enumerate}
        \item $\mu'\notin\bbQ_{<0}\subset\bbk$,
        \item $\bK',\bD, \bH' \in \End_\bbk(H)$,
        \item $\bK_{ij}' ,\bH_{ij}'\in\bbk [\bc_1,\dots ,\bc_r]$, and 
        \item $\bD_{ii} = \bd$ and $\bD_{ij} = 0$ for $i\neq j$.
    \end{enumerate}
\end{definition}

\begin{theorem} \label{theorem:gauge-eq-F-bundle-point}
    Let $(\cH ,\nabla , \nabla^{\fr})$ be as above.
    Assume $\nabla\in \cF ( H ,\iso , \bd , (\bc_i )_{1\leq i\leq r} )$ and let $\nabla'\in\cF ( H ,\iso , \bd , (\bc_i )_{1\leq i\leq r} )$.
    Write 
    \begin{align*}
        \nabla_{u\partial_u} &= u\partial_u + u^{-1} \bK + (\mu \bD + \bH), \\
        \nabla_{u\partial_u}' &= u\partial_u + u^{-1} \bK'+ (\mu' \bD + \bH') .
    \end{align*}
    Then $\nabla$ is gauge-equivalent to $\nabla'$ under $\bPhi (u) \in \GL (H \dbb{u} )$ with $\bPhi_{ij} (u)\in \bbk [\bc_1,\dots ,\bc_r] \dbb{u}$ if and only if the following three conditions are satisfied:
        \begin{enumerate}
            \item there exists $\phi\in\GL (H)$ with $\phi_{ij}\in\bbk [\bc_1,\dots, \bc_r]$ such that $\bK = \phi^{-1}\circ \bK'\circ\phi$,
            \item $\mu = \mu'$, and 
            \item for all $1\leq i\leq m$, $\bH_{ii} = (\phi^{-1}\circ\bH'\circ \phi )_{ii}\mod (\bc_1,\dots ,\bc_r)$.
        \end{enumerate}
Furthermore, $\bPhi$ is then uniquely determined by the initial condition $\bPhi\vert_{u=0} = \phi \mod (\bc_1,\dots ,\bc_r)$.
\end{theorem}   

The assumptions on the form of the operators allow us to work in the non-commutative subalgebra $\bbk [\bd , \bc_1,\dots, \bc_r ]\subset \End_{\bbk} (H_0)$.
We then reduce to the case of simple eigenvalues by treating the operators $\bd$, $(\bc_i)_{1\leq i\leq r}$ as formal variables.
In the simple eigenvalues case, the gauge equivalence can be constructed inductively. 

As a corollary, in the simple eigenvalue case we obtain a classification of F-bundles over a point with a fixed framing. 

\begin{corollary} \label{corollary:classification-framed-F-bundle-point}
Let $\cH\simeq H\times\bbk\dbb{u}$ be a trivialized rank $m$ vector bundle over $\bbk \dbb{u}$.
    Let $(\cH ,\nabla )$ and $(\cH ,\nabla')$ be two F-bundle structures framed in the given trivialization, and write 
    \begin{align*}
        \nabla_{u\partial_u} &= u\partial_u + u^{-1} \bK + \bG, \\
        \nabla_{u\partial_u} ' &= u\partial_u + u^{-1} \bK' + \bG'.
    \end{align*}
    Assume $\bK$ has simple eigenvalues.
    Then $(\cH ,\nabla )$ is isomorphic to $(\cH,\nabla' )$ if and only if there exists $\phi\in \GL (H)$ such that 
        \begin{enumerate}
            \item $\bK = \phi^{-1}\circ \bK'\circ\phi$, and
            \item in an eigenbasis of $\bK$, we have $(\bG)_{ii} = (\phi^{-1}\circ\bG'\circ\phi )_{ii}$ for $1\leq i\leq m$.
        \end{enumerate}  
    Furthermore, the gauge equivalence is uniquely and explicitly determined by the initial condition $\phi$ at $u=0$.
\end{corollary}

\begin{proof}
    The choice of an eigenbasis for $\bK$ produces a splitting $\iso\colon \bbk^{\oplus m}\xrightarrow{\sim} H$ as in \eqref{eq:splittig-fiber-F-bundle-point}.
    Since there are no nilpotent operators in $\End_{\bbk} (\bbk )\simeq \bbk$, and this algebra is commutative, the content of \cref{def:F-bundle-point-restricted-connection} becomes empty, and the corollary is just a reformulation of \cref{theorem:gauge-eq-F-bundle-point} in this special case. 
\end{proof}

\begin{proof}[Proof of \cref{theorem:gauge-eq-F-bundle-point}]
    Let $R_0 = \bbk\dbb{c_1,\dots, c_r}$ and $R = \bbk [\deg ]\dbb{c_1 ,\dots, c_r}$ where $\lbrace (c_i)_{1\leq i\leq r} ,\deg \rbrace$ are formal variables satisfying the commutation relations $[c_i,c_j ] = 0$ and $[\deg ,c_i] = d_ic_i$.
    There is a specialization map $R\rightarrow \bbk [\bd,\bc_1,\dots, \bc_r]$.
    Using $\iso$ we also have a specialization map 
    \[\Mat (m\times m,R)\rightarrow \End_{\bbk} (H). \]
    By the definition of $\cF (H,\iso ,\bd ,(\bc_i)_{1\leq i\leq r} )$, the connections $\nabla_{u\partial_u}$ and $\nabla_{u\partial_u}'$ lift to differential operators of the form
    \begin{align*}
        &u\partial_u + u^{-1} K + \mu D + H ,\\
        &u\partial_u + u^{-1} K' + \mu' D + H' ,
    \end{align*}
    with $K, K', H, H' \in \Mat (m\times m ,R_0)$ and $D = \deg \cdot \Id_m$.
    A gauge equivalence $\bPhi$ as in the theorem also lifts along the specialization map, so we have reduced the problem to finding $\Phi (u)\in \GL (m,R_0\dbb{u} )$ such that 
    \begin{equation} \label{eq:gauge-eq-F-bundle-point}
        \Phi^{-1} (u\partial_u + u^{-1} K + \mu D + H )\Phi = u\partial_u + u^{-1} K' + \mu' D + H' .
    \end{equation}
    The conditions (1)-(3) also lift under the specialization map, so we are left to prove the following lemma.
\end{proof}

    \begin{lemma}\label{lemma:gauge-eq-F-bundle-point}
        There exists a gauge equivalence $\Phi(u)\in \GL (m,R_0\dbb{u})$ solving \eqref{eq:gauge-eq-F-bundle-point} if and only if there exists $Q\in \GL (m,R_0)$ such that
        \begin{enumerate}[label=(\alph*)]
            \item $K = Q^{-1} K' Q$,
            \item $\mu = \mu'$, and
            \item $H_{ii} = (Q^{-1} H'Q)_{ii}\mod (c_1,\dots, c_r)$.
        \end{enumerate}
        In this case, $\Phi (u)$ is uniquely determined by the initial condition $\Phi\vert_{u=0} = Q \mod (c_1,\dots ,c_r)$.
    \end{lemma}
\begin{proof}
By construction of the splitting \eqref{eq:splittig-fiber-F-bundle-point} and \cref{def:F-bundle-point-restricted-connection}(2), the matrix $K$ is diagonal, and $K_{ii} = \xi_i \mod (c_1,\dots, c_r)$, where $\lbrace \xi_1,\dots, \xi_m\rbrace$ are the distinct eigenvalues of $\bK$.
In particular $K$ has simple eigenvalues, so $\ad_{K} = [K,\cdot ]$ has kernel given by diagonal matrices, and image given by matrices with vanishing diagonal.

Let us first prove that the conditions (a)-(c) are sufficient. 
Fix $Q\in \GL (m, R_0)$ satisfying (a) and (c).
We are looking for $\Phi (u)$ such that $\Phi\vert_{u=0} = Q\mod (c_1,\dots, c_r)$ solving \eqref{eq:gauge-eq-F-bundle-point}.
Write $\Phi(u) = Q P(u)$ with $P(u) = \sum_{k\geq 0} P_k u^k$ satisfying $P_0 = \Id_m \mod (c_1,\dots, c_r)$.
 
    Equation \eqref{eq:gauge-eq-F-bundle-point} then reduces to the system
    \begin{align}\label{eq:gauge-eq-F-bundle-point-initial-condition}
            [K, P_0] &= 0 ,\\
            \label{eq:gauge-eq-F-bundle-point-recursion}
            [K, P_{k+1}] &= \varphi (P_k) - k P_k,
    \end{align}        
    where 
    \[\varphi \colon M\mapsto M (\mu' Q^{-1}D Q + Q^{-1}H'Q ) - (\mu D + H ) M .\]
    Before analyzing the existence of solutions, let us rewrite \eqref{eq:gauge-eq-F-bundle-point-recursion} in order to isolate the terms involving the non-commutative variable $\deg$.
    Define the $\bbk$-linear operator $\Eu (\cdot )\coloneqq [\deg , \cdot ]$ on $R$.
    The commutations relations in $R$ give $ \Eu (\cdot ) = \sum_{1\leq i\leq r}d_ic_i\partial_{c_i}$.
    For $M\in \Mat (m\times m, R)$, we write $\Eu (M) \coloneqq (\Eu (M_{ij} ))_{1\leq i,j\leq m}$.
    We have $\Eu (M) = [D,M]$, so
    \begin{align*}
        \varphi (M) &= M(\mu' D + Q^{-1} \Eu (Q) + Q^{-1} H' Q) - (\mu D + H ) M \\
        &= \mu' MD - \mu DM + M (Q^{-1}\Eu (Q) + Q^{-1} H' Q) - H M \\
        &= (\mu' - \mu ) MD - \mu \Eu (M) + M (Q^{-1}\Eu (Q) + Q^{-1} H' Q) - HM.
    \end{align*}
    Since $\mu  =\mu'$, the term involving $D$ vanishes. 
    
    We now prove by induction on $k$ the following: there exists a unique sequence of matrices $(P_0,\dots , P_k)$ such that 
    (i) $P_0 = \Id_m \mod (c_1,\dots ,c_r)$ and $[K,P_0] = 0$, (ii) $(P_{\ell} , P_{\ell +1} )$ solves \eqref{eq:gauge-eq-F-bundle-point-recursion} for $0\leq \ell \leq k-1$, and (iii) $\varphi (P_k) - k P_k\in \im \ad_K$.
    
    We construct $P_0$ satisfying (i), (ii) and (iii).
    The condition $[K , P_0] = 0$ implies that $P_0$ is a diagonal matrix, $P_0 = \Diag (\delta_{1} , \dots ,\delta_{n} )$.
    The initial condition $P_0 = \Id_m\mod (c_1,\dots, c_r)$ gives $\delta_{i} = 1 \mod (c_1,\dots ,c_r)$.
    To ensure that we can solve the recursion for $P_1$, we need $\varphi (P_0)_{ii} = 0$ for all $i$.
    This provides the relation
    \begin{equation}\label{eq: solve diagonals}
        \mu \Eu (\delta_{i} ) =\alpha_{i} \delta_{i} ,
    \end{equation}
    for all $i$, where $\alpha_{i} = (Q^{-1} \Eu (Q) + Q^{-1} H'Q - H)_{ii}$.
    For any $x\in R_0$, we have $\Eu (x)\in (c_1,\dots, c_r)R_0$.
    Together with Condition (c), this implies that $\alpha_{i} = 0 \mod (c_1,\dots , c_r)$.
    We can then solve for $\delta_{i}$ order by order in $(c_1,\dots ,c_r)$ and determine $P_0$ uniquely from the initial condition $P_0 = \Id_m \mod (c_1,\dots , c_r)$.
    Note that the condition on $\mu$ in \cref{def:F-bundle-point-restricted-connection} ensures that we obtain a recursion that we can solve.
    
    Let $k\geq 1$, and assume $(P_0,\dots, P_{k-1} )$ are constructed. 
    The existence of a matrix $P$ such that $[K,P] = \varphi (P_{k-1} ) - (k-1) P_{k-1}$ is guaranteed by Condition (iii) of the induction hypothesis. 
    The matrix $P$ is determined up to a diagonal matrix.
    We first prove that for any choice of $P$, there exists a unique diagonal matrix $\Delta$ such that $\varphi (P+\Delta) - k (P+\Delta )\in \im \ad_K$, i.e.\ has vanishing diagonal. 
    Let $\Delta = \Diag (\delta_{1} ,\dots ,\delta_{n} )$ be a diagonal matrix, the vanishing of the $i$-th diagonal term of $\varphi (P + \Delta ) - k (P + \Delta)$ is equivalent to an equation of the form
    \begin{align}\label{eq:F-bundle-point-recursion-diag-matrix}
        \mu \Eu (\delta_{i} ) + k\delta_{i} = \alpha \delta_{i}  + \beta,
    \end{align}
    with 
    \begin{align*}
        \alpha &= (Q^{-1}\Eu (Q))_{ii} + (Q^{-1} H' Q)_{ii} - H_{ii} ,\\
        \beta &= (\varphi (P) - kP)_{ii}.
    \end{align*}
    As in the initial step ($k=0$), we have $\alpha = 0 \mod (c_1,\dots, c_r)$.
    Since $\delta_{i}$ is a power series in $(c_1,\dots, c_r)$, \eqref{eq:F-bundle-point-recursion-diag-matrix} provides a recursion relation on the coefficients of $\delta_{i}$.
    Since $k\geq 1$ the constant term of $\delta_{i}$ is uniquely determined by looking at the equation modulo $(c_1,\dots ,c_r)$, where it gives $k\delta_{i}  =\beta \mod (c_1,\dots , c_r)$.
    The other coefficients are then uniquely determined inductively.
    The condition on $\mu$ in \cref{def:F-bundle-point-restricted-connection} ensures that we obtain a recursion that we can solve, thus $\delta_{i}$ is uniquely determined from $P$.
    We have proved the existence of a matrix $P_k$ satisfying Conditions (ii) and (iii) of the induction. 
    Now we prove uniqueness. 
    Let $P_k$ and $\widetilde{P_k}$ be two matrices satisfying (ii) and (iii).
    In particular, they are solutions of the equation $[K,P] = \varphi (P_{k-1} ) - (k-1)P_{k-1}$, so there exists a diagonal matrix $\Delta$ such that $P_k = \widetilde{P_k} + \Delta$.
    Condition (iii) gives $\varphi (\widetilde{P_k} + \Delta ) - k (\widetilde{P_k} + \Delta )\in\im\ad_K$.
    Since $\widetilde{P_k}$ already satisfies (iii) and $\Delta$ is diagonal, we deduce from the uniqueness in the previous paragraph that $\Delta = 0$.
    Hence $P_k = \widetilde{P_k}$, concluding the induction.

    Now we prove Conditions (a)-(c) assuming that there exists $\Phi (u) = \sum_{k\geq 0} P_k u^k\in \GL (m,R_0\dbb{u})$ solving \eqref{eq:gauge-eq-F-bundle-point}.
    In particular $P_0\in \GL (m,R_0)$.
    Multiplying \eqref{eq:gauge-eq-F-bundle-point} on the left by $\Phi$ and isolating the $u^k$ term, we obtain for $k=-1$ and $k\geq 0$ respectively:
    \begin{align*}
       K P_0 & =P_0 K',\\
        k P_k+K P_{k+1}+(\mu D)P_k+HP_k &= P_{k+1}K'+P_k(\mu' D)+P_kH'.
    \end{align*} 
    Let $Q=P_0^{-1}$, it satisfies Condition (a). 
    For any $1\leq i\leq r$ we have $\deg \cdot\, c_i = c_i\cdot \deg +d_i c_i$.
    By comparing the coefficient of the formal variable $\deg$ we obtain $\mu=\mu'$, verifying Condition (b). 
    Looking at the $u^0$ term, using $K = P_0 K'P_0^{-1}$ and modding out $(c_i)_{1\leq i\leq r}$, we obtain
    \begin{equation*}
        K (P_1P_0^{-1})-P_1P_0^{-1}K+H = P_0H'P_0^{-1} \mod (c_1,\dots, c_r).
    \end{equation*}
    Since $[K, P_1(P_0)^{-1} ]$ has vanishing diagonal, Condition (c) follows.       
\end{proof}

\section{Application: quantum cohomology of projective bundle}
\label{sec:app-projective-bundle}

In this section, we study the decomposition of the maximal A-model F-bundle associated to a projective bundle.
We prove the existence of the decomposition when restricting the F-bundle to a point, as well as the uniqueness of the decomposition (\cref{thm:uniqueness-projective-bundle-u=0,thm:uniqueness-projective-bundle}).
In \cref{subsec:dec-blowup}, we state the analogous results in the case of a blowup of algebraic varieties (\cref{thm:uniqueness-blowup-u=0,thm:uniqueness-blowup}).

Let $X$ be a smooth complex projective variety of dimension $d$, $V\rightarrow X$ a vector bundle of rank $m$ on $X$, $P \coloneqq \bbP (V)$ the associated projective bundle of lines in $V$, and write $\pi\colon P\rightarrow X$.
We fix an ample divisor class $\omega_X\in H^2 (X,\bbZ )$, and a homogeneous basis $\lbrace T_i \rbrace_{0\leq i\leq N}$ of $H^{\ast} (X ,\bbQ )$ extending $\lbrace \1 , \omega_X \rbrace$.

\subsection{\texorpdfstring{A-model F-bundle of $P$ at the limiting point}{A-model F-bundle of P at the limiting point}} \label{subsec:A-model-F-bundle-projective-bundle}

We have the following classical decomposition of the cohomology of $P$, as a special case of Leray-Hirsch theorem (see \cite[Theorem 4D-1]{Hatcher_Algebraic_topology}).

\begin{proposition} \label{prop:splitting_proj}
	Let $h\coloneqq c_1(\cO_P(1))$.
	We have the splitting isomorphism of cohomology groups
	\begin{equation}\label{eq:splitting_proj}
	    \iso\colon H_\spl \coloneqq \bigoplus_{i=0}^{m-1} H^{\ast} (X , \bbQ)[-2i] \xrightarrow{\sum h^{i}\cup \pi^*} H^{\ast}(P,\bbQ).
	\end{equation}
\end{proposition}

\begin{lemma} \label{lem:K_P}
	We have
	\[K_P = \pi^*K_X - m h - \pi^*c_1 V.\]
\end{lemma}
\begin{proof}
	It follows from the relative Euler sequence
	\[0\to\Omega_{P/X}\to\cO_P(-1)\otimes\pi^* V^\vee\to\cO_P\to 0\]
	that
	\[K_{P/X} = -mh - \pi^*c_1 V.\]
	Hence
	\[K_P = \pi^*K_X+K_{P/X} = \pi^*K_X - m h - \pi^*c_1 V. \qedhere\]
\end{proof}

Recall that we fixed an ample class $\omega_X$ on $X$.
Let $\omega_P \coloneqq \pi^{\ast} \omega_X$.

\begin{lemma}
    The class $\omega_P$ is nef and satisfies \cref{ass:finitely_many_beta}.
\end{lemma}

\begin{proof}
    Since $\omega_X$ is ample, its pullback $\omega_P$ is nef.
    Furthermore, there exists $\varepsilon > 0$ such that $\omega_X + \varepsilon (c_1 T_X + c_1 V)$ is ample.
    Then, by \cref{lem:K_P}, we have $\omega_P + \varepsilon c_1P = \pi^{\ast} (\omega_X +\varepsilon (c_1T_X + c_1V)) + \varepsilon mh$.
    It is ample, since it is the sum of a nef class and an ample class (\cite[Corollary 1.4.10]{Lazarsfeld_Positivity_in_algebraic_geometry_I}).
    We conclude by \cref{lem:finitely_many_beta}.
\end{proof}

Using the homogeneous basis $\lbrace T_i\rbrace_{0\leq i\leq N}$ of $H^{\ast} (X,\bbQ )$, we produce a homogeneous basis 
\[\lbrace \pi^{\ast}(T_i) h^j ,\;  0\leq i\leq N ,\, 0\leq j\leq m-1 \rbrace\]
of $H^{\ast} (P,\bbC )$ extending $\omega_P$.
We denote by $\lbrace t_{i,j}\rbrace$ the induced linear coordinates on $H^{\ast} (P,\bbC )$.

Let $(\cH ,\nabla )/B$ denote the maximal A-model F-bundle of $P$ constructed from $\omega_P$, with base point $0\in H^{\ast} (P,\bbC )$ (see \cref{example:modified-quantum-F-bundle}).
Write $(q, t=\lbrace t_{i,j},\; (i,j)\neq (1,0)\rbrace )$ for the coordinates on $B$.
Let $b$ denote the closed point of $B$, given by $q=0, t=0$, which we refer to as the limiting point in this section. Let $\bK_{\lim}$ and $\bG_{\lim}$ denote the restrictions of the operators $\bK$ and $\bG$ at the limiting point (see \cref{definition:A-model-F-bundle}).

Let us compute the matrices of $\bK_{\lim}$ and $\bG_{\lim}$ under the splitting $\iso$ in \eqref{eq:splitting_proj}.

We have
\[\bG_{\lim} = \begin{pmatrix}
    \bG_X - \frac{m-1}{2} &  & & &  \\
    & & \bG_X - \frac{m-3}{2} &   & \\
    & & & \ddots & \\
     & & & & \bG_X + \frac{m-1}{2}
\end{pmatrix} , \]
and $\bK_{\lim}$ is computed in the following proposition.

\begin{proposition} \label{prop:Klim_proj}
	The operator $\bKlim$ on $H^{\ast}(P,\bbC)$ has the following matrix with respect to the splitting in \eqref{eq:splitting_proj}:
	\[\bK_{\lim}=\begin{pmatrix}
	c_1 T_X+c_1 V & & & & m(1- c_m V)\\
	m & c_1 T_X+c_1 V & & & -m c_{m-1} V\\
	& m & \ddots & & \vdots\\
	& & \ddots & c_1 T_X+c_1 V & -m c_2 V\\
	& & & m& c_1 T_X+c_1 V - m c_1 V
	\end{pmatrix}.\]
\end{proposition}
\begin{proof}
    Consider four operators $K_1, \dots, K_4$ on $H_\spl$ such that for $\gamma \in H^*(P,\bbC) \simeq H_\spl$, we have
    \begin{enumerate}
		\item $K_1 (\gamma) = \pi^{\ast} (c_1T_X)\cup\gamma$,
		\item $K_2 (\gamma ) = h\cup\gamma$,
		\item $K_3 (\gamma ) = \pi^{\ast} c_1V \cup \gamma$, and
		\item $K_4 (\gamma) = p_{\ast}q^{\ast}\gamma $, where $p,q\colon P\times_X P\rightarrow P$ are the projections.
	\end{enumerate}
	By \cref{lem:K_P}, the classical multiplication by $c_1 T_P$ has matrix $K_1 + m K_2 + K_3$.

    The non-classical part of $\bK_{\lim}$ is expressed in terms of 3-pointed Gromov-Witten invariants of the form $\langle c_1P ,\gamma_1, \gamma_2 \rangle_{0,3}^{\beta}$ for an effective curve class $\beta\neq 0$ such that $\beta\cdot\omega_P = 0$ and cohomology classes $\gamma_1,\gamma_2\in H^{\ast} (P,\bbC )$.
    Fix such a $\beta$, by the projection formula, we have $\beta\cdot \omega_P = (\pi_{\ast} \beta )\cdot \omega_X$.
    Since $\omega_X$ is ample, this implies that $\pi_{\ast}\beta = 0$, i.e.\ $\beta = \delta [L]$ for $[L]$ the class of a line in a fiber of $\pi$ and $\delta\in\bbN_{>0}$ ($\delta = 0$ gives the classical contribution).
    By the divisor axiom and \cref{lem:K_P}, we have
    \[\langle c_1P ,\gamma_1, \gamma_2 \rangle_{0,3}^{\beta} = \left( \beta \cdot  c_1P\right) \langle \gamma_1 ,\gamma_2 \rangle_{0,2}^{\beta} =\delta m\langle \gamma_1 ,\gamma_2 \rangle_{0,2}^{\beta}. \]
    Let $M \coloneqq \widebar{\cM}_{0,2} (P,\delta [L] )$ denote the moduli stack of $2$-pointed rational stable maps of class $\beta$.
    By the Riemann-Roch formula, the virtual dimension $\dim_{\vir} M$ of $M$ is equal to $\dim P -3 +\int_{\beta} c_1(P) +2 = d-2+ m(\delta + 1)$.
    Since $\beta$ is a fiber class, the evaluation map 
    \[\ev_1\times\ev_2\colon M\rightarrow P\times P\]
    factors through 
    \[P\times_X P \subset P\times P.\]
    In order to have nonzero counts, we need $\dim_{\vir} M\leq \dim P\times_X P$ which implies that $\delta =1$, i.e.\ the curve class can only be $[L]$.
    We then have an isomorphism 
    \[\ev_1\times\ev_2\colon M\overset{\sim}{\rightarrow} P\times_X P\subset P\times P.\]
    In particular, $M$ is smooth, so $[M]^{\vir} = [M]$.
    Under this isomorphism, the operator 
    \[\gamma\mapsto \ev_{1,\ast}\left( \ev_2^{\ast}\gamma\cup [M]^{\vir} \right) = \ev_{1,\ast} \ev_2^{\ast}\gamma \] 
    is equal to $mK_4$.
    Therefore, the non-classical contribution to $\bKlim$ is $mK_4$.
        
    We obtain
	\begin{equation} \label{eq:Klim_proj}
		\bKlim = \iso\circ (K_1 + m K_2 + K_3 + m K_4 )\circ \iso^{-1}.
	\end{equation}
	Now let us calculate the four matrices $K_1,\dots, K_4$.
	For any $\alpha_i\in H^{\ast}(X ,\bbC)[-2i]$, we have
	\[\pi^*(c_1 T_X) \cup (h^i\cup \pi^*\alpha_i) = h^i\cup\pi^*(c_1 T_X\cup\alpha_i),\]
	hence $K_1 = (c_1 T_X\cup )\cdot\id_{H_{\spl}}$.
	Similarly, we have that $K_3 = (c_1 V\cup ) \cdot\id_{H_{\spl}}$.
	For $i=0,\dots,m-1$, we have
	\[h\cup (h^i\cup \pi^*\alpha_i) = h^{i+1}\cup \pi^*\alpha_i.\]
	When $i=m-1$, by \cite[Eq. (20.6)]{Bott_Differential-forms-algebraic-topology} we have
	\[h\cup (h^{m-1}\cup\pi^*\alpha_i) = h^m\cup \pi^*\alpha_{m-1} = -\sum_{j=0}^{m-1} h^j \cup \pi^*(c_{m-j} V\cup\alpha_{m-1}).\]
	So
	\[K_2=\begin{pmatrix}
	&  & &  & -c_m V\\
	1 & & & & -c_{m-1} V\\
	& \ddots & & & \vdots \\
	& & 1 &  & -c_2 V\\
	& & & 1 & -c_1 V
	\end{pmatrix}.\]
	For any $\alpha_i\in H^{\ast}(X ,\bbC)[-2i]$, $i=0,\dots,m-1$, since $\pi\circ p = \pi\circ q$, by the projection formula we have
        \[
        p_{\ast}q^{\ast} (h^i \cup \pi^{\ast}\alpha_i ) = 
        p_{\ast} (q^{\ast} (h^i) \cup q^{\ast} \pi^{\ast}\alpha_i )
        = p_{\ast} (q^{\ast} (h^i) \cup p^{\ast} \pi^{\ast}\alpha_i )
        = p_{\ast} q^{\ast} (h^i ) \cup \pi^{\ast} \alpha_i.
    \]
    Since $p_{\ast}q^{\ast} (h^i )\in H^{2(i -(m-1))} (P ,\bbC)$, it vanishes unless $i = m-1$, in which case it is equal to the identity.
    We deduce that the matrix of $K_4$ has only one nonzero block: the top-right corner, which is $\id_{H^{\ast}(X,\bbC)}$.
    
	Substituting the above computations into \eqref{eq:Klim_proj}, we conclude the proof.
\end{proof}

\subsection{\texorpdfstring{Decomposition of $\bK_{\lim}$}{Decomposition of Klim}}
In this subsection, we study the generalized eigenspaces of $\bK_{\lim}$.
We will consider the commutative subalgebra $\bbC [\bt ,\bc_1,\dots, \bc_m]$ of $\End_{\bbC} (H^{\ast}(X,\bbC))$ generated by the commuting nilpotent operators 
\[\bt \coloneqq  c_1T_X\cup \quad \mathrm{and} \quad \bc_i \coloneqq c_iV \cup \;(1\leq i\leq m) .\]
Let $\bd \coloneqq \bG_X = \frac{1}{2} (\deg_X - \dim X)$, where $\deg_X (\alpha ) = i\alpha$ for $\alpha \in H^{i} (X,\bbC )$.
We have the commutation relations
\begin{equation}
    [\bd ,\bt ] =  \bt ,\quad [\bd , \bc_i ] = i\bc_i .
\end{equation}

\begin{lemma} \label{lemma:jordan-dec-Klim}
\begin{enumerate}[wide]
    \item There exists $\phi = (\phi_{ij} )\in \GL (H_{\spl} )$ with entries $\phi_{ij}\in \bbC [\bc_1,\dots, \bc_m]$, and $\blambda_i = \lambda_i\cup \in \bbC [\bc_1 ,\dots ,\bc_m]\subset \End_{\bbC} (H^{\ast}(X,\bbC))$ such that 
    \[\bK_{\spl}\coloneqq \phi^{-1} \bK_{\lim} \phi = \begin{pmatrix}
        \bt + \bc_1 & & \\
        & \ddots & \\
        & & \bt + \bc_1
    \end{pmatrix} + m \begin{pmatrix}
        \blambda_1 & & \\
         & \ddots & \\
          & & \blambda_m
    \end{pmatrix}.\]
    \item Up to reordering the blocks, for $1\leq i\leq m$ we have 
    \[\blambda_i = \xi^{i-1}  - \frac{\bc_1}{m} \mod (\bc_1^2 ,\bc_2,\dots ,\bc_m) ,\quad \xi = e^{\frac{2\pi \iunit}{m}} .\]
\end{enumerate}
    In particular the $i$-th diagonal block of $\bK_{\spl}$ is the cup-product with an element in $H^{\ast} (X ,\bbC )$ whose $H^2$-component is $c_1 T_X$. 
\end{lemma}

\begin{proof}
As an element of $\Mat (m\times m, \bbC [\bc_1,\dots ,\bc_m] )$, we have $\bK_{\lim} = (\bt + \bc_1 ) \Id_m + mM$, where $M$ is the companion matrix
    \begin{equation}\label{eq:companion-matrix}
        M=\begin{pmatrix}
	   0 & & & & 1-\bc_m \\
	   1 & &  & & -\bc_{m-1} \\
	   & 1 & & & -\bc_{m-2} \\
	   & & \ddots & & \vdots \\
	   & & & 1 & -\bc_1
    \end{pmatrix}.
    \end{equation}
    The characteristic polynomial of $M$ is $\lambda^m + \sum_{i=1}^{m-1} \bc_{m-i} \lambda^i + (\bc_m - 1)$.
    Modulo $(\bc_1,\dots ,\bc_m )$ this polynomial has simple roots given by $m$-th roots of unity. 
    Since it is monic, we can lift these roots to $\bbC [\bc_1,\dots ,\bc_m]$ by solving the equation order by order.
    (1) follows. 

    For (2), the characteristic polynomial of $M$ modulo $(\bc_1^2, \bc_2,\dots ,\bc_m )$ is
    \[\lambda^m + \bc_1 \lambda^{m-1} -1 = \left( \lambda + \frac{\bc_1}{m} \right)^m -1 .\]
    We deduce that $m\blambda_i = m e^{\frac{2\pi\iunit (i-1)}{m}}- \bc_1$ modulo those classes, proving (2). 
\end{proof}

\cref{lemma:jordan-dec-Klim} implies that $\bK_{\lim}$ has $m$ generalized eigenspaces, all isomorphic to $H^{\ast} (X,\bbC )$, matching the setup of \cref{subsec:F-bundle-over-a-point}.
The splitting considered in \eqref{eq:splittig-fiber-F-bundle-point} is given by the modified isomorphism 
\begin{equation}\label{eq:modified-splitting-projective-bundle}
    \iso\circ\phi^{-1}\colon H_{\spl} \xrightarrow{\sim} H^{\ast}(P ,\bbC) .
\end{equation}

We will use the following lemma to check Condition (c) of \cref{theorem:gauge-eq-F-bundle-point}.

\begin{lemma} \label{lemma:vanishing-diagonal}
    Let $H = \Diag (\mu_1  ,\cdots , \mu_m  )\in \GL (H_{\spl} )$ be a block diagonal matrix with scalar entries.
    Let $\phi = (\phi_{ij})\in \GL (H_{\spl} )$ be as in \cref{lemma:jordan-dec-Klim}.
    Assume that $\sum_{1\leq j\leq m} \mu_j = 0$.
    Then $(\phi^{-1}\circ H\circ \phi )_{ii} = 0$ for all $1\leq i\leq m$.
\end{lemma}

\begin{proof}
    As in the previous lemma, we view $\phi$ and $H$ as elements in $\Mat (m\times m ,\bbC \dbb{\bc_1,\dots , \bc_m} )$.
    By construction, $\phi$ diagonalizes the companion matrix $M\in \Mat (m\times m ,\bbC [\bc_1,\dots ,\bc_m] )$ from \eqref{eq:companion-matrix}.
    Let $\Lambda = \Diag (\blambda_1,\dots ,\blambda_m )$.
    By construction we have $M\phi = \phi\Lambda$. For every $1\leq i\leq m$, we deduce
    \[\phi_{mi} = \blambda_i \phi_{1i} ,\; \phi_{1i} = \blambda_i \phi_{2i} , \; \phi_{2i} = \blambda_i \phi_{3i} , \cdots , \phi_{m-1,i} = \blambda_i \phi_{mi} .\]
    Similarly, for $\psi\coloneqq \phi^{-1}$ we have that $\Lambda\psi = \psi M$, and we obtain for all $1\leq i\leq m$
    \[\blambda_i \psi_{i1} = \psi_{i2} ,\; \blambda_i \psi_{i2} = \psi_{i3} ,\dots ,\ \blambda_i \psi_{i,m-1} = \psi_{im} ,\; \blambda_i \psi_{im} = \psi_{i1} .\]
    In particular for $1\leq i\leq m$, we have
    \[\psi_{i1} \phi_{1i} = \psi_{i2} \phi_{2i} =  \cdots = \psi_{im} \phi_{mi}.\]
    We deduce 
    \[
        (\phi^{-1}\circ H \circ \phi)_{ii} = \sum_{1\leq j\leq m} \psi_{ij} (H)_{jj} \phi_{ji} 
        = \psi_{i1} \phi_{1i} \sum_{1\leq j\leq m} \mu_j = 0.\qedhere
    \]
\end{proof}

\begin{remark}
    The automorphism $\phi\mod (\bc_1,\dots ,\bc_m)$ gives the initial condition for the gauge equivalence in \cref{prop:coordinates_of_b}.
    Since it diagonalizes the (block) circulant matrix $M\mod (\bc_1,\dots ,\bc_m )$ it can be chosen to be the matrix
    \[Q = \frac{1}{\sqrt{m}}\begin{pmatrix}
        1 & \xi^{-1}  & \cdots & \xi^{-(m-1)} \\
        1 & \xi^{-2} & \cdots & (\xi^{-2} )^{m-1}  \\
        \vdots & & & \vdots \\
        1 & \xi^{-(m-1)} & \cdots & (\xi^{-(m-1)})^{m-1}
    \end{pmatrix} .\]
\end{remark}

\begin{example}[Trivial bundle case]
    If $V = \cO_X^{\oplus m}$ is a trivial vector bundle, then $\bc_i = 0$ for $1\leq i\leq m$.
    In particular, we have $\lambda_i = \xi^{i-1}$, where $\xi = e^{\frac{2\pi i}{m}} $.

\end{example}

\begin{example}[$\bbP^1$-bundle case]
Let $V$ be a rank $2$ bundle over $X$ of dimension $d$.
Then the classes $(\lambda_1 ,\lambda_2)$ are obtained by solving the quadratic equation
\[\lambda^2 + \bc_1 \lambda + \bc_2 - 1 = 0,\]
where $\bc_i$ is the cup product with $c_i V$.
Since $(\bc_1^2)^{\frac{d}{2}} = (\bc_2)^{\frac{d}{2}} = 0$,
the discriminant $\Delta = \bc_1^2 - 4\bc_2 +4$ admits a square-root in $\bbC [\bc_1 ,\bc_2 ]$ given by
\begin{align*}
    \sqrt{\Delta} &= 2\sqrt{1 + \frac{\bc_1^2}{4} - \bc_2} 
    = 2\left( 1 + \sum_{1\leq n\leq \frac{d}{2}} \binom{1/2}{n} \left( \frac{\bc_1^2}{4} - \bc_2 \right)^n \right).
\end{align*}
Using the quadratic formula, we obtain the roots ($i=1,2$)
\begin{align*}
    \lambda_i &= (-1)^{i-1} - \frac{c_1V}{2} + (-1)^{i-1} \sum_{1\leq n\leq  \frac{d}{2}} \binom{1/2}{n}  \left( \frac{(c_1V)^2}{4} - c_2V \right)^n.
\end{align*}
\end{example}

\subsection{Uniqueness of the decomposition} \label{sec:uniqueness_projective_bundle}

In this subsection, we prove the uniqueness of the decomposition of the maximal A-model associated to a projective bundle, as well as its existence at the limiting point (\cref{thm:uniqueness-projective-bundle-u=0,thm:uniqueness-projective-bundle}).
We will consider a maximal A-model F-bundle $(\cH', \nabla')$ of $X' \coloneqq \coprod_{i=1}^m X$ with a shifted base point, and use \cref{theorem:gauge-eq-F-bundle-point} to construct a gauge equivalence between the F-bundle $(\cH,\nabla)$ of $P$ and $(\cH', \nabla')$ over the base points.
The uniqueness results will follows from \cref{theorem:gauge-eq-F-bundle-point} and the extension of framing theorem.

We have 
\begin{equation}\label{eq:cohomology-disjoint-union}
    H^{\ast}(X' ,\bbQ ) \xrightarrow{\sim} \bigoplus_{i=1}^{m} H^{\ast} (X ,\bbQ ) .
\end{equation}
Let $\omega'\in H^2 (X' ,\bbQ )$ denote the class corresponding to $(\omega_X, \dots ,\omega_X)$ under \eqref{eq:cohomology-disjoint-union}, it is ample so \cref{ass:finitely_many_beta} is satisfied.

Fix a homogeneous basis of $H^2 (X' ,\bbQ )$ extending $\omega '$.
Complete it to a homogeneous basis of $H^{\ast} (X',\bbQ )$ by adding the elements $\lbrace T_i ,\; \deg T_i \neq 2\rbrace$ in each copy of $H^{\ast} (X ,\bbQ )$.

Let $\Delta (a)\in H^{\ast} (X ' ,\bbC )$ be a cohomology class at which the quantum product is well-defined.
We produce $(\cH',\nabla' ) /B'$, the maximal A-model F-bundle of $X'$ associated to $\omega'$ with base point $\Delta (a)$ as in \cref{example:modified-quantum-F-bundle}.
Let $(q,t)$ denote the coordinates on $B'$, and let $b'$ denote the closed point of $B'$, given by $t=0, q=0$, which we refer to as the limiting point for $X'$.

Using the last observation of \cref{lemma:jordan-dec-Klim}, we will interpret $\bK_{\spl}$ as the $\bK$-operator of $(\cH' ,\nabla ' )$ for certain values of $\Delta (a)$.

For $i\in\lbrace 1,\dots, m\rbrace$ and $j$ such that $\deg T_j\neq 2$, we denote by $a_{i,j}$ the coordinate of $\Delta (a)$ along the basis element $T_j$ in the $i$-th copy of $H^{\ast} (X,\bbC )$ in $H^{\ast} (X' ,\bbC )$.

\begin{theorem} \label{thm:uniqueness-projective-bundle-u=0}\label{prop:coordinates_of_b}
There exists an F-bundle isomorphism 
\[\Phi(u)\colon (\cH  ,\nabla )\vert_b \rightarrow (\cH' ,\nabla' )\vert_{b'},\]
whose components $\Phi_{ij}$ (as power series in $u$) are given by the cup-product with elements in $H^{\ast} (X,\bbC )$ if and only if the coordinates of the base point $\Delta (a)$ satisfy 
\begin{equation}\label{eq:coordinate-base point}
\sum_{j\colon\deg T_j\neq 2} \frac{\deg T_j - 2}{2} a_{i,j} T_j = c_1V + m\lambda_i ,
\end{equation}
where $\lambda_i$ was defined in \cref{lemma:jordan-dec-Klim}.

Furthermore, in this case $\Phi$ is uniquely and explicitly determined by the $H^0$-components of $\Phi_{ij}\vert_{u=0}$, and $\Delta (a)$ is uniquely determined by \eqref{eq:coordinate-base point}, up to a shift in $\bigoplus_{i=1}^m H^2 (X,\bbC )$.
\end{theorem}

\begin{proof}
    The bundles $\cH\vert_b$ and $\cH'\vert_{b'}$ are trivial by definition, their fibers are identified with $H_{\spl}$ through \eqref{eq:splitting_proj} and \eqref{eq:cohomology-disjoint-union}, and the the connections $\nabla$ and $\nabla'$ are framed. 
    We use \cref{theorem:gauge-eq-F-bundle-point} to prove the proposition. 

    The matrices of $\bK_{\lim}$, $\bG_{\lim}$ were computed in \cref{subsec:A-model-F-bundle-projective-bundle}.
    Write $\nabla_{u\partial_u}\vert_{b} = u\partial_u -u^{-1} \bK_{\spl} + \bG_{\spl}$, we have
    \[\bG_{\spl} = \begin{pmatrix}
    \bG_X & & \\
     & \ddots & \\
      & & \bG_X
    \end{pmatrix} . \]
    To compute $\bK_{\spl}$, note that the class $\omega'$ is ample.
    In particular, the restriction to $q=t=0$ of the quantum product associated to $\Phi^{\omega'}$ is the classical cup-product.
    Then, $\bK_{\spl}$ is block diagonal, and its $i$-th block is given by 
    \begin{equation}\label{eq:Kspl}
        (\bK_{\spl} )_{ii} = \bigg(c_1T_X +  \sum_{j\colon\deg T_j\neq 2} \frac{\deg T_j - 2}{2} a_{i,j} T_j \bigg)\cup .
    \end{equation}
    Thus, after identifying the fibers with $H_{\spl}$, the connections $\nabla\vert_b$ and $\nabla'\vert_{b'}$ lie in $\cF (H_{\spl} , \id , \bG_X ,\allowbreak (T_j\cup )_{0\leq j\leq N} )$, see \cref{def:F-bundle-point-restricted-connection}.
    We apply \cref{theorem:gauge-eq-F-bundle-point} with
    $\bK = -\bK_{\spl}$, $\bD = \bG_{\spl}$, $\bH = 0$, $\bK' = -\bK_{\lim}$ and 
    \[\bH' = \bG_{\lim} - \bG_{\spl} = \begin{pmatrix}
        - \frac{m-1}{2}  & & & \\
         & - \frac{m-3}{2}  & & \\
          & & \ddots & \\
          & & & \frac{m-1}{2}
    \end{pmatrix}.\]

    Assume first that the coordinates of $\Delta (a)$ satisfy \eqref{eq:coordinate-base point}.  
    Let $\phi = (\phi_{ij} ) \in \GL (H_{\spl} )$ denote the automorphism from \cref{lemma:jordan-dec-Klim}.
    Equations \eqref{eq:coordinate-base point} and \eqref{eq:Kspl} imply that $\phi^{-1} \bK_{\lim} \phi = \bK_{\spl}$, which is Condition (1) of the theorem.
    Condition (2) is satisfied with $\mu = \mu'  =1$.
    Condition (3) follows from \cref{lemma:vanishing-diagonal} and our choice of $\bH$.
    We conclude that the connections $\nabla\vert_{b}$ and $\nabla'\vert_{b'}$ are gauge equivalent through a bundle isomorphism $\Phi (u)$ satisfying the conditions of the theorem.

    Now, assume that there exists a bundle isomorphism $\Phi (u)$ as in the theorem, in particular each component $\phi_{ij}$ of $\Phi\vert_{u=0}$ is given by the cup-product with a cohomology class.
    Let $\phi\coloneqq (\phi_{ij} )\in \GL (H_{\spl} )$.
    Since $\Phi (u)$ is a gauge equivalence, we have in particular $\phi^{-1} \bK_{\lim} \phi = \bK_{\spl}$.
    Recall from \eqref{eq:Kspl} that $\bK_{\spl}$ is block diagonal, and that its coefficients are given by the cup-product with cohomology classes in $H^{\ast} (X,\bbC )$.
    The assumption on the components of $\Phi\vert_{u=0}$ implies that $\phi $ diagonalizes $\bK_{\lim}$ viewed as an element of $\Mat (m\times m ,R)$, where $R = \lbrace \alpha\mapsto x\cup \alpha \;\vert\; x\in H^{\ast} (X , \bbC) \rbrace$.
    The eigenvalues of $\bK_{\lim}$ as an $R$-linear map were computed in \cref{lemma:jordan-dec-Klim}, they are $(c_1T_X + c_1V + m\lambda_i ) \cup$ with $1\leq i\leq m$.
    In particular, $\Delta (a)$ satisfies \eqref{eq:coordinate-base point}.

    The uniqueness part of the theorem follows from the uniqueness of \cref{theorem:gauge-eq-F-bundle-point}, and the non-degeneracy of the Poincaré pairing.
\end{proof}

\begin{remark}
    If the $H^2$-component of the base point $\Delta (a)$ is $0$, then the quantum product converges at $\Delta (a)$ by \cref{lemma:shifted-quantum-product}.
\end{remark}

\begin{theorem}\label{thm:uniqueness-projective-bundle}
Let $(f, \Phi)\colon (\cH,\nabla)/B \to (\cH',\nabla')/B'$ be an isomorphism of F-bundles. Then 
\begin{enumerate}[wide]
    \item The bundle map $\Phi$ is uniquely and explicitly determined by its restriction to $b\in B$.
    \item The base map $f$ is uniquely and explicitly determined by its restriction to $b\in B$, up to a multiplicative constant in the $q$ direction.
\end{enumerate}
\end{theorem}

\begin{proof}
    The F-bundle $(\cH ',\nabla ')/B'$ is framed by definition.
    Since $\omega'$ is ample, at the point $b'$ the quantum product reduces to the classical cup-product. 
    In particular $(\1,\dots ,\1 )\in H^{\ast} (X' ,\bbC )$ is a cyclic vector.   
    The theorem thus follows from a direct application of \cref{lemma:comparison-framed-F-bundles}.
\end{proof}

We refer to \cite{Iritani_Quantum-cohomology-projective-bundle} regarding the existence of the isomorphism.

\subsection{Case of blowups of algebraic varieties} \label{subsec:dec-blowup}

In this subsection, we state the analogs of the results in \cref{sec:uniqueness_projective_bundle} in the case of blowups of algebraic varieties.

Let $X$ be a smooth project complex algebraic variety, and $\sigma\colon Z\hookrightarrow X$ a smooth closed subvariety of codimension $m\geq 2$.
Let $\pi\colon \tX \rightarrow X$ be the blowup of $X$ along $Z$.
Similar to the projective bundle case, we have a classical decomposition 
\begin{equation}\label{eq:splitting-cohomology-blowup}
    \iso\colon H^{\ast}(X ,\bbQ ) \oplus \bigoplus_{i=1}^{m-1} H^{\ast} (Z,\bbQ ) [-2i] \xrightarrow{\sim} H^{\ast} (\tX ,\bbQ ) .
\end{equation}
Let $X' \coloneqq X\sqcup\coprod_{i=1}^{m-1} Z$.
Fix an ample class $\omega_X\in H^2 (X ,\bbQ )$.

Let $(\cH,\nabla ) /B$ denote the maximal A-model F-bundle of $X$ associated to the nef class $\pi^{\ast}\omega_X$, with base point $b=0\in H^{\ast} (\tX ,\bbQ )$ and coordinates $(q,t)$.
Fix a class $\Delta (a)\in H^{\ast} (X',\bbC )\simeq H^{\ast} (X,\bbC )\oplus\bigoplus_{1\leq i\leq m-1} H^{\ast} (Z,\bbC )$ at which the quantum product is well-defined.
Let $(\cH' ,\nabla' )/B'$ denote the maximal A-model F-bundle associated to the class $(\omega_X ,\sigma^{\ast}\omega_X,\cdots ,\sigma^{\ast}\omega_X)$, with base point $b'=\Delta (a)$ and coordinates $(q,t)$ such that $q=t=0$ at $b'$.
Since $X'$ is a disjoint union, $(\cH' ,\nabla' )$ is the product of a maximal A-model F-bundle associated to $X$ and $\omega_X$, and $m-1$ copies of maximal F-bundles associated to $Z$ and $\sigma^{\ast}\omega_X$.

We can prove a result analogous to \cref{thm:uniqueness-projective-bundle-u=0}.
For $1\leq i\leq m$, let $\bc_i$ denote the cup-product with $c_i (N_{Z/X} )$.
The polynomial $\lambda^m + \sum_{i=0} \bc_{m-i}\lambda^i + \lambda$ has $m$ distinct roots $\blambda_i = \lambda_i\cup \in\bbC [\bc_1,\dots, \bc_m]$, with
\[ \lambda_1 = 0, \quad \lambda_i = \xi^{2(i-1)-1} - \frac{c_1N_{Z/X}}{m-1} \mod H^{\geq 3}(X,\bbC ),\]
where $\xi = e^{\frac{\pi i}{m-1}}$ and $2\leq i\leq m$, up to a permutation of the indices $\lbrace 2,\dots, m\rbrace$.
Those are the analogs of the eigenvalues computed for $\bK_{\lim}$ in the projective bundle case.

Let $\lbrace S_j\rbrace_{1\leq j\leq \dim H^{\ast} (Z,\bbC )}$ be a basis of $H^{\ast} (Z,\bbC )$ extending $\sigma^{\ast}\omega_X$.
For $1\leq i\leq m$, let $\Delta_i (a)$ denote the component of $\Delta (a)$ in the $i$-th summand of \eqref{eq:splitting-cohomology-blowup}, and for $2\leq i\leq m$ decompose it as 
\[\Delta_i (a) = \sum_{j} a_{i,j} S_j .\]

Using the splitting \eqref{eq:splitting-cohomology-blowup}, we can view an element $\Phi\in\End_{\bbC} (H^{\ast} (\tX ,\bbC ))$ as a matrix $(\Phi_{i,j})_{1\leq i,j\leq m}$, with $\Phi_{1,1}\in\End_{\bbC} (H^{\ast} (X,\bbC ))$, and $\Phi_{i,i}\in\End_{\bbC} (H^{\ast} (Z,\bbC ))$ for $2\leq i\leq m$.
The following result is analogous to \cref{thm:uniqueness-projective-bundle-u=0}.

\begin{theorem} \label{thm:uniqueness-blowup-u=0}
    Let $\Delta (a) \in H^{\ast} (X',\bbC )$ be a cohomology class at which the quantum product converges, such that $\Delta_1 (a) \in H^2 (X,\bbC )$, and for $2\leq i\leq m$, we have
        \begin{equation}
            \sum_{j\colon \deg S_j\neq 2} \frac{\deg_Z S_j - 2}{2} a_{i,j} S_j = c_1 N_{Z/X} + (m-1)\lambda_i.
        \end{equation}
    Then, there exists an F-bundle isomorphism $\Phi \colon (\cH ,\nabla)\vert_b\rightarrow (\cH ',\nabla' )\vert_{b'}$.

    Furthermore, if we restrict the coefficients of $\Phi$ to lie in a universal algebra as in the projective bundle case, then $\Phi$ is uniquely determined by its restriction to $u=0$, and the base point $\Delta (a)$ is uniquely determined up to a shift in $H^2 (X,\bbC )\oplus \bigoplus_{i=1}^{m-1} H^2 (Z,\bbC )$.
\end{theorem}

A direct consequence of \cref{lemma:comparison-framed-F-bundles} is the following, which is analogous to \cref{thm:uniqueness-projective-bundle}.

\begin{theorem}\label{thm:uniqueness-blowup}
Let $(f, \Phi)\colon (\cH,\nabla)/B \to (\cH',\nabla')/B'$ be an isomorphism of F-bundles.
Then 
\begin{enumerate}[wide]
    \item The bundle map $\Phi$ is uniquely and explicitly determined by its restriction to $b\in B$.
    \item The base map $f$ is uniquely and explicitly determined by its restriction to $b\in B$, up to a multiplicative constant in the $q$ direction.
\end{enumerate}
\end{theorem}

We refer to \cite{Iritani_Quantum-cohomology-projective-bundle} regarding the existence of the isomorphism.

\bibliographystyle{plain}
\bibliography{dahema}

\end{document}